\documentclass[a4paper,7pt,fleqn]{article}
\usepackage[latin1]{inputenc}
\usepackage[T1]{fontenc}
\usepackage[left=0.3cm,right=0.3cm,bottom=2cm,top=2cm,]{geometry}

\usepackage[all,2cell,v2,arrow, curve, frame, matrix, graph, tips]{xy}

\UseAllTwocells

\usepackage[square,numbers,sort]{natbib}
\usepackage{tikz-cd}
\usepackage{ifthen}
\usepackage{mathbbol}
\usepackage{savesym}
\usepackage{mathtools}
\usepackage{amsmath}
\usepackage{amssymb}
\usepackage{tcolorbox}
\usepackage{enumerate}
\usepackage[thmmarks,standard]{ntheorem}
\usepackage{microtype}
\usepackage{stmaryrd}

\usepackage{verbatim}
\usepackage[bookmarks=true,
            bookmarksnumbered=true, breaklinks=true,
            pdfstartview=FitH, hyperfigures=false,
            plainpages=false, naturalnames=true,
            colorlinks=true,pagebackref=true,
            pdfpagelabels]{hyperref}
\usepackage{cleveref}
\usepackage[toc,page]{appendix}
\usepackage{xcolor}
\hypersetup{
    colorlinks=true,
}

\theoremstyle{plain}

\newtheorem*{conj}{Conjecture}

\newcommand{\E}{\ensuremath{\mathbb{S}\text{ets}}}

\newcommand{\CS}{\ensuremath{\mathbb{C}\mathbb{S}\text{ets}}}

\newcommand{\I}{\ensuremath{\infty\text{-}}}
\newcommand{\CC}{\ensuremath{\infty\text{-}\mathbb{C}\mathbb{C}\text{AT}}}

\newcommand{\Mag}{\ensuremath{\mathbb{M}\text{ag}}}
\newcommand{\Cu}{\ensuremath{\mathbb{C}}}

\newcommand*{\Cat}{\ensuremath{\mathbb{C}\text{at}}}

\newcommand*{\CAT}{\ensuremath{\mathbb{C}\text{AT}}}

\title{Combinatorial approach to the category $\Theta_0$ of cubical pasting diagrams} 
\author{Camell Kachour}
\begin{document}
\maketitle
\vspace*{3.5cm}
\begin{abstract}
 In these notes we describe models of globular weak $(\infty,m)$-categories ($m\in\mathbb{N}$) in the Grothendieck style, i.e
 for each $m\in\mathbb{N}$ we define a globular coherator $\Theta^{\infty}_{\mathbb{M}^m}$ whose set-models are globular weak 
 $(\infty,m)$-categories. Then we describe the combinatorics of the small category $\Theta_0$ whose objects are cubical pasting diagrams and
 whose morphisms are morphisms of cubical sets. This provides an accurate description of the monad, on the category of
 cubical sets (without degeneracies and connections), of cubical strict $\infty$-categories with connections. 
 We prove that it is a cartesian monad, solving a conjecture in \cite{camark-cub-1}. 
 This puts us in a position to describe the cubical coherator $\Theta^{\infty}_W$ whose set-models are cubical weak $\infty$-categories with connections and the
 cubical coherator $\Theta^{\infty}_{W^{0}}$ whose set-models are cubical weak $\infty$-groupoids with connections.  
 \end{abstract}

{
  \hypersetup{linkcolor=magenta}
  \tableofcontents
}

%
\vspace*{1cm}

\section*{Introduction}
Coherators were initiated by Alexander Grothendieck \cite{grothendieck-pursuing} to properly define globular weak $\infty$-groupoids. 
A coherator $\Theta^{\infty}_{\mathbb{M}^0}$ for globular weak $\infty$-groupoids is a theory in the sense of \cite{bourke-garner-monads-theories} such that $\mathbb{M}\text{od}(\Theta^{\infty}_{\mathbb{M}^0})$
is the category of globular weak $\infty$-groupoids. These theories generalization those of Lawvere and are powerful yet simple enough to capture many higher structures. For example, a 
slight modification of the definition of the globular coherator $\Theta^{\infty}_{\mathbb{M}^0}$ (see \cite{Maltsin-Gr}) leads to the definition of an other 
globular coherator $\Theta^{\infty}_{\mathbb{M}}$ whose set-models are globular weak $\infty$-categories;  
such models are thus called \textit{Grothendieck's globular weak $\infty$-categories}. In \cite{ara-these} it is conjectured that these models 
are equivalent to Batanin's globular weak $\infty$-categories \cite{batanin-main}, and this conjecture has been proved in \cite{bourke-injectif}. In order to have a feel for the simplicity of this Grothendieck's approach, we first use it to describe globular models of weak $(\infty,m)$-categories ($m\in\mathbb{N}$). 
Thus for each $m\in\mathbb{N}$ we build a globular coherator $\Theta^{\infty}_{\mathbb{M}^m}$ whose set-models 
are globular weak $(\infty,m)$-categories. The author believes these models (for all $m\in\mathbb{N}$) are 
the simplest in the literature so far (see for example \cite{bergner,Cam}). 

The non-trivial part of this article is to describe cubical pasting diagrams. 
For that we use coordinates of networks which is a 
formalism close to that of tensors for differential geometry. 
In our language cubical pasting diagrams are called
\textit{rectangular divisors} which are formal finite sums of cubes indexed with coordinates of rectangular shapes. For each  
rectangular divisor we associate a specific inductive sketch. 
In fact we shall see that such rectangular
divisors form a cubical strict monoidal $\infty$-category thus lead to a cubical strict monoidal $\infty$-category for their underlying
inductive sketches. These sketches are the objects of the cubical $\Theta_0$. This combinatoric description of 
cubical pasting diagrams lead us to the monad, on the category of
 cubical sets (without degeneracies and connections), of cubical strict $\infty$-categories with connections. Having then an accurate
 description, we prove that these monads are cartesian, as conjectured in \cite{camark-cub-1}. See also
 \cite{camark-cub-2} where, based on this conjecture, we constructed a fundamental cubical weak 
 $\infty$-groupoid functor. Also we include here an accurate
 construction of the cubical coherator $\Theta^{\infty}_W$ whose set-models are cubical weak $\infty$-categories with connections and the
 cubical coherator $\Theta^{\infty}_{W^{0}}$ whose set-models are cubical weak $\infty$-groupoids with connections. Cubical coherators 
 have also been introduced recently and independently in \cite{thibault-benjamin-these}, and it will be interesting to compare that 
 approach with ours. 
 
 Here we summarize main achievements of this article :
 \begin{itemize}
 \color{blue}
 \item In \ref{coherator-M-m} we build for each $m\in\mathbb{N}$, a globular coherator $\Theta^{\infty}_{\mathbb{M}^m}$ which set-models are models
 of Grothendieck's globular weak $(\infty,m)$-categories.
 
 \item In \ref{monad-R} we prove that the monad $\mathbb{R}=(R,i,m)$ of cubical reflexive sets is cartesian.

 \item In \ref{rectangle} we prove that the set $\mathbb{C}\text{-}\mathbb{P}\text{ast}$ of cubical pasting diagrams (called rectangular divisors here) is equipped with a structure of 
 cubical strict $\infty$-category with connections. 
 
 \item In \ref{rectangle-sketch} we prove that the set of sketches associated to cubical pasting diagrams (called rectangular divisors here) is equipped with a structure of cubical strict $\infty$-category with connections.
 
 \item In \ref{monad-S} we prove that the monad $\mathbb{S}=(S,\lambda,\mu)$ acting on $\CS$ which algebras are cubical strict $\infty$-categories with connections (described in \cite{cam-cubique,camark-cub-1}) is cartesian. A simple consequence appears in \ref{monad-S-prime} where we indicate that the other monad $\mathbb{S}=(S,\lambda,\mu)$ acting on $\CS$ which algebras are cubical strict $\infty$-categories (without connections) is cartesian. 
 
 \item In \ref{coherator-cubique-M} we build the cubical coherator $\Theta^{\infty}_W$ which models are cubical weak $\infty$-categories with connections,
 and in \ref{coherator-cubique-M0} we build the cubical coherator $\Theta^{\infty}_{W^{0}}$ which models are cubical weak $\infty$-groupoids with connections.

 \end{itemize}

 {\bf Acknowledgement.} 
I thank mathematicians of the team AGA (Arithm\'etique et G\'eom\'etrie Alg\'ebrique) who kindly organized my	
talk on homotopy types (27th November 2019), and creating the good ambience in the LMO, 
Paris-Saclay; especially I want to mention Olivier Schiffmann, Benjamin Hennion, Fran\c{c}ois Charles,
Valentin Hernandez, and Patrick Massot. 
I also thank Ross Street, Michael Batanin, Mark Weber, Ronald Brown, 
Richard Steiner, with whom I interacted during the preparation of this article. 
Finally I thank St\'ef Bonnot-Briey, Pascale Marchal, Ghislain R\`emy
and Jean-Pierre Ledru, for their trust and help.
This article has been written in November 2019, and circulated to these mathematicians who provided feedback. 

I dedicate this work to my sons, Mohamed-R\'eda and Ali-R\'eda.

\section{Coherators for globular weak $(\infty,m)$-categories ($m\in\mathbb{N}$)}

\subsection{Globular magmatic structures}


Consider the small category $\mathbb{G}$ with objects $1(n)$ for all $n\in\mathbb{N}$, with
morphisms those generated for all $n\in\mathbb{N}$
by the \textit{cosources} 
\begin{tikzcd}
1(n-1)\arrow[r,"s^{n}_{n-1}"]&1(n)
\end{tikzcd}
 and the \textit{cotargets}  
\begin{tikzcd}
1(n-1)\arrow[r,"t^{n}_{n-1}"]&1(n)
\end{tikzcd},
which satisfy the following coglobular relations :

\begin{enumerate}[(i)]
\item $ s^{n}_{n-1}\circ s^{n+1}_{n}=t^{n}_{n-1}\circ s^{n+1}_{n},$
\item $s^{n}_{n-1}\circ t^{n+1}_{n}=t^{n}_{n-1}\circ t^{n+1}_{n},$
\end{enumerate}

The small category $\mathbb{G}$ is called the \textit{globe category} and we 
may represent it schematically with its generators :

$$\begin{tikzcd}
 1(0)\arrow[rr, yshift=1.5ex,"s^{1}_{0}"]
 \arrow[rr, yshift=-1.5ex,"t^{1}_{0}"{below}]
 &&1(1) 
\arrow[rr, yshift=1.5ex,"s^{2}_{1}"]
 \arrow[rr, yshift=-1.5ex,"t^{2}_{1}"{below}]
 && 1(2) 
 \arrow[rr, yshift=1.5ex,"s^{3}_{2}"]
 \arrow[rr, yshift=-1.5ex,"t^{3}_{2}"{below}]  
   && 
         1(3)
 \arrow[rr, yshift=1.5ex,"s^{4}_{3}"]
 \arrow[rr, yshift=-1.5ex,"t^{4}_{3}"{below}]     
 && 
        1(4)\cdots 1(n-1)  
\arrow[rr, yshift=1.5ex,"s^{n}_{n-1}"]
 \arrow[rr, yshift=-1.5ex,"t^{n}_{n-1}"{below}]
&& 1(n)\cdots    
  \end{tikzcd}$$
  
\begin{definition}
Globular sets are presheaves on $\mathbb{G}^{\text{op}}$. The category of 
globular sets is denoted $\mathbb{G}\text{lob}$.
\end{definition}

A globular $\infty$-magma $M$ is given by a globular set 
\begin{tikzcd}
\mathbb{G}^{\text{op}}\arrow[r,"M"]&\E
\end{tikzcd}
equipped with operations
\begin{tikzcd}
M_n\times_{M_p} M_n\arrow[r,"\circ^{n}_p"]&M_n
\end{tikzcd} 
for all $n\geq 1$ and all $0\leq p\leq n-1$ such that :
\begin{itemize}
\item for $0\leq p<q<m$, 
$s^m_q(y\circ^{m}_{p}x)=s^m_q(y)\circ^{q}_{p}s^m_q(x)$ and 
$t^m_q(y\circ^{m}_{p}x)=t^m_q(y)\circ^{q}_{p}t^m_q(x)$

\item for $0\leq q<p<m$, $s^m_q(y\circ^{m}_{p}x)=s^m_q(y)=s^m_q(x)$
and $t^m_q(y\circ^{m}_{p}x)=t^m_q(y)=t^m_q(x)$

\item for $0\leq p=q<m$, $s^m_q(y\circ^{m}_{p}x)=s^m_q(x)$ and 
$t^m_q(y\circ^{m}_{p}x)=t^m_q(x)$

\end{itemize}

A globular reflexive $\infty$-magma is an $\infty$-magma equipped with map for
reflexivity :
\begin{tikzcd}
M_n\arrow[rr,"1^{n}_{n+1}"]&&M_{n+1}
\end{tikzcd},
$n\geq 0$ such that :
\begin{itemize}
\item $s^n_k(1^{k}_{n}(x))=x=t^n_k(1^{k}_{n}(x))$

\item $1^{q}_{n}(1^{p}_{q}(x))=1^{p}_{n}(x)$

\end{itemize}

Morphisms between reflexive $\infty$-magmas are morphisms of reflexive globular sets
between their underlying reflexive globular set structure, i.e for 
\begin{tikzcd}
M\arrow[rr,"f"]&&M'
\end{tikzcd}
we have commutative
diagrams :

\begin{tikzcd}
M_{n+1}\arrow[rr,"f_{n+1}"]&&M'_{n+1}\\
M_n\arrow[u,"1^n_{n+1}"]\arrow[rr,"f_{n}"{below}]&&M'_n\arrow[u,"1^n_{n+1}"{right}]
\end{tikzcd}

which also preserve operations $\circ^n_p$. The category of reflexive $\infty$-magmas
is denoted $\infty\text{-}\mathbb{M}\text{ag}_{\text{r}}$.


An $(\infty,m)$-globular set is a globular set $X$ equipped with $j^{n}_{n-1}$-reversors, i.e
with maps 
\begin{tikzcd}
X_n\arrow[rr,"j^{n}_{n-1}"]&&X_n
\end{tikzcd}
which satisfy the following equalities :

$$\begin{tikzcd}
X_n\arrow[rd,"s^n_{n-1}"{left}]\arrow[rr,"j^{n}_{n-1}"]&&X_n\arrow[ld,"t^{n}_{n-1}"]\\
&X_{n-1}
\end{tikzcd}\qquad
\begin{tikzcd}
X_n\arrow[rd,"t^n_{n-1}"{left}]\arrow[rr,"j^{n}_{n-1}"]&&X_n\arrow[ld,"s^{n}_{n-1}"]\\
&X_{n-1}
\end{tikzcd}$$
A morphism of $(\infty,m)$-globular sets is a morphism
\begin{tikzcd}
X\arrow[rr,"f"]&&X'
\end{tikzcd}
of globular sets which satisfy for all $n\geq m$ the following
equalities :
$$\begin{tikzcd}
X_n\arrow[d,"j^{n}_{n-1}"{left}]\arrow[rr,"f_n"]&&X'_n\arrow[d,"j^{n}_{n-1}"]\\
X_n\arrow[rr,"f_n"{below}]&&X'_n
\end{tikzcd}$$
The category of $(\infty,m)$-globular sets is denoted $(\infty,m)\text{-}\mathbb{G}\text{lob}$.

A globular reflexive $(\infty,m)$-magma is a globular reflexive 
$\infty$-magma $M$ equipped with a structure of globular
$(\infty,m)$-set; a morphism 
\begin{tikzcd}
M\arrow[rr,"f"]&&M'
\end{tikzcd}
of globular reflexive $(\infty,m)$-magmas is a morphism
of globular reflexive $\infty$-magmas which is also a 
morphism of $(\infty,m)$-sets; the category of globular 
reflexive $(\infty,m)$-magmas is denoted 
$(\infty,m)\text{-}\mathbb{M}\text{ag}_{\text{r}}$. 

\begin{remark}
A globular strict $\infty$-category $C$ is given by a globular reflexive 
$\infty$-magma
$C$ such that we have the following equalities :

\begin{itemize}
\item $x\circ^n_k1^k_n(s^n_k(x))=x$ and $1^k_n(t^n_k(x))\circ^n_kx=x$

\item $1^q_n(y\circ^q_px)=1^q_p(y)\circ^{n}_{p}1^q_p(x)$

\item $x\circ^n_k(y\circ^n_kz)=(x\circ^n_ky)\circ^n_kz$

\item $(y'\circ^n_qx')\circ^n_p(y\circ^n_qx)=(y'\circ^n_py)\circ^n_q(x'\circ^n_px)$

\end{itemize}

The category of globular strict $\infty$-categories is denoted $\infty\text{-}\mathbb{C}\text{AT}$.
A globular strict $(\infty,m)$-category is given by an $(\infty,m)$-globular set $C$ which is
also a globular strict $\infty$-category such that if $\alpha\in C_n$ ($n\geq m$) then 
$\alpha\circ^{n}_{n-1}j^{n}_{n-1}(\alpha)=1^{n-1}_n(t^n_{n-1}(\alpha))$ and
$j^{n}_{n-1}(\alpha)\circ^{n}_{n-1}\alpha=1^{n-1}_n(s^n_{n-1}(\alpha))$. This 
$n$-cell $j^{n}_{n-1}(\alpha)$ of $C_n$ is called a $\circ^n_{n-1}$-inverse of 
$\alpha$ and it is straightforward to see that such $\circ^n_{n-1}$-inverse is
uniquely defined. The category of globular strict $(\infty,m)$-categories is defined
as the full subcategory of $\infty\text{-}\mathbb{C}\text{AT}$ which objects
are globular strict $(\infty,m)$-categories and is denoted 
$(\infty,m)\text{-}\mathbb{C}\text{AT}$.
\end{remark}

\subsection{Globular Theories}

\subsubsection{Globular extensions}

A globular tree $t$ is given by a table of non-negative integers :

$$\begin{pmatrix} 
i_1&&&&i_2&&&&i_3&\cdot&\cdot&\cdot&i_{k-1}&&&&i_{k}\\\\
&&i'_1&&&&i'_2&&\cdot&\cdot&\cdot&\cdot&\cdot&&i'_{k-1}
\end{pmatrix}$$

where $k\geq 1$, $i_l>i'_l<i_{l+1}$ and $1\leq l\leq k-1$.

Let $\mathcal{C}$ a category and let 
\begin{tikzcd}
\mathbb{G}\arrow[rr,"F"]&&\mathcal{C}
\end{tikzcd}
a functor. We denote $F(1(n))=D^n$ and we shall keep the same 
notations for the image of cosources : $F(s^{i_l}_{i_{l'}})=s^{i_l}_{i_{l'}}$, 
and for the image of cotargets : $F(t^{i_l}_{i_{l'}})=t^{i_l}_{i_{l'}}$, because
no risk of confusion will occur. In this case 
\begin{tikzcd}
\mathbb{G}\arrow[rr,"F"]&&\mathcal{C}
\end{tikzcd}
is called a globular extension if for all trees $t$ as just above, the
colimit of the following diagram exist in $\mathcal{C}$ :
$$\begin{tikzcd}
D^{i_1}&&D^{i_2}&&D^{i_3}&&\cdots&&D^{i_{k-1}}&&D^{i_k}\\
&D^{i'_1}\arrow[lu,"t^{i_1}_{i'_1}"]\arrow[ru,"s^{i_2}_{i'_1}"]
&&D^{i'_2}\arrow[lu,"t^{i_2}_{i'_2}"]\arrow[ru,"s^{i_3}_{i'_2}"]&&
D^{i'_3}\arrow[lu]&\cdots&
D^{i'_{k-2}}\arrow[ru]&&
D^{i'_{k-1}}\arrow[lu,"t^{i_{k-1}}_{i'_{k-1}}"]\arrow[ru,"s^{i_k}_{i'_{k-1}}"]
\end{tikzcd}$$

\begin{remark}
In \cite{grothendieck-pursuing} Alexander Grothendieck calls these colimits \textit{globular sums}.
\end{remark}

A morphism of globular extensions, also called \textit{globular functor}, is 
given by a commutative triangle in $\CAT$ :

$$\begin{tikzcd}
&&\mathcal{C}\arrow[dd,"H"]\\
\mathbb{G}\arrow[rru,"F"]\arrow[rrd,"F'"{below}]\\
&&\mathcal{C}'
\end{tikzcd}$$

such that the functor $H$ preserves
globular sums. The category of globular extensions
is denoted $\mathbb{G}\text{-}\mathbb{E}\text{xt}$. In
fact this category has an initial object denoted 
\begin{tikzcd}
\mathbb{G}\arrow[rr,"i"]&&\Theta_0
\end{tikzcd}.
And the small category $\Theta_0$ can be described as the full subcategory
of $\mathbb{G}\text{lob}$ which objects are globular trees, and its role is
central for describing different sketches which set models are globular higher
structures. In particular this small category $\Theta_0$ is the basic inductive sketch 
we shall need to describe coherators which set models are globular weak 
$(\infty,m)$-categories ($m\in\mathbb{N}$).

\subsubsection{Globular theories}

A globular theory is given by a globular extension
\begin{tikzcd}
\mathbb{G}\arrow[rr,"F"]&&\mathcal{C}
\end{tikzcd}
such that the unique induced functor $\overline{F}$ which makes
commutative the diagram :

$$\begin{tikzcd}
&&\Theta_0\arrow[dd,"\overline{F}"]\\
\mathbb{G}\arrow[rru,"i"]\arrow[rrd,"F"{below}]\\
&&\mathcal{C}
\end{tikzcd}$$

induces a bijection between objects of $\Theta_0$ and
objects of $\mathcal{C}$. The full subcategory of 
$\mathbb{G}\text{-}\mathbb{E}\text{xt}$ which objects
are globular theories is denoted
$\mathbb{G}\text{-}\mathbb{T}\text{h}$.
Consider an object
\begin{tikzcd}
\mathbb{G}\arrow[rr,"F"]&&\mathcal{C}
\end{tikzcd},
in particular it induces the globular functor
\begin{tikzcd}
\Theta_0\arrow[rr,"\overline{F}"]&&\mathcal{C}
\end{tikzcd} 
as just above, which is a bijection on objects. A set model 
of $(F,\mathcal{C})$ or for $\mathcal{C}$ for short, is given 
by a functor :
\begin{tikzcd}
\mathcal{C}\arrow[rr,"X"]&&\E
\end{tikzcd},
such that the functor $X\circ\overline{F}$ :

$$\begin{tikzcd}
\Theta_0\arrow[rr,"\overline{F}"]&&\mathcal{C}\arrow[rr,"X"]&&\E
\end{tikzcd}$$
sends globular sums to globular products\footnote{Globular products
are just dual to globular sums.}, thus for all objects $t$ of $\Theta_0$ :

$$\begin{pmatrix} 
i_1&&&&i_2&&&&i_3&\cdot&\cdot&\cdot&i_{k-1}&&&&i_{k}\\\\
 &&i'_1&&&&i'_2&&\cdot&\cdot&\cdot&\cdot&\cdot&&i'_{k-1}
\end{pmatrix}$$

we have 

\begin{eqnarray*}
X(\overline{F}(t)) & = &{X\left(\text{colim}\left(\begin{tikzcd}
D^{i_1}&&D^{i_2}&&\cdots&&D^{i_{k-1}}&&D^{i_k}\\
&D^{i'_1}\arrow[lu,"t^{i_1}_{i'_1}"]\arrow[ru,"s^{i_2}_{i'_1}"]
&&&\cdots&&&
D^{i'_{k-1}}\arrow[lu,"t^{i_{k-1}}_{i'_{k-1}}"]\arrow[ru,"s^{i_k}_{i'_{k-1}}"]
\end{tikzcd}\right)\right)} \\
 & = & {X\left( (D^{i_1},t^{i_1}_{i'_1})\underset{D^{i'_1}}\coprod (s^{i_2}_{i'_1},D^{i_2},t^{i_3}_{i'_2})\underset{D^{i'_2}}\coprod
\cdots\underset{D^{i'_{k-1}}}\coprod
(s^{i_k}_{i'_{k-1}},D^{i_k})\right)}\\
& \simeq &
{X(D^{i_1})\underset{X(D^{i'_1})}\times\cdots
\underset{X(D^{i'_{k-1}})}\times X(D^{i_k})}
\end{eqnarray*}

$$X(\overline{F}(t))=X\left(\text{colim}\left(\begin{tikzcd}
D^{i_1}&&D^{i_2}&&\cdots&&D^{i_{k-1}}&&D^{i_k}\\
&D^{i'_1}\arrow[lu,"t^{i_1}_{i'_1}"]\arrow[ru,"s^{i_2}_{i'_1}"]
&&&\cdots&&&
D^{i'_{k-1}}\arrow[lu,"t^{i_{k-1}}_{i'_{k-1}}"]\arrow[ru,"s^{i_k}_{i'_{k-1}}"]
\end{tikzcd}\right)\right)$$

$X\left( (D^{i_1},t^{i_1}_{i'_1})\underset{D^{i'_1}}\coprod (s^{i_2}_{i'_1},D^{i_2},t^{i_3}_{i'_2})\underset{D^{i'_2}}\coprod
\cdots\underset{D^{i'_{k-1}}}\coprod
(s^{i_k}_{i'_{k-1}},D^{i_k})\right)\simeq
X(D^{i_1})\underset{X(D^{i'_1})}\times\cdots
\underset{X(D^{i'_{k-1}})}\times X(D^{i_k})$

The category of set models of $\mathcal{C}$ is the full 
subcategory of the category of presheaves
$[\mathcal{C},\E]$ which objects are set models
of $\mathcal{C}$, and it is denoted $\mathbb{M}\text{od}(\mathcal{C})$.

\subsubsection{Examples of globular theories}

\begin{Example}{The theory $\Theta_{\mathbb{M}}$}

The forgetful functor $U$ :

$$\begin{tikzcd}
\infty\text{-}\mathbb{M}\text{ag}_{\text{r}}\arrow[dd,xshift=1ex,"\dashv"{left},"U"]\\\\
\mathbb{G}\text{lob}\arrow[uu,xshift=-1ex,"F",dotted] 
  \end{tikzcd}$$
from the category $\infty\text{-}\mathbb{M}\text{ag}_{\text{r}}$ of globular reflexive $\infty$-magmas to the category $\mathbb{G}\text{lob}$ of globular sets is right adjoint, which left adjoint is denoted $F$, and this induce a monad $\mathbb{M}=(M,\eta,\mu)$ on $\mathbb{G}\text{lob}$
such that we have the equivalence of categories 
$\infty\text{-}\mathbb{M}\text{ag}_{\text{r}}\simeq
\mathbb{M}\text{-}\mathbb{A}\text{lg}$ because $U$ is 
monadic. The full subcategory 
$\Theta_{\mathbb{M}}\subset\mathbb{K}\text{l}(\mathbb{M})$ of
the Kleisli category of $\mathbb{M}$ which objects are trees is
called the theory of reflexive globular $\infty$-magmas. In fact
we have the following equivalences of categories :
$$\infty\text{-}\mathbb{M}\text{ag}_{\text{r}}\simeq
\mathbb{M}\text{-}\mathbb{A}\text{lg}\simeq
\mathbb{M}\text{od}(\Theta_{\mathbb{M}})$$

\end{Example}

\begin{Example}{The theories $\Theta_{\mathbb{M}^m}$ 
($m\in\mathbb{N}$)}

The forgetful functor $U^m$ ($m\in\mathbb{N}$) :

$$\begin{tikzcd}
(\infty,m)\text{-}\mathbb{M}\text{ag}_{\text{r}}\arrow[dd,xshift=1ex,"\dashv"{left},"U^m"]\\\\
\mathbb{G}\text{lob}\arrow[uu,xshift=-1ex,"F^m",dotted] 
  \end{tikzcd}$$
from the category $(\infty,m)\text{-}\mathbb{M}\text{ag}_{\text{r}}$ of globular reflexive $(\infty,m)$-magmas to the category 
$\mathbb{G}\text{lob}$ of globular sets is right adjoint, which left adjoint is denoted $F^m$, and this induce a monad $\mathbb{M}^m=(M^m,\eta^m,\mu^m)$ on $\mathbb{G}\text{lob}$
such that we have the equivalence of categories 
$(\infty,m)\text{-}\mathbb{M}\text{ag}_{\text{r}}\simeq
\mathbb{M}^m\text{-}\mathbb{A}\text{lg}$ because $U^m$ is 
monadic. The full subcategory 
$\Theta_{\mathbb{M}^m}\subset\mathbb{K}\text{l}(\mathbb{M}^m)$ of
the Kleisli category of $\mathbb{M}^m$ which objects are trees is
called the theory of reflexive globular $(\infty,m)$-magmas. In fact
we have the following equivalences of categories :
$$(\infty,m)\text{-}\mathbb{M}\text{ag}_{\text{r}}\simeq
\mathbb{M}^m\text{-}\mathbb{A}\text{lg}\simeq
\mathbb{M}\text{od}(\Theta_{\mathbb{M}^m})$$

\end{Example}

\subsection{Globular coherators}

\subsubsection{Admissibility}

Let 
\begin{tikzcd}
\mathbb{G}\arrow[rr,"F"]&&\mathcal{C}
\end{tikzcd}
be a globular theory, i.e an object of 
$\mathbb{G}\text{-}\mathbb{T}\text{h}$;
two arrows :
\begin{tikzcd}
D^n\arrow[rr, yshift=1.2ex,"f"]
 \arrow[rr, yshift=-1.2ex,"g"{below}]&&t
\end{tikzcd}
in $\mathcal{C}$ are parallels if $fs^{n}_{n-1}=gs^{n}_{n-1}$
and $ft^{n}_{n-1}=gt^{n}_{n-1}$ :
$$\begin{tikzcd}
D^n\arrow[rr, yshift=1.2ex,"f"]
 \arrow[rr, yshift=-1.2ex,"g"{below}]&&t\\\\
 D^{n-1}\arrow[uu, xshift=-1.2ex,"s^{n}_{n-1}"{left}]
 \arrow[uu, xshift=1.2ex,"t^{n}_{n-1}"{right}] 
\end{tikzcd}$$

Consider a couple $(f,g)$ of parallels arrows in $\mathcal{C}$
as just above. We say that it is admissible or algebraic if they
don't belong to the image of the globular functor $\overline{F}$ :
$$\begin{tikzcd}
&&\Theta_0\arrow[dd,"\overline{F}"]\\
\mathbb{G}\arrow[rru,"i"]\arrow[rrd,"F"{below}]\\
&&\mathcal{C}
\end{tikzcd}$$

Consider a couple $(f,g)$ of arrows of $\mathcal{C}$ which is admissible
 as just above; a lifting of $(f,g)$ is given by an arrow $h$ :
$$\begin{tikzcd}
D^{n+1}\arrow[rrdd,"h"]\\\\
D^{n}\arrow[uu, xshift=-1.2ex,"s^{n+1}_{n}"{left}]
 \arrow[uu, xshift=1.2ex,"t^{n+1}_{n}"{right}] 
\arrow[rr, yshift=1.2ex,"f"]
 \arrow[rr, yshift=-1.2ex,"g"{below}]&&t
 \end{tikzcd}$$
 such that $hs^{n+1}_{n}=f$ and $ht^{n+1}_{n}=g$
 
 \subsubsection{Batanin-Grothendieck Sequences}
 
 We now define the Batanin-Grothendieck sequence\footnote{Coherators associated to such sequence
 are called \textit{of Batanin-Leinster type} by some authors.} associated to a globular theory 
 \begin{tikzcd}
\mathbb{G}\arrow[rr,"F"]&&\mathcal{C}
\end{tikzcd}. 
We build it by the following induction :

\begin{itemize}
\item If $n=0$ we start with the couple $(\mathcal{C},E)$ where 
$E$ denotes the set of admissible pairs of arrows of $\mathcal{C}$;
we shall write $(\mathcal{C}_0,E_0)=(\mathcal{C},E)$ this first
step.

\item If $n=1$ we consider then the couple $(\mathcal{C}_1,E_1)$
where $\mathcal{C}_1$ is obtained by formally adding in 
$\mathcal{C}_0=\mathcal{C}$ the liftings of all elements 
$(f,g)\in E_0=E$, and $E_1$ is the set of admissible couples of arrows 
in $\mathcal{C}_1$ which are not elements of the set $E_0$;

\item If for $n\geq 2$ the couple $(\mathcal{C}_n,E_n)$ is well defined
then $\mathcal{C}_{n+1}$ is obtained by formally adding in 
$\mathcal{C}_{n}$ the liftings of all elements of $E_n$, and 
$E_{n+1}$ is the set of couples of arrows of $\mathcal{C}_{n+1}$
which are not elements of $E_n$
\end{itemize}
we give a slightly different but equivalent induction to build the Batanin-Grothendieck sequence for such globular theory
\begin{tikzcd}
\mathbb{G}\arrow[rr,"F"]&&\mathcal{C}
\end{tikzcd} :

\begin{itemize}
\item If $n=0$ we start with the couple $(\mathcal{C},E)$ where 
$E$ is the set of couple of arrow which are admissible of $\mathcal{C}$;
we denote $E=E_0=E'_0=E'_0\setminus\emptyset$ (we shall see soon the reason of these notations), and $\mathcal{C}_0=\mathcal{C}$;

\item If $n=1$ we consider the couple $(\mathcal{C}_1,E_1)$ where
$\mathcal{C}_1$ is obtained by formally adding in $\mathcal{C}_0$
all liftings of the elements $(f,g)\in E_0$, $E'_1$ is the set of all
pairs of arrows which are admissible in $\mathcal{C}_1$, and 
$E_1=E'_1\setminus E_0$; remark that $E_0=E'_0\subset E'_1$;

\item If $n=2$ we consider the couple $(\mathcal{C}_2,E_2)$ where
$\mathcal{C}_2$ is obtained by formally adding in $\mathcal{C}_1$
all liftings of the elements $(f,g)\in E_1$, $E'_2$ is the set of all
pairs of arrows which are admissible in $\mathcal{C}_2$, and
$E_2=E'_2\setminus E'_1$;

\item For $n\geq 3$ we suppose that the couple $(\mathcal{C}_n,E_n)$
is well defined with $E_n=E'_n\setminus E'_{n-1}$, then 
$\mathcal{C}_{n+1}$ is obtained by formally adding in $\mathcal{C}_n$
all liftings of the elements $(f,g)\in E_n$, $E'_{n+1}$ is the set of all
pairs of arrows which are admissible in $\mathcal{C}_{n+1}$, and
$E_{n+1}=E'_{n+1}\setminus E'_n$;
\end{itemize}

 The Batanin-Grothendieck sequence of the globular theory
\begin{tikzcd}
\mathbb{G}\arrow[rr,"F"]&&\mathcal{C}
\end{tikzcd} 
produces the following filtered diagram

\begin{tikzcd}
(\mathbb{N},\leq)\arrow[rr,"\mathcal{C}_{\bullet}"]&&
\mathbb{G}\text{-}\mathbb{T}\text{h}
\end{tikzcd}
in the category $\mathbb{G}\text{-}\mathbb{T}\text{h}$ :

$$\begin{tikzcd}
\mathcal{C}_0\arrow[rr,"i_1"]&&\mathcal{C}_1\arrow[rr,"i_2"]&&\cdots\arrow[rr,"i_n"]&&\mathcal{C}_n\arrow[rr]&&\cdots
\end{tikzcd}$$

\subsubsection{Coherators for globular theories}

We start with datas of the previous subsection, i.e 
with the Batanin-Grothendieck sequence \begin{tikzcd}
(\mathbb{N},\leq)\arrow[rr,"\mathcal{C}_{\bullet}"]&&
\mathbb{G}\text{-}\mathbb{T}\text{h}
\end{tikzcd}
for a globular theory 
\begin{tikzcd}
\mathbb{G}\arrow[rr,"F"]&&\mathcal{C}
\end{tikzcd}.

\begin{definition}
 The colimit 
\begin{tikzcd}
\mathbb{G}\arrow[rr,"F_{\infty}"]&&\mathcal{C}_{\infty}
\end{tikzcd}
of the previous filtered diagram $\mathcal{C}_{\bullet}$ :
$$\begin{tikzcd}
\mathcal{C}_0\arrow[rr,"i_1"]\arrow[rrrrrrrrdd]
&&\mathcal{C}_1\arrow[rr,"i_2"]\arrow[rrrrrrdd]&&\cdots\arrow[rr,"i_n"]&&\mathcal{C}_n\arrow[rr]\arrow[rrdd]&&\cdots\\\\
&&&&&&&&\mathcal{C}_{\infty}
\end{tikzcd}$$
is called the globular coherator of the type Batanin-Grothendieck associated to the globular theory
\begin{tikzcd}
\mathbb{G}\arrow[rr,"F"]&&\mathcal{C}
\end{tikzcd}.
\end{definition}
For shorter terminology we shall say that 
\begin{tikzcd}
\mathbb{G}\arrow[rr,"F_{\infty}"]&&\mathcal{C}_{\infty}
\end{tikzcd}
is the coherator associated to the globular theory
\begin{tikzcd}
\mathbb{G}\arrow[rr,"F"]&&\mathcal{C}
\end{tikzcd}.
It is straightforward to see that the Batanin-Grothendieck construction of coherators associated to globular theory is functorial, and the following functor $\Phi$ is called the 
Batanin-Grothendieck functor :
$$\begin{tikzcd}
\mathbb{G}\text{-}\mathbb{T}\text{h}
\arrow[rr,"\Phi"]&&\mathbb{G}\text{-}\mathbb{T}\text{h}\\
\mathcal{C}\arrow[rr,mapsto]&&\mathcal{C}^{\infty}
\end{tikzcd}$$

\subsubsection{The coherator $\Theta^{\infty}_{\mathbb{M}}$}
\label{coherator-M}

The coherator associated to the globular theory 
\begin{tikzcd}
\mathbb{G}\arrow[rr,"j"]&&\Theta_{\mathbb{M}}
\end{tikzcd}
that we obtaine with the composition :

$$\begin{tikzcd}
\mathbb{G}\arrow[rr,"i"]&&\Theta_{0}
\arrow[rr,hook]&&\Theta_{\mathbb{M}}
\end{tikzcd}$$

is denoted $\Theta^{\infty}_{\mathbb{M}}$ and $\mathbb{M}\text{od}(\Theta^{\infty}_{\mathbb{M}})$
is the category of globular weak $\infty$-categories of Grothendieck.

\begin{remark}
In 2019 John Bourke has proved \cite{bourke-injectif} the \textit{Ara conjecture} \cite{ara-these} which says that the category of 
globular weak $\infty$-categories of Batanin is equivalent to the category of globular weak 
$\infty$-categories of Grothendieck :
$$\mathbb{M}\text{od}(\Theta_{\mathbb{B}^{0}_C})
\simeq\mathbb{M}\text{od}(\Theta^{\infty}_{\mathbb{M}})$$
where here $\mathbb{B}^{0}_C$ denotes the globular operad of Batanin \cite{batanin-main} which 
algebras are his models of globular weak $\infty$-categories and $\Theta_{\mathbb{B}^{0}_C}$ is its associated theory.
\end{remark}

\subsubsection{The coherator $\Theta^{\infty}_{\mathbb{M}^m}$
($m\in\mathbb{N}$)}
\label{coherator-M-m}

The coherator associated to the globular theory 
\begin{tikzcd}
\mathbb{G}\arrow[rr,"j_m"]&&\Theta_{\mathbb{M}^m}
\end{tikzcd}
is denoted $\Theta^{\infty}_{\mathbb{M}^m}$ and $\mathbb{M}\text{od}(\Theta^{\infty}_{\mathbb{M}^m})$
is the category of globular weak $(\infty,m)$-categories of Grothendieck ($m\geq 0$).
If $m=0$, the coherator $\Theta^{\infty}_{\mathbb{M}^0}$ is the one of globular weak
$\infty$-groupoids of Grothendieck.
We trivially have the following filtration in the category $\mathbb{G}\text{-}\mathbb{T}\text{h}$ :

$$\begin{tikzcd}
\cdots\arrow[rr]&&\Theta^{\infty}_{\mathbb{M}^{m+1}}\arrow[rr]&&\Theta^{\infty}_{\mathbb{M}^m}\arrow[rr]&&
\cdots\arrow[rr]&&
\Theta^{\infty}_{\mathbb{M}^0}
\end{tikzcd}$$

$$\begin{tikzcd}
&&&&\Theta^{\infty}_{\mathbb{M}}
\arrow[lllldd]\arrow[lldd]\arrow[dd]
\arrow[rrrrdd]\\\\
\cdots\arrow[rr]&&\Theta^{\infty}_{\mathbb{M}^{m+1}}\arrow[rr]&&\Theta^{\infty}_{\mathbb{M}^m}\arrow[rr]&&
\cdots\arrow[rr]&&
\Theta^{\infty}_{\mathbb{M}^0}
\end{tikzcd}$$

which shows that we have the following inclusion of functors when
passing to set models :

$$\begin{tikzcd}
\mathbb{M}\text{od}(\Theta^{\infty}_{\mathbb{M}^0})
\arrow[rr,"i_1"]\arrow[rrrrrrrrdd]
&&\mathbb{M}\text{od}(\Theta^{\infty}_{\mathbb{M}^1})
\arrow[rr,"i_2"]\arrow[rrrrrrdd]&&\cdots\arrow[rr,"i_n"]&&
\mathbb{M}\text{od}(\Theta^{\infty}_{\mathbb{M}^m})
\arrow[rr]\arrow[rrdd]&&\cdots\\\\
&&&&&&&&
\mathbb{M}\text{od}(\Theta^{\infty}_{\mathbb{M}})
\end{tikzcd}$$

We finish this section by recalling the \textit{Grothendieck Conjecture for Homotopy Theory} :

\begin{conj}[Grothendieck's Conjecture for Homotopy Theory]
The category $\mathbb{M}\text{od}(\Theta^{\infty}_{\mathbb{M}^m})$ is Quillen equivalent to categories of simplicial models of
weak $(\infty,m)$-categories (for all $m\in\mathbb{N}$).
\end{conj}
See for example \cite{bergner} for such existing simplicial models.

\section{Cubical Pasting Diagrams}

\subsection{Tensorial notation}

In this section we introduce tensorial notation and shall see that contraction and dilatation 
of tensors provide interesting cubical strict $\infty$-categories, though trivial. In particular it reveals 
that tensorial calculus has an intrinsic cubical nature.

\begin{itemize}

\item 
For each $n\in\mathbb{N}$ we shall use a coordinate system $\mathcal{Z}_n$ of $n$-dimensional networks $i\in\{1,\cdots,n\}$ such that each
$i\in\{1,\cdots,n\}$ is the direction of a $n$-cube whose coordinates are indexed by this network. The coordinate
of a $n$-cube $C$ in $\mathcal{Z}_n$ is written  
$dx^1_{k_1}\otimes\cdots\otimes dx^j_{k_j}\otimes\cdots dx^n_{k_n}$
which means that $C$ is located for each direction $j\in\{1,\cdots,n\}$ at the \textit{depth} $k_j\in\mathbb{Z}$. When
no confusion occur we shall denote $dx^i_{k_i}:=dx^1_{k_1}\otimes\cdots\otimes dx^j_{k_j}\otimes\cdots dx^n_{k_n}$.

\begin{remark}
A coordinate $dx^i_{k_i}$ must be thought up to its translations in the network $\mathcal{Z}_n$.
Indeed it is straightforward to see that two coordinates $dx^i_{k_i}, dx^i_{k'_i}\in\mathcal{Z}_n$ are related by 
translations. For example any coordinates $dx^i_{k_i}\in\mathcal{Z}_n$ gives the coordinate
$dx^i_{1}:=dx^1_{1}\otimes\cdots\otimes dx^j_{1}\otimes\cdots dx^n_{1}$ by translations along
all directions $j\in\llbracket 1,n\rrbracket$. 
\end{remark}

Two coordinates $dx^{i}_{k_i}=dx^{1}_{k_1}\otimes\cdots\otimes dx^{n}_{k_n}$ and
$dx^{i}_{k'_i}=dx^{1}_{k'_1}\otimes\cdots\otimes dx^{n}_{k'_n}$ are $j$-adjacent if $k_j=k'_j+1$ or 
$k_j=k'_j-1$.

The $j$-contraction of the coordinate $dx^1_{k_1}\otimes\cdots\otimes dx^j_{k_j}\otimes\cdots dx^n_{k_n}$
is defined as the coordinate 
$$dx^i_{k_i}\setminus{j}=dx^1_{k_1}\otimes\cdots\otimes \widehat{dx^j_{k_j}}\otimes\cdots dx^n_{k_n}$$
in $\mathcal{Z}_{n-1}$ defined by removing the direction $j$ and re-indexing :
$$dx^i_{k_i}\setminus{j}:=dx^1_{k_1}\otimes\cdots\otimes dx^{j-1}_{k_{j-1}}
\otimes dx^j_{k_{j+1}}\otimes\cdots dx^{n-1}_{k_n}$$ 

Sometimes we use also the notation $a_j(dx^i_{k_i})$ for $dx^i_{k_i}\setminus{j}$.

If we apply these contractions $p$-times then we obtain the following coordinate in $\mathcal{Z}_{n-p}$ :

$$dx^i_{k_i}\setminus{(j_1,\cdots,j_p)}$$  

where the order of occurences of the $j's$ in $(j_1,\cdots,j_p)$ is important just because if $\sigma$ is an element of
the permutation group $S_p$ then the action : 

$$\sigma\cdot dx^i_{k_i}\setminus{(j_1,\cdots,j_p)}:=dx^i_{k_i}\setminus{(j_{\sigma(1)},\cdots,j_{\sigma(p)})}$$

does not imply the equality between $dx^i_{k_i}\setminus{(j_1,\cdots,j_p)}$ and $dx^i_{k_i}\setminus{(j_{\sigma(1)},\cdots,j_{\sigma(p)})}$.

The $j$-dilatation of the coordinate $dx^1_{k_1}\otimes\cdots\otimes dx^j_{k_j}\otimes\cdots dx^n_{k_n}$
is a coordinate in $\mathcal{Z}_{n+1}$ defined by adding in the direction $j$ the guy $dx^{j}_1$ and
re-indexing : 
$$dx^i_{k_i}+j:=dx^1_{k_1}\otimes\cdots\otimes dx^{j-1}_{k_{j-1}}
\otimes dx^j_{k_1}\otimes dx^{j+1}_{k_{j}}
\cdots dx^{n+1}_{k_n}$$

and if we apply these dilatations $p$-times then we obtain the following coordinate in $\mathcal{Z}_{n+p}$ :

$$dx^i_{k_i}+(j_1,\cdots,j_p)$$  
where the order of occurrences of the $j's$ in $(j_1,\cdots,j_p)$ is important. 

\item A \textit{\text{n-}configuration} is given by a family $C_n=\{dx^{i}_{k_{i}}/k_1\in K_1,\cdots,k_n\in K_n, 
\forall i\in\llbracket 1,n\rrbracket, K_i\subset\mathbb{Z}\text{ is a finite set }\}$
of coordinates $dx^{i}_{k_{i}}$ in $\mathcal{Z}_n$. We can also use the notation
$C_n=dx^{i,1}_{k_{i}}+\cdots+dx^{i,l}_{k_{i}}+\cdots+dx^{i,r}_{k_{i}}$ for this $n$-configuration, 
where each $dx^{i,l}_{k_{i}}$ ($l\in\llbracket 1,r\rrbracket$) is a coordinate in $\mathcal{Z}_{n}$
and $r=\sharp(K_1\times\cdots\times K_n)$. This last notation shall be useful especially when we shall 
deal with \textit{divisors} in the section \ref{connected-divisors}.

\begin{remark}
A $n$-configuration $C_n$ must be thought up to its translations in the network $\mathcal{Z}_n$. 
\end{remark}

\item The $j$-contraction $a_j(C_n)$ of the $n$-configuration $C_n$ is given by the following $(n-1)$-configuration :
$a_j(C_n)=a_j(dx^{i,1}_{k_{i}})+\cdots+a_j(dx^{i,l}_{k_{i}})+\cdots+a_j(dx^{i,r}_{k_{i}})$.

\item If $C_n$ is an $n$-configuration then it is straightforward to see that others $n$-configurations $C'_n$ can be
equivalent to it by translations. For example if we write 
$C_n=dx^{i,1}_{k_{i}}+\cdots+dx^{i,l}_{k_{i}}+\cdots+dx^{i,r}_{k_{i}}$ then we can translate it along the direction
$j\in\llbracket 1,n\rrbracket$ with any integer $k\in\mathbb{Z}$ such that the resulting $n$-configurations $C'_n$
has all its coordinates with depth $\geq 1$ for the direction $j$. 

\item Two configurations $C_n$ and $C'_n$ are adjacent if 
\begin{itemize}
\item $C_n\cap C'_n=\varnothing$
\item $\exists j\in\llbracket 1,n\rrbracket$ and $dx^{i}_{k_{i}}\in C_n$, $dx^{i}_{k'_{i}}\in C'_n$ such that 
$a_j(dx^{i}_{k_{i}})=a_j(dx^{i}_{k'_{i}})$
\end{itemize}
If two configurations $C_n$ and $C'_n$ are adjacent then they produce the new configuration $C_n+_{j}C'_n$ :
               $$C_n+_{j}C'_n:=C_n\cup C'_n$$
that we call their \textit{pasting} along the direction $j$.

\item If $C_n=dx^{i,1}_{k_{i}}+\cdots+dx^{i,l}_{k_{i}}+\cdots+dx^{i,r}_{k_{i}}$ is an $n$-configuration then 
we can associate to it its $j$-dilatation $\text{dilat}_j(C_n)$, which is the $(n+1)$-configuration in $\mathcal{Z}_{n+1}$ given by
$\text{dilat}_j(C_n)=C_n+j:=(dx^{i,1}_{k_{i}}+j)+\cdots+(dx^{i,l}_{k_{i}}+j)+\cdots+(dx^{i,r}_{k_{i}}+j)$.

\item A \textit{connected \text{n-}configuration} is given by a family $C_n=\{dx^{i}_{k_{i}}/k_1\in K_1,\cdots,k_n\in K_n, 
\forall i\in\llbracket 1,n\rrbracket, K_i\subset\mathbb{N}\text{ is a finite set }\}$
of coordinates $dx^{i}_{k_{i}}$ in $\mathcal{Z}_n$ such that we have the following 
\textit{connexity property} :

$\forall dx^{i}_{k_{i}}, dx^{i}_{k'_{i}}$ in $C_n$ we have :

$\exists r\in\mathbb{N}$, $\exists l\in\llbracket 1,r\rrbracket$, $\exists k^{l}_{j}\in\mathbb{N}$ where 
each coordinate $dx^{1}_{k^{l}_1}\otimes\cdots\otimes dx^{n}_{k^{l}_n}$ belongs to $C_n$, such that we have 
the following \textit{zigzag of contractions} between $dx^{i}_{k_{i}}$ and $dx^{i}_{k'_{i}}$ :

$a_{i_{1}}(dx^{1}_{k^{1}_1}\otimes\cdots\otimes dx^{n}_{k^{1}_n})=
a_{i_{1}}(dx^{1}_{k^{2}_1}\otimes\cdots\otimes dx^{n}_{k^{2}_n})$

$a_{i_{2}}(dx^{1}_{k^{2}_1}\otimes\cdots\otimes dx^{n}_{k^{2}_n})=
a_{i_{2}}(dx^{1}_{k^{3}_1}\otimes\cdots\otimes dx^{n}_{k^{3}_n})$

...........

$a_{i_{l}}(dx^{1}_{k^{l}_1}\otimes\cdots\otimes dx^{n}_{k^{l}_n})=
a_{i_{l}}(dx^{1}_{k^{l+1}_1}\otimes\cdots\otimes dx^{n}_{k^{l+1}_n})$

..........

$a_{i_{r}}(dx^{1}_{k^{r}_1}\otimes\cdots\otimes dx^{n}_{k^{r}_n})=
a_{i_{r}}(dx^{1}_{k^{r+1}_1}\otimes\cdots\otimes dx^{n}_{k^{r+1}_n})$

where $k^{1}_1=k_1, k^{1}_n=k_n$ and $k^{r+1}_1=k'_1, k^{r+1}_n=k'_n$.

Then we say that the two coordinates $dx^{i}_{k_{i}}$ and $dx^{i}_{k'_{i}}$ are
connected by the zigzag 
$(a_{i_{1}},a_{i_{2}},\cdots,a_{i_{r}})$. Of course two coordinates $dx^{i}_{k_{i}}$ and $dx^{i}_{k'_{i}}$ may
have several equivalents zigzag of contractions.

\item It is straightforward to see that if two connected $n$-configurations $C_n$ and 
$C'_n$ are adjacent, for example along the direction $j$, then their pasting $C_n+_j C'_n$
is still a connected $n$-configuration.

\item We can see that each $n$-configuration in $\mathcal{Z}_n$ is built with subsets in it which are 
connected $n$-configurations. Thus any configuration $C_n$ is written as a formal sum
            $$C_n=C^{1}_n+\cdots+C^{l}_n+\cdots+C^{r}_n$$
such that $r=\sharp\{\text{Connected components of } C_n\}$, and for each $l\in\llbracket 1,r\rrbracket$,
$C^{l}_n$ denote the connected configurations inside $C_n$.

\item Let $C_n=\{dx^{i}_{k_{i}}/k_1\in K_1,\cdots,k_n\in K_n, 
\forall i\in\llbracket 1,n\rrbracket, K_i\subset\mathbb{Z}\text{ is a finite set }\}$ a connected configuration.

\begin{lemma}
 Its $j$-dilatation $\text{dilat}_j(C_n)=\{dx^{i}_{k_{i}}+j/dx^{i}_{k_{i}}\in C_n\}$ is a connected $(n+1)$-configuration.
\end{lemma}

\begin{proof}
We have to prove that if two coordinates $dx^{i}_{k_{i}}+j$ and $dx^{i}_{k'_{i}}+j$ belong to $\text{dilat}_j(C_n)$ then there
is a zigzag of contractions in $\text{dilat}_j(C_n)$ between them. Consider a zigzag 
$(a_{i_{1}},a_{i_{2}},\cdots,a_{i_{r}})$  of contractions between the two coordinates $dx^{i}_{k_{i}}$ and $dx^{i}_{k'_{i}}$ 
 in $C_{n}$. If $l\in\llbracket 1,r\rrbracket$ write $i'_{l}=i_{l}$ if $i_{l}< j$ and 
$i'_{l}=i_{l}+1$ if $i_{l}\geq j$. Then it is easy to see that $(a_{i'_{1}},a_{i'_{2}},\cdots,a_{i'_{r}})$ is such zigzag. 
 \end{proof}
 
 The $j$-dilatation $\text{dilat}_j(C_n)$ of a connected configuration $C_n$ is written also 
 $1^{n}_{n+1,j}(C_n)=1^{n,\gamma=\pm}_{n+1,j}(C_n)=\text{dilat}_j(C_n)$ in order to 
 have a smell of the structure of cubical strict $\infty$-category that we shall put on 
 connected configurations.

\item Let $C_n=\{dx^{i}_{k_{i}}/k_1\in K_1,\cdots,k_n\in K_n, 
\forall i\in\llbracket 1,n\rrbracket, K_i\subset\mathbb{Z}\text{ is a finite set }\}$ a connected configuration.

Its $j$-presources are given by the sub-configuration $C^{j\text{-so}}_{n}\subset C_{n}$ built as follow : 

First for each coordinate $dx^{i}_{k_{i}}=dx^1_{k_1}\otimes\cdots\otimes dx^j_{k_j}\otimes\cdots dx^n_{k_n}$
in $C_n$ we consider the set $C^{j,dx^{i}_{k_{i}}}_{n}\subset C_{n}$ of all coordinates 
$dx^{i}_{k'_{i}}=dx^1_{k'_1}\otimes\cdots\otimes dx^j_{k'_j}\otimes\cdots dx^n_{k'_n}$
in $C_n$ such that for all $l\in\llbracket 1,n\rrbracket\setminus j$, $k'_l=k_l$. 
The set $C^{j,dx^{i}_{k_{i}}}_{n}\subset C_{n}$ is called a $j$-partition of
$C_n$. These partitions of $C_n$ form a finite set and they may not be connected. 
Of course if $dx^{i}_{k'_{i}}\in C^{j,dx^{i}_{k_{i}}}_{n}$ then 
$C^{j,dx^{i}_{k'_{i}}}_{n}=C^{j,dx^{i}_{k_{i}}}_{n}$. Thus when we consider the set 
$C^{j,dx^{i}_{k_{i}}}_{n}$ it means that we have chosen one coordinate
$dx^{i}_{k_{i}}$ representing this set and we use $dx^{i}_{k_{i}}$ to denote this
set. 

Let us fix a $j$-partition $C^{j,dx^{i}_{k_{i}}}_{n}$ of $C_n$. The integer
$\text{min}_j=min\{k'_j\in\mathbb{Z}/dx^{i}_{k'_{i}}\in C^{j,dx^{i}_{k_{i}}}_{n}\}$ provides a specific 
coordinate $dx^{i}_{\text{min}_j}\in C^{j,dx^{i}_{k_{i}}}_{n}$. We isolate these coordinates for all
$j$-partitions of $C_n$. They form the set $C^{j\text{-so}}_{n}\subset C_{n}$ of $j$-presources 
of $C_n$.

The $j$-contraction $contr_j(C^{j\text{-so}}_{n})$ of $C^{j\text{-so}}_{n}$ is given by the following set of coordinates in $\mathcal{Z}_{n-1}$ ::
$$contr_j(C^{j\text{-so}}_{n})=\{dx^{i}_{k_{i}}\setminus j/dx^{i}_{k_{i}}\in C^{j\text{-so}}_{n}\}$$

\begin{lemma}
The $j$-contraction $contr_j(C^{j\text{-so}}_{n})$ of $C^{j\text{-so}}_{n}$ is a connected 
$(n-1)$-configuration.
\end{lemma}

\begin{definition}
If $C_n$ is a connected configuration, its $j$-source is the connected $(n-1)$-configuration
$contr_j(C^{j\text{-so}}_{n})$. We denote it by $\sigma^{n}_{n-1,j}(C_n)$
\end{definition}

Its $j$-pretargets are given by the sub-configuration $C^{j\text{-tar}}_{n}\subset C_{n}$ built as follow : 

As above we fix a $j$-partition $C^{j,dx^{i}_{k_{i}}}_{n}$ of $C_n$.
The integer
$\text{max}_j=max\{k'_j\in\mathbb{Z}/dx^{i}_{k'_{i}}\in C^{j,dx^{i}_{k_{i}}}_{n}\}$ provides a specific 
coordinate $dx^{i}_{\text{max}_j}\in C^{j,dx^{i}_{k_{i}}}_{n}$. We isolate these coordinates for all
$j$-partitions of $C_n$. They form the set $C^{j\text{-tar}}_{n}\subset C_{n}$ of $j$-pretargets 
of $C_n$.

The $j$-contraction $contr_j(C^{j\text{-tar}}_{n})$ of $C^{j\text{-tar}}_{n}$ is given by the following set of coordinates in $\mathcal{Z}_{n-1}$ ::
$$contr_j(C^{j\text{-tar}}_{n})=\{dx^{i}_{k_{i}}\setminus j/dx^{i}_{k_{i}}\in C^{j\text{-tar}}_{n}\}$$

\begin{lemma}
The $j$-contraction $contr_j(C^{j\text{-tar}}_{n})$ of $C^{j\text{-tar}}_{n}$ is a connected 
$(n-1)$-configuration.
\end{lemma}

\begin{definition}
If $C_n$ is a connected configuration, its $j$-target is the connected $(n-1)$-configuration
$contr_j(C^{j\text{-tar}}_{n})$. We denote it by $\tau^{n}_{n-1,j}(C_n)$
\end{definition}

\item Let $C_n=\{dx^{i}_{k_{i}}/k_1\in K_1,\cdots,k_n\in K_n, 
\forall i\in\llbracket 1,n\rrbracket, K_i\subset\mathbb{Z}\text{ is a finite set }\}$ a connected configuration. 
A $j$-move of $C_n$ is given by a new $n$-configuration $j\text{-move}(C_n)$ built as follow :

Consider a $j$-partition $C^{j,dx^{i}_{k_{i}}}_{n}$ of $C_n$ as above.
If $dx^{i}_{k'_{i}}=dx^1_{k_1}\otimes\cdots\otimes dx^j_{k'_j}\otimes\cdots dx^n_{k_n}$ is
in $C^{j,dx^{i}_{k_{i}}}_{n}\subset C_{n}$ then define the new coordinate
$\text{j-trans}_k(dx^{i}_{k'_{i}})=dx^1_{k_1}\otimes\cdots\otimes dx^j_{k'_j+k}\otimes\cdots dx^n_{k_n}$
$(k\in\mathbb{Z})$
as the translation of $dx^{i}_{k'_{i}}$ by the integer $k$ along the direction $j$. Then 
put 
$$\text{j-trans}_k(C^{j,dx^{i}_{k_{i}}}_{n})=\{\text{j-trans}_k(dx^{i}_{k'_{i}})\in\mathcal{Z}_n/dx^{i}_{k'_{i}}\in C^{j,dx^{i}_{k_{i}}}_{n}\}$$

Now suppose that $C_n$ has $m$ $j$-partitions $C^{j,dx^{i}_{k_{i}}}_{n,k}$ 
($m\in\mathbb{N}\text{ and }k\in\llbracket 1,m\rrbracket$). For each
$j$-partitions chose an integer $l_k\in\mathbb{Z}$ ($k\in\llbracket 1,m\rrbracket$). Now for each of these partitions $C^{j,dx^{i}_{k_{i}}}_{n,k}$
of $C_n$, consider their different translations by the integers $l_k$ along the direction $j$ : 
$$\text{j-trans}_{l_k}(C^{j,dx^{i}_{k_{i}}}_{n})$$
The new set of coordinates :
 $$j\text{-move}_{l_1,\cdots,l_m}(C_n):=\underset{k\in\llbracket 1,m\rrbracket}\bigcup\text{j-trans}_{l_k}(C^{j,dx^{i}_{k_{i}}}_{n,k})$$
 is called a $j$-move of $C_n$. Such $j$-move of $C_n$ is denoted 
 $j\text{-move}(C_n)$ when no confusion occur for its underlying translations
 $l_k\in\mathbb{Z}$. A $j$-move of $C_n$ may not be connected. Such $j$-moves are central tools
 to build compositions $\circ^{n}_{n-1,j}$ between connected $n$-configurations.

 \item Now we are going to define some specific $j$-moves of connected configurations which 
 play a central role for the definition of the partial compositions $\circ^{n}_{n-1,j}$ for the
 cubical strict $\infty$-category of connected configurations. 
 
 Let $C_n=\{dx^{i}_{k_{i}}/k_1\in K_1,\cdots,k_n\in K_n, 
\forall i\in\llbracket 1,n\rrbracket, K_i\subset\mathbb{Z}\text{ is a finite set }\}$ 
and $C'_n=\{dx^{i}_{k'_{i}}/k'_1\in K'_1,\cdots,k'_n\in K'_n, 
\forall i\in\llbracket 1,n\rrbracket, K'_i\subset\mathbb{Z}\text{ is a finite set }\}$ 
be two connected $n$-configurations such that $\tau^{n}_{n-1,j}(C_n)=\sigma^{n}_{n-1,j}(C'_n)$.

Our goal is to define a new connected configuration $C'_n\circ^{n}_{n-1,j}C_n$. Thanks to the definitions
of the targets $\tau^{n}_{n-1,j}$ and the sources $\sigma^{n}_{n-1,j}$ we know that such connected
configurations $C_n$ and $C'_n$ must have respectively the same number $m$ of $j$-targets
and $j$-sources, and thus $C_n$ and $C'_n$ must have the same number $m$ of $j$-partitions. In fact
any $j$-partition $C^{j,dx^{i}_{k_{i}}}_{n,k}$ ($k\in\llbracket 1,m\rrbracket$) of $C_n$
correspond to another $j$-partition $C'^{j,dx^{i}_{k'_{i}}}_{n,k}$ of $C'_n$, in the sense
that if $dx^{i}_{k_{i}}=dx^1_{k_1}\otimes\cdots\otimes dx^j_{k_j}\otimes\cdots dx^n_{k_n}$ is
in $C^{j,dx^{i}_{k_{i}}}_{n,k}$ and if $dx^{i}_{k'_{i}}=dx^1_{k'_1}\otimes\cdots\otimes dx^j_{k'_j}\otimes\cdots dx^n_{k'_n}$ is
in $C'^{j,dx^{i}_{k'_{i}}}_{n,k}$ then for all $l\in\llbracket 1,n\rrbracket\setminus j$ we have $k_l=k'_l$. 

We are going to define a $j$-move of $C'_n$, denoted by $j\text{-move}(C'_n)$, in order to glue $C_n$ with $C'_n$ along
the direction $j$. 

Let us fix a $k\in\llbracket 1,m\rrbracket$, i.e we work with the $j$-partition
$C^{j,dx^{i}_{k_{i}}}_{n,k}$ of $C_n$ and its corresponding $j$-partition $C'^{j,dx^{i}_{k'_{i}}}_{n,k}$ of $C'_n$.
In order to build $j\text{-move}(C'_n)$ we are going to define $\text{j-trans}_{l_k}(C'^{j,dx^{i}_{k'_{i}}}_{n,k})$ for 
such $k\in\llbracket 1,m\rrbracket$, and then define :

$$j\text{-move}(C'_n):=\underset{k\in\llbracket 1,m\rrbracket}\bigcup\text{j-trans}_{l_k}(C'^{j,dx^{i}_{k'_{i}}}_{n,k})$$

 We denote $dx^{i}_{\text{min}^{k}_j}\in C'^{j,dx^{i}_{k'_{i}}}_{n}$ the coordinate of the $j$-partition
 $C'^{j,dx^{i}_{k'_{i}}}_{n,k}$ of $C'_n$ which is an element of the set $C'^{j\text{-so}}_{n}$ 
 i.e it is a specific $j$-presource of $C'_n$. Also we denote $dx^{i}_{\text{max}^{k}_j}\in C^{j,dx^{i}_{k_{i}}}_{n}$ 
 the coordinate of the $j$-partition
 $C^{j,dx^{i}_{k_{i}}}_{n,k}$ of $C_n$ which is an element of the set $C^{j\text{-tar}}_{n}$ 
 i.e it is a specific $j$-pretarget of $C_n$. Thus for this fixed $k\in\llbracket 1,m\rrbracket$ we define
 the following $j$-translation : 
 
 \begin{itemize}
 \item If $\text{min}^{k}_j=\text{max}^{k}_j$ then we do the $j$-translation of $C'^{j,dx^{i}_{k'_{i}}}_{n,k}$
 with the translation $l_k=1$ : $\text{j-trans}_{1}(C'^{j,dx^{i}_{k'_{i}}}_{n,k})$. 
 \item If either $\text{min}^{k}_j<\text{max}^{k}_j$ or $\text{min}^{k}_j>\text{max}^{k}_j$ then we do the $j$-translation 
 of $C'^{j,dx^{i}_{k'_{i}}}_{n,k}$ with the translation $l_k=\text{max}^{k}_j-\text{min}^{k}_j+1$ : $\text{j-trans}_{l_k}(C'^{j,dx^{i}_{k'_{i}}}_{n,k})$.  
 \end{itemize}  
 
 Now we add $C_n$ with the $n$-configuration $\text{j-trans}_{l_k}(C'^{j,dx^{i}_{k'_{i}}}_{n,k})$ i.e we do the 
 $j$-pasting of $C_n$ with $\text{j-trans}_{l_k}(C'^{j,dx^{i}_{k'_{i}}}_{n,k})$ and we denote this new connected
 $n$-configuration by : $C_n+\text{j-trans}_{l_k}(C'^{j,dx^{i}_{k'_{i}}}_{n,k})$.
 
 When doing that for all
$k\in\llbracket 1,m\rrbracket$ we obtain a new connected $n$-configuration : 
$$C_n+\text{j-trans}_{l_1}(C'^{j,dx^{i}_{k'_{i}}}_{n,1})+\cdots+\text{j-trans}_{l_k}(C'^{j,dx^{i}_{k'_{i}}}_{n,k})+\cdots+\text{j-trans}_{l_m}(C'^{j,dx^{i}_{k'_{i}}}_{n,m})$$
that we denote by $C'_n\circ^{n}_{n-1,j}C_n$.

Let us denote by $\mathbb{C}\text{on-}\mathbb{C}\text{onf}_n$ the set of connected $n$-configurations
of $\mathcal{Z}_n$. Also let us denote by $\mathbb{C}\text{on-}\mathbb{C}\text{onf}$ the set
of all $n$-configurations for all $n\in\mathbb{N}$. The operations $C'_n\circ^{n}_{n-1,j}C_n$ plus the one 
$1^{n}_{n+1,j}(C_n)=1^{n,\gamma=\pm}_{n+1,j}(C_n)=\text{dilat}_j(C_n)$
(defined above) put on the following cubical set

$$\begin{tikzcd}
 \cdots\mathbb{C}\text{on-}\mathbb{C}\text{onf}_n  \arrow[r, yshift=1.5ex,"\sigma^{n}_{n-1,1}"]
  \arrow[r, yshift=-1.5ex,"\tau^{n}_{n-1,1}"{below}]
  \arrow[r, yshift=4.5ex,dotted] \arrow[r, yshift=-4.5ex,dotted]
  \arrow[r, yshift=7.5ex,"\sigma^{n}_{n-1,i}"] 
  \arrow[r, yshift=-7.5ex,"\tau^{n}_{n-1,i}"{below}]
  \arrow[r, yshift=10.5ex,dotted] 
  \arrow[r, yshift=-10.5ex,dotted]      
\arrow[r, yshift=13.5ex,"\sigma^{n}_{n-1,n}"]
\arrow[r, yshift=-13.5ex,"\tau^{n}_{n-1,n}"{below}]  
& \mathbb{C}\text{on-}\mathbb{C}\text{onf}_{n-1} 
 \cdots \mathbb{C}\text{on-}\mathbb{C}\text{onf}_{4}\arrow[r, yshift=1.5ex,"\sigma^{4}_{3,1}"]
  \arrow[r, yshift=-1.5ex,"\tau^{4}_{3,1}"{below}]
  \arrow[r, yshift=4.5ex,"\sigma^{4}_{3,2}"] \arrow[r, yshift=-4.5ex,"\tau^{4}_{3,2}"{below}]
  \arrow[r, yshift=7.5ex,"\sigma^{4}_{3,3}"] 
  \arrow[r, yshift=-7.5ex,"\tau^{4}_{3,3}"{below}]
  \arrow[r, yshift=10.5ex,"\sigma^{4}_{3,4}"] 
  \arrow[r, yshift=-10.5ex,"\tau^{4}_{3,4}"{below}]      & 
  \mathbb{C}\text{on-}\mathbb{C}\text{onf}_{3}\arrow[r, yshift=1.5ex,"\sigma^{3}_{2,1}"]
  \arrow[r, yshift=-1.5ex,"\tau^{3}_{2,1}"{below}]
  \arrow[r, yshift=4.5ex,"\sigma^{3}_{2,2}"] 
  \arrow[r, yshift=-4.5ex,"\tau^{3}_{2,2}"{below}]
  \arrow[r, yshift=7.5ex,"\sigma^{3}_{2,3}"] 
  \arrow[r, yshift=-7.5ex,"\tau^{3}_{2,3}"{below}]
      & \mathbb{C}\text{on-}\mathbb{C}\text{onf}_{2}  
      \arrow[r, yshift=1.5ex,"\sigma^{2}_{1,1}"]
      \arrow[r, yshift=-1.5ex,"\tau^{2}_{1,1}"{below}]
      \arrow[r, yshift=4.5ex,"\sigma^{2}_{1,2}"] 
      \arrow[r, yshift=-4.5ex,"\tau^{2}_{1,2}"{below}]
      &\mathbb{C}\text{on-}\mathbb{C}\text{onf}_{1}
      \arrow[r, yshift=1.5ex,"\sigma^{1}_{0}"]\arrow[r, yshift=-1.5ex,"\tau^{1}_{0}"{below}]   
        & \mathbb{C}\text{on-}\mathbb{C}\text{onf}_{0}
 \end{tikzcd}$$
 
 a structure of cubical strict $\infty$-category with connections. As we see degeneracies
 which are defined by dilatations of coordinates collapse classical degeneracies and
 connections when $n\geq 1$. Thus this structure is interesting to have a real
 first smell of the formalism that we are going to build in order to reach cubical pasting diagrams. 
 In fact we shall see that cubical pasting diagrams are built by using a reacher version of this formalism of
 connected configurations. This richness allows to distinguished well degeneracies which won't be
 collapsed, but also shall give a more precise view of sources, targets and compositions.


\item If $C_n=\{dx^{i}_{k_{i}}/k_1\in K_1,\cdots,k_n\in K_n, 
\forall i\in\llbracket 1,n\rrbracket, K_i\subset\mathbb{Z}\text{ is a finite set }\}$ is a connected $n$-configuration then 
its $j$-\textit{gluing locus} is the set $C^{j\text{-gluing}}_{n}$ of pairs $(dx^{i}_{k_{i}},dx^{i}_{k'_{i}})$ of coordinates in $C_n$ such 
that $k'_{j}=k_{j}+1$. 

\item Diagrams of the $j$-gluing locus $C^{j\text{-gluing}}_{n}$ of a connected configuration $C_n$ are given by formal diagrams
 
 \begin{tikzcd}
dx^{i}_{k_{i}}&&dx^{i}_{k'_{i}}\\
&dx^{i}_{k_{i}}\setminus j=dx^{i}_{k'_{i}}\setminus j
\arrow[lu,"t^{n}_{n-1,j}"]\arrow[ru,"s^{n}_{n-1,j}"]
\end{tikzcd} 

\item If $C_n=C^{1}_n+\cdots+C^{l}_n+\cdots+C^{r}_n$ is an $n$-configuration such that 
$r=\sharp\{\text{Connected components of } C_n\}$ and for each $l\in\llbracket 1,r\rrbracket$,
$C^{l}_n$ denote the connected configurations inside $C_n$, then its $j$-\textit{gluing locus} is the set 

$$C^{j\text{-gluing}}_{n}=\underset{l\in\llbracket 1,r\rrbracket}\bigcup C^{l,j\text{-gluing}}_{n}$$ 

\item Diagrams of the $j$-gluing locus $C^{j\text{-gluing}}_{n}$ of a configuration 
 $C_n=C^{1}_n+\cdots+C^{l}_n+\cdots+C^{r}_n$ 
 are given by the union of diagrams of each $j$-gluing locus $C^{l,j\text{-gluing}}_{n}$ of its 
 connected components $C^{l}_n$. 

\item Let $C_n=\{dx^{i}_{k_{i}}/k_1\in K_1,\cdots,k_n\in K_n, 
\forall i\in\llbracket 1,n\rrbracket, K_i\subset\mathbb{Z}\text{ is a finite set }\}$ a connected configuration.

Its $j$-predomain is given by the sub-configuration $C^{j\text{-dom}}_{n}\subset C_{n}$ such that 
if $dx^{i}_{k_{i}}=dx^1_{k_1}\otimes\cdots\otimes dx^j_{k_j}\otimes\cdots dx^n_{k_n}$ belongs
to $C^{j\text{-dom}}_{n}$ then the coordinate $dx^1_{k_1}\otimes\cdots\otimes dx^j_{k_j-1}\otimes\cdots dx^n_{k_n}$
doesn't belong to $C_n$.

\begin{lemma}
The $j$-contraction $\text{contr}_{j}(C^{j\text{-dom}}_{n})$ of the $j$-predomain $C^{j\text{-dom}}_{n}$ of $C_{n}$ just above is a 
connected $(n-1)$-configuration. 
\end{lemma}
\begin{proof}
We have to prove that if two coordinates $dx^{i}_{k_{i}}\setminus j$ and $dx^{i}_{k'_{i}}\setminus j$ belong to $\text{contr}_{j}(C^{j\text{-dom}}_{n})$ then there
is a zigzag of contractions in $\text{contr}_{j}(C^{j\text{-dom}}_{n})$ between them. Consider a zigzag 
$(a_{i_{1}},a_{i_{2}},\cdots,a_{i_{r}})$  of contractions between the two coordinates $dx^{i}_{k_{i}}$ and $dx^{i}_{k'_{i}}$ 
of $C^{j\text{-dom}}_{n}$ in $C_{n}$. If $l\in\llbracket 1,r\rrbracket$ write $i'_{l}=i_{l}$ if $i_{l}\leq j$ and 
$i'_{l}=i_{l}-1$ if $i_{l}>j$. Then it is easy to see that $(a_{i'_{1}},a_{i'_{2}},\cdots,a_{i'_{r}})$ is such zigzag.
\end{proof}

The $j$-contraction $\text{contr}_{j}(C^{j\text{-dom}}_{n})$ of the $j$-predomain $C^{j\text{-dom}}_{n}$ of $C_{n}$
is called the $j$-domain of $C_n$ and is written $\Sigma^{n}_{n-1,j}(C_n)$.

The $j$-precodomain of $C_n$ is given by the sub-configuration $C^{j\text{-codom}}_{n}\subset C_{n}$ such that 
if $dx^{i}_{k_{i}}=dx^1_{k_1}\otimes\cdots\otimes dx^j_{k_j}\otimes\cdots dx^n_{k_n}$ belongs
to $C^{j\text{-codom}}_{n}$ then the coordinate $dx^1_{k_1}\otimes\cdots\otimes dx^j_{k_j+1}\otimes\cdots dx^n_{k_n}$
doesn't belong to $C_n$.

\begin{lemma}
The $j$-contraction $\text{contr}_{j}(C^{j\text{-codom}}_{n})$ of the $j$-precodomain $C^{j\text{-codom}}_{n}$ of $C_{n}$ just above is a 
connected $(n-1)$-configuration. 
\end{lemma}
The proof is the same. The $j$-contraction $\text{contr}_{j}(C^{j\text{-codom}}_{n})$ of the $j$-precodomain $C^{j\text{-codom}}_{n}$ of $C_{n}$
is called the $j$-codomain of $C_n$ and is written $\text{T}^{n}_{n-1,j}(C_n)$.

\item Diagrams of the $j$-predomain $C^{j\text{-dom}}_{n}$ of a connected configuration $C_n$ are given by the 
formal arrows

\begin{tikzcd}
dx^{i}_{k_{i}}\setminus j\arrow[rr,"s^{n}_{n-1,j}"]&&dx^{i}_{k_{i}}
\end{tikzcd}
 
 where $dx^{i}_{k_{i}}$ belongs to $C^{j\text{-dom}}_{n}$, and diagrams of the $j$-precodomain 
 $C^{j\text{-codom}}_{n}$ of a connected configuration $C_n$ are given by the 
formal arrows

\begin{tikzcd}
dx^{i}_{k_{i}}\setminus j\arrow[rr,"t^{n}_{n-1,j}"]&&dx^{i}_{k_{i}}
\end{tikzcd}
 
 where $dx^{i}_{k_{i}}$ belongs to $C^{j\text{-codom}}_{n}$.

\item Diagrams of the $j$-predomain $C^{j\text{-dom}}_{n}$ of a configuration $C_n$ are given by the 
the union of all diagrams of the $j$-predomain of its connected components, and diagrams of the $j$-precodomain 
 $C^{j\text{-codom}}_{n}$ of a configuration $C_n$ are given by the union of all diagrams of the 
 $j$-precodomain of its connected components. 

\item A crucial and straightforward fact is that given a coordinate $dx^i_{k_i}$ in $\mathcal{Z}_{n}$, it 
has a \textit{trivial structure} of $n$-cubical set where sources and targets are defined by contractions :
\begin{itemize}
\item $s^{n}_{n-1,j}(dx^i_{k_i})=t^{n}_{n-1,j}(dx^i_{k_i}):=dx^i_{k_i}\setminus{j}$, 

\item $s^{n-p}_{n-p-1,k}(dx^i_{k_i}\setminus{(j_1,\cdots,j_p))}=
t^{n-p}_{n-p-1,k}(dx^i_{k_i}\setminus{(j_1,\cdots,j_p))}:=dx^i_{k_i}\setminus{(j_1,\cdots,j_p,k)}$
\end{itemize}
thus different contractions of $dx^i_{k_i}$ are the faces of its underlying trivial $n$-cubical set. 

However this structure of $n$-cube that $dx^i_{k_i}$ has is too trivial because it does not
distinguished sources and targets with the same direction $j$. And this distinction is crucial because our idea is too
label any $n$-cubical sets $A$ with a coordinate $dx^i_{k_i}$ of $\mathcal{Z}_{n}$, such that
faces of $A$ must have new coordinates $dx^i_{k_i}\setminus{(j_1,\cdots,j_p)}$ build by contractions and
weighted by a notion of sources and targets. In order to correct this default we are going to enriched
the coordinates with a notion of \textit{weighted coordinate} or \textit{link}, which are roughly
speaking \textit{coordinates equipped with or weighted with} the symbols $\{-,+\}$. 

Thus for each coordinate
$dx^i_{k_i}$ of the infinite network $\mathcal{Z}_{n}$  
we shall associate an other $n$-cubical set $\square^{dx^i_{k_i}}_{1(n)}$ called \textit{the box} of $dx^i_{k_i}$ 
and which formalise better the notion of $n$-cubical set
$A$ labelled by $dx^i_{k_i}$, in the sense that sources and targets of $A$ are then labelled with
weighted coordinates, which give the right information of the location of faces of $A$. Without
these weights any $p$-face of $A$ which is a source in the direction $j$ has the same coordinate
(because the trivial structure collapse this source-target information) 
as the other $p$-face of $A$ which is a target in the same direction $j$, and this is 
counterintuitive : the role of $\square^{dx^i_{k_i}}_{1(n)}$ is to distinguished well coordinates
of any faces of any $n$-cubical set labelled with the coordinate $dx^i_{k_i}$. The 
next section is devoted to the description of these boxes $\square^{dx^i_{k_i}}_{1(n)}$.   
 \end{itemize}

\subsection{The basic box $\square^{dx^i_{k_i}}_{1(n)}$ of a coordinate $dx^i_{k_i}$}
\label{basic-box}

Given a coordinate $dx^i_{k_i}$ and the elementary $n$-cube $1(n)$ (which is the unique $n$-cell of the 
cubical sketch $\mathbb{C}$), we associate to it a canonical \textit{free box} $(dx^i_{k_i})=\square^{dx^i_{k_i}}_{1(n)}$ 
which is an $n$-cubical set which faces are congruences of \textit{links}. This $n$-cubical set $\square^{dx^i_{k_i}}_{1(n)}$ 
is called \textit{the basic box of the coordinate} $dx^i_{k_i}$. Its sources and its targets are compatible with contractions and obtained by contraction of $dx^i_{k_i}$,
and its different degeneracies (classical and connections) are compatible with dilatations and obtained by dilatation of $dx^i_{k_i}$.
Its links
are seen as terms of a language equipped with the different
contractions of $dx^i_{k_i} : dx^i_{k_i}\setminus(j_1,j_2,\cdots,j_p)$ plus 
two symbols $\{-,+\}$ which label these contractions. These symbols 
$\{-,+\}$ must be interpreted as sources and targets of the different contractions 
they equipped, and provide a good notion of sources and targets for $\square^{dx^i_{k_i}}_{1(n)}$.
These terms are built inductively (see below) and congruences on it
use notions of zigzag build with the cubical identities of sources and targets (see below).  
An other possible description of faces of $\square^{dx^i_{k_i}}_{1(n)}$
is given in the remark below, which looks more natural (it uses the Reverse Polish Notation), but 
less intuitive for us. Perhaps in the future we would prefer these \textit{RPN notations}. 

In this section we will describe only the underlying cubical set of $\square^{dx^i_{k_i}}_{1(n)}$ and degeneracies
of it shall be described only in the next section, because they are more subtile and involve \textit{notions
of dilated free boxes equipped congruences for degeneracies} (see below). As we wrote in the 
previous section the role of $\square^{dx^i_{k_i}}_{1(n)}$ can be summarized as follow : if a $n$-cubical
set $X$ is labelled by a coordinate $dx^i_{k_i}$ it means that it is contained in the box $\square^{dx^i_{k_i}}_{1(n)}$
which faces are congruences of \textit{weighted coordinates or links}. 
The box $\square^{dx^i_{k_i}}_{1(n)}$ and all faces of $\square^{dx^i_{k_i}}_{1(n)}$ have underlying
free boxes (see below). But when we consider the free box associated to a face of $\square^{dx^i_{k_i}}_{1(n)}$
we forget that it was "linked" to $\square^{dx^i_{k_i}}_{1(n)}$.

In order to keep the \textit{linked information} of the faces of $\square^{dx^i_{k_i}}_{1(n)}$ we write these links
as finite sequences of the form :
 $$X=(dx^i_{k_i},(dx^i_{k_i}\setminus j_1,\pm), (dx^i_{k_i}\setminus(j_1,j_2),\pm),...,
(dx^i_{k_i}\setminus(j_1,j_2,\cdots,j_r),\pm))$$

We can define them by finite decreasing induction :
\begin{definition}
\begin{itemize}
\item For any direction $j\in\llbracket 1,n\rrbracket$, the term $s^{n}_{n-1,j}(\square^{dx^i_{k_i}}_{1(n)})=(dx^i_{k_i},(dx^i_{k_i}\setminus j,-))$ and
the term $t^{n}_{n-1,j}(\square^{dx^i_{k_i}}_{1(n)})=(dx^i_{k_i},(dx^i_{k_i}\setminus j,+))$
 are $1$-links which must be interpreted respectively as the $j$-source and the 
$j$-target of the box $\square^{dx^i_{k_i}}_{1(n)}$. 

\item If $X=(dx^i_{k_i},(dx^i_{k_i}\setminus j_1,\pm), (dx^i_{k_i}\setminus(j_1,j_2),\pm),...,
(dx^i_{k_i}\setminus(j_1,j_2,\cdots,j_{n-r}),\pm))$ is an $(n-r)$-link of the box $\square^{dx^i_{k_i}}_{1(n)}$, then for any direction 
$j\in\llbracket 1,r\rrbracket$, the terms

$$s^{r}_{r-1,j}(X)=(dx^i_{k_i},(dx^i_{k_i}\setminus j_1,\pm), (dx^i_{k_i}\setminus(j_1,j_2),\pm),...,
(dx^i_{k_i}\setminus(j_1,j_2,\cdots,j_{n-r}),\pm),(dx^i_{k_i}\setminus(j_1,j_2,\cdots,j_{n-r},j),-))$$

$$t^{r}_{r-1,j}(X)=(dx^i_{k_i},(dx^i_{k_i}\setminus j_1,\pm), (dx^i_{k_i}\setminus(j_1,j_2),\pm),...,
(dx^i_{k_i}\setminus(j_1,j_2,\cdots,j_{n-r}),\pm),(dx^i_{k_i}\setminus(j_1,j_2,\cdots,j_{n-r},j),+))$$
are $(n-r-1)$-links of $\square^{dx^i_{k_i}}_{1(n)}$.

\item $(n-r)$-links of sources-targets of $\square^{dx^i_{k_i}}_{1(n)}$, or $(n-r)$-links of $\square^{dx^i_{k_i}}_{1(n)}$ for short, are given by such sequences
$$(dx^i_{k_i},(dx^i_{k_i}\setminus j_1,\pm), (dx^i_{k_i}\setminus(j_1,j_2),\pm),...,
(dx^i_{k_i}\setminus(j_1,j_2,\cdots,j_{n-r}),\pm)).$$
\end{itemize}
\end{definition}

Some notations shall be useful :

$$s^{n}_{n_{2},j^{1}}:=s^{n_{2}+1}_{n_{2},j^{1}_{n_{2}+1}}
\circ s^{n_{2}+2}_{n_{2}+1,j^{1}_{n_{2}+2}}\cdots \circ s^{n-1}_{n-2,j^{1}_{n-1}}\circ s^{n}_{n-1,j^{1}_n}$$
where $j^{1}=(j^{1}_n,\cdots,j^{1}_{n_{2}+1})$ and 
$j^{1}_n\in\llbracket 1,n \rrbracket, j^{1}_{n-1}\in\llbracket 1,n-1\rrbracket,\cdots,
j^{1}_{n_{2}+1}\in\llbracket 1,n_{2}+1\rrbracket$

$$t^{n}_{n_{2},j^{1}}:=t^{n_{2}+1}_{n_{2},j^{1}_{n_{2}+1}}
\circ t^{n_{2}+2}_{n_{2}+1,j^{1}_{n_{2}+2}}\cdots \circ t^{n-1}_{n-2,j^{1}_{n-1}}\circ t^{n}_{n-1,j^{1}_n}$$
where $j^{1}=(j^{1}_n,\cdots,j^{1}_{n_{2}+1})$
 and $j^{1}_n\in\llbracket 1,n \rrbracket, j^{1}_{n-1}\in\llbracket 1,n-1\rrbracket,\cdots,
j^{1}_{n_{2}+1}\in\llbracket 1,n_{2}+1\rrbracket$.

Also for any partition $n_{p}<n_{p-1}<\cdots<n_{k}<\cdots<n_{2}<n_{1}=n$ with $(p-1)$
intervals $\llbracket n_{k+1},n_{k}\rrbracket$ we have $6$ different zigzags of sources 
and targets :
\begin{itemize}

\item $s^{n}_{n_{2},j}:=s^{n_{2}+1}_{n_{2},j_{n_{2}+1}}
\circ s^{n_{2}+2}_{n_{2}+1,j_{n_{2}+2}}\cdots \circ s^{n-1}_{n-2,j_{n-1}}\circ s^{n}_{n-1,j_n}$
where $j=(j_n,\cdots,j_{n_{2}+1})$ and 
$j_n\in\llbracket 1,n \rrbracket, j_{n-1}\in\llbracket 1,n-1\rrbracket,\cdots,
j_{n_{2}+1}\in\llbracket 1,n_{2}+1\rrbracket$ called string of sources of type $s$.

\item $t^{n}_{n_{2},j}:=t^{n_{2}+1}_{n_{2},j_{n_{2}+1}}
\circ t^{n_{2}+2}_{n_{2}+1,j_{n_{2}+2}}\cdots \circ t^{n-1}_{n-2,j_{n-1}}\circ t^{n}_{n-1,j_n}$
where $j=(j_n,\cdots,j_{n_{2}+1})$ and 
$j_n\in\llbracket 1,n \rrbracket, j_{n-1}\in\llbracket 1,n-1\rrbracket,\cdots,
j_{n_{2}+1}\in\llbracket 1,n_{2}+1\rrbracket$ called string of targets of type $t$

\item $s^{n_{p-1}}_{n_{p},j^{p-1}}\circ t^{n_{p-2}}_{n_{p-1},j^{p}}\cdots t^{n_{k}}_{n_{k+1},j^{k}}\circ s^{n_{k-1}}_{n_{k},j^{k-1}}\cdots 
t^{n_2}_{n_{3},j^{2}}\circ s^{n}_{n_{2},j^{1}}$ called zigzag of sources-targets of type $(s,s)$.

\item $s^{n_{p-1}}_{n_{p},j^{p-1}}\circ t^{n_{p-2}}_{n_{p-1},j^{p}}\cdots t^{n_{k}}_{n_{k+1},j^{k}}\circ s^{n_{k-1}}_{n_{k},j^{k-1}}\cdots 
s^{n_2}_{n_{3},j^{2}}\circ t^{n}_{n_{2},j^{1}}$ called zigzag of sources-targets of type $(s,t)$.

\item $t^{n_{p-1}}_{n_{p},j^{p-1}}\circ s^{n_{p-2}}_{n_{p-1},j^{p}}\cdots t^{n_{k}}_{n_{k+1},j^{k}}\circ s^{n_{k-1}}_{n_{k},j^{k-1}}\cdots 
s^{n_2}_{n_{3},j^{2}}\circ t^{n}_{n_{2},j^{1}}$ called zigzag of sources-targets of type $(t,t)$.

\item $t^{n_{p-1}}_{n_{p},j^{p-1}}\circ s^{n_{p-2}}_{n_{p-1},j^{p}}\cdots t^{n_{k}}_{n_{k+1},j^{k}}\circ s^{n_{k-1}}_{n_{k},j^{k-1}}\cdots 
t^{n_2}_{n_{3},j^{2}}\circ s^{n}_{n_{2},j^{1}}$ called zigzag of sources-targets of type $(t,s)$
\end{itemize}

The number of occurences of the $s$ and of the $t$ in a string or zigzag is called \textit{the
size} of the string or zigzag. If $X$ is an $r$-link of $\square^{dx^i_{k_i}}_{1(n)}$ :

$$X=(dx^i_{k_i},(dx^i_{k_i}\setminus j_1,\pm), (dx^i_{k_i}\setminus(j_1,j_2),\pm),...,
(dx^i_{k_i}\setminus(j_1,j_2,\cdots,j_r),\pm))$$

then it can be written

$$X=z_X(\square^{dx^i_{k_i}}_{1(n)})$$

where $z_X$ denotes its underlying string or zigzag of sources-targets.

All these zigzags or strings build the $(n-n_{p})$ faces of any $n$-cube. Thanks to the cubical identities two differents 
zigzags or strings can be equal. And these equalities build congruences on the sequences defined below, such that equivalence
relations of these sequences are the faces of the free box $\square^{dx^i_{k_i}}_{1(n)}$.

More precisely consider two $(n-r)$-links
$X=(dx^i_{k_i},(dx^i_{k_i}\setminus j_1,\pm), (dx^i_{k_i}\setminus(j_1,j_2),\pm),...,
(dx^i_{k_i}\setminus(j_1,j_2,\cdots,j_{n-r}),\pm))$ and 
$X'=(dx^i_{k_i},(dx^i_{k_i}\setminus j'_1,\pm), (dx^i_{k_i}\setminus(j'_1,j'_2),\pm),...,
(dx^i_{k_i}\setminus(j'_1,j'_2,\cdots,j'_{n-r}),\pm))$. Denote by $z_X$ the string or zigzag of sources-targets which gives $X$, i.e $X=z_X(\square^{dx^i_{k_i}}_{1(n)})$, and $z_{X'}$ the string or zigzag of sources-targets which gives $X'$, i.e 
$X'=z_{X'}(\square^{dx^i_{k_i}}_{1(n)})$. 

\begin{definition}
With the above notations, the $(n-r)$-link $X$ is congruent to the $(n-r)$-link $X'$ if 
and only if $z_X=z_{X'}$; in this case it is trivial to see that $z_X$ and $z_{X'}$ have the same size. Then we write $X\equiv X'$. An equivalence classe of $(n-r)$-link
of the free box $\square^{dx^i_{k_i}}_{1(n)}$ are $r$-faces of $\square^{dx^i_{k_i}}_{1(n)}$.
\end{definition}



In fact the \textit{terminal element} of the $(n-r)$-link $X$ :
$$(dx^i_{k_i}\setminus(j_1,j_2,\cdots,j_{n-r}),\pm)$$
  gives the precise information of an 
$r$-face of $\square^{dx^i_{k_i}}_{1(n)}$
that it can be a source or a target, depending on the sign in $\{-,+\}$ : $"-"$ means sources and $"+"$ means target. 

\begin{lemma}
If two $(n-r)$-links of $\square^{dx^i_{k_i}}_{1(n)}$ are congruents then they have the same terminal element.
\end{lemma}
\begin{proof}
The proof is easy and is made by finite decreasing induction :

\begin{itemize}
\item We start the induction by proving it with sources and targets of $(dx^i_{k_i})=\square^{dx^i_{k_i}}_{1(n)}$ (by using the whole cubical
identities $ss=ss$, $st=ts$, etc.) and verify that 
indeed they give the same terminal coordinates : this step shows the magical role of the trivial 
cubical structure of the coordinates. See the section above. 

\item We suppose that this is true for two congruent $(n-r)$-links. When we apply sources and targets of these 
$(n-r)$-links then it is straightforward to see that they have the same 
terminal coordinates. 
\end{itemize}
\end{proof} 

A face of $\square^{dx^i_{k_i}}_{1(n)}$ is thus an equivalent classe of links
of $\square^{dx^i_{k_i}}_{1(n)}$ with the same terminal element.

We can have in mind also that 
$(dx^i_{k_i}\setminus(j_1,j_2,\cdots,j_{n-r}),\pm)$ is an $r$-face of $\square^{dx^i_{k_i}}_{1(n)}$ 
equipped with 
(or linked by) the link $$(dx^i_{k_i},(dx^i_{k_i}\setminus j_1,\pm), (dx^i_{k_i}\setminus(j_1,j_2),\pm),...,
(dx^i_{k_i}\setminus(j_1,j_2,\cdots,j_{n-r}),\pm))$$

Thus when there is no confusion about the prescribed link of a face 
$(dx^i_{k_i}\setminus(j_1,j_2,\cdots,j_{n-r}),\pm)$ of $\square^{dx^i_{k_i}}_{1(n)}$ we  
denote this $r$-face of $\square^{dx^i_{k_i}}_{1(n)}$ just by $(dx^i_{k_i}\setminus(j_1,j_2,\cdots,j_{n-r}),\pm)$ without
referring its link in $\square^{dx^i_{k_i}}_{1(n)}$. 

The previous lemma allows to build the \textit{free boxes associate to any faces} of $\square^{dx^i_{k_i}}_{1(n)}$ :

\begin{definition}
The \textit{free box} $(dx^i_{k_i}\setminus(j_1,j_2,\cdots,j_{n-r}))=\square^{dx^i_{k_i}\setminus(j_1,j_2,\cdots,j_{n-r})}_{1(r)}$ of the link
$(dx^i_{k_i},(dx^i_{k_i}\setminus j_1,\pm), (dx^i_{k_i}\setminus (j_1,j_2),\pm),...,
(dx^i_{k_i}\setminus(j_1,j_2,\cdots,j_{n-r}),\pm))$ which represent an $r$-face of 
$\square^{dx^i_{k_i}}_{1(n)}$, is the basic box of the coordinate
$dx^i_{k_i}\setminus(j_1,j_2,\cdots,j_{n-r})$ in $\mathcal{Z}_{r}$. 
\end{definition}

When working with this free box 
$\square^{dx^i_{k_i}\setminus(j_1,j_2,\cdots,j_{{n-r}})}_{1(r)}$, we forget the previous 
information that it was linked to $\square^{dx^i_{k_i}}_{1(n)}$.
Thus the link $(dx^i_{k_i},(dx^i_{k_i}\setminus j_1,\pm), (dx^i_{k_i}\setminus(j_1,j_2),\pm),...,
(dx^i_{k_i}\setminus(j_1,j_2,\cdots,j_{n-r-1}),\pm))$ which represents a face of $\square^{dx^i_{k_i}}_{1(n)}$, 
 represents also a face of the underlying free box $\square^{dx^i_{k_i}\setminus(j_1,j_2,\cdots,j_{n-r})}_{1(r)}$, but with the simpler link
$(dx^i_{k_i}\setminus(j_1,j_2,\cdots,j_{n-r}),(dx^i_{k_i}\setminus(j_1,j_2,\cdots,j_{n-r-1}),\pm))$
when we see it as a face of the free box $\square^{dx^i_{k_i}\setminus(j_1,j_2,\cdots,j_{n-r})}_{1(r)}$.

\begin{remark}
We have others natural notations for links $X$ of $\square^{dx^i_{k_i}}_{1(n)}$ (Reverse Polish Notation, RPN) :

$$X=(dx^i_{k_i},dx^i_{k_i}\setminus j_1,dx^i_{k_i}\setminus(j_1,j_2),...,
dx^i_{k_i}\setminus(j_1,j_2,\cdots,j_{n-r}),\pm,\cdots,\pm)$$

This presentation allow the following definition of sources and targets of links
of $\square^{dx^i_{k_i}}_{1(n)}$ by using underlying free boxes of it :

\begin{itemize}
\item $s^{n}_{n-1,j}(\square^{dx^i_{k_i}}_{1(n)})=(dx^i_{k_i},(dx^i_{k_i}\setminus j,-))$

\item $s^{r}_{r-1,l}(X)=(dx^i_{k_i},dx^i_{k_i}\setminus j_1,dx^i_{k_i}\setminus(j_1,j_2),...,
dx^i_{k_i}\setminus(j_1,j_2,\cdots,j_{n-r-1}),s^{r}_{r-1,l}(\square^{dx^i_{k_i}
\setminus(j_1,j_2,\cdots,j_{n-r})}_{1(r)}),\pm,\cdots,\pm)$
\end{itemize}

Thus 
$$s^{r}_{r-1,l}(X)=(dx^i_{k_i},dx^i_{k_i}\setminus j_1,dx^i_{k_i}\setminus(j_1,j_2),...,
(dx^i_{k_i}\setminus(j_1,j_2,\cdots,j_{n-r}),dx^i_{k_i}\setminus(j_1,j_2,\cdots,j_{n-r},j_{n-(r-1)}=l),-),\pm,\cdots,\pm)$$
that we write when removing redondant occurrences of brackets :

$$s^{r}_{r-1,j}(X)=(dx^i_{k_i},dx^i_{k_i}\setminus j_1,dx^i_{k_i}\setminus(j_1,j_2),...,
dx^i_{k_i}\setminus(j_1,j_2,\cdots,j_{n-r}),dx^i_{k_i}\setminus(j_1,j_2,\cdots,j_{n-r},j_{n-(r-1)}=l),-,\pm,\cdots,\pm)$$

and for targets :

\begin{itemize}
\item $t^{n}_{n-1,j}(\square^{dx^i_{k_i}}_{1(n)})=(dx^i_{k_i},(dx^i_{k_i}\setminus j,+))$

\item $t^{r}_{r-1,l}(X)=(dx^i_{k_i},dx^i_{k_i}\setminus j_1,dx^i_{k_i}\setminus(j_1,j_2),...,
dx^i_{k_i}\setminus(j_1,j_2,\cdots,j_{n-r-1}),t^{r}_{r-1,l}(\square^{dx^i_{k_i}
\setminus(j_1,j_2,\cdots,j_{n-r})}_{1(r)}),\pm,\cdots,\pm)$
\end{itemize}
Thus 
$$t^{r}_{r-1,l}(X)=(dx^i_{k_i},dx^i_{k_i}\setminus j_1,dx^i_{k_i}\setminus(j_1,j_2),...,
(dx^i_{k_i}\setminus(j_1,j_2,\cdots,j_{n-r}),dx^i_{k_i}\setminus(j_1,j_2,\cdots,j_{n-r},j_{n-(r-1)}=l),+),\pm,\cdots,\pm)$$
that we write when removing redondant occurrences of brackets :
$$t^{r}_{r-1,j}(X)=(dx^i_{k_i},dx^i_{k_i}\setminus j_1,dx^i_{k_i}\setminus(j_1,j_2),...,
dx^i_{k_i}\setminus(j_1,j_2,\cdots,j_{n-r}),dx^i_{k_i}\setminus(j_1,j_2,\cdots,j_{n-r},j_{n-(r-1)}=l),+,\pm,\cdots,\pm)$$
\end{remark}

%
%
%
%
%
%
%
%
%
%
%
%

\subsection{Degenerate boxes $(\square^{dx^i_{k_i}}_{1(n)},\equiv_A)$}
\label{degenerate-boxes}

We know that the following forgetful functor : 

$$\begin{tikzcd}
\left[\mathbb{C}_r^{op},\mathbb{S}\text{ets}\right]\arrow[rr,"U"]&&\left[\mathbb{C}^{op},\mathbb{S}\text{ets}\right]=\CS
\end{tikzcd}$$

which sends cubical sets equipped with degeneracies and connections \cite{cam-cubique} to
cubical sets is right adjoint. Its induced monad $\mathbb{R}$ applied to the terminal object 
$1$ of the category $\left[\mathbb{C}^{op},\mathbb{S}\text{ets}\right]$ of cubical sets, gives all kind
of degenerates $n$-cells $A\in\mathbb{R}(1)(n)$ (for all integers $n\in\mathbb{N}$) we need for cubical pasting diagrams. In the
next section we shall describe this monad accurately 
in order to see that it is a cartesian monad.

Now we are going to define the notion of zigzag of degeneracies in order to capture \textit{the depth of a degenerate $n$-cell $A$
} in $\mathbb{R}(1)(n)$ which is the greatest integer $r$ such that $r$-faces of $A$ are of the form $1(r)$. We begin with the
notations :

 $1^{n_{2}}_{n,i^{1}}:=1^{n-1}_{n,i^{1}_{1}}\circ 1^{n-2}_{n-1,i^{1}_{2}}\circ\cdots\circ 1^{n-k}_{n-k+1,i^{1}_{k}}\circ\cdots 1^{n_{2}}_{n_{2}+1,i^{1}_{n-n_{2}}}$ where $i^{1}=(i^{1}_{1},\cdots,i^{1}_{k},\cdots,i^{1}_{n-n_{2}})$, $k\in\llbracket 1,n-n_{2}\rrbracket$
 and 
 $$i^{1}_{1}\in\llbracket 1,n\rrbracket,\cdots i^{1}_{k}\in\llbracket 1,n-k+1\rrbracket\cdots i^{1}_{n-n_{2}}\in\llbracket 1,n_{2}+1\rrbracket.$$

$1^{n_{2},\gamma}_{n,j^{1}}:=1^{n-1,\gamma}_{n,j^{1}_{1}}\circ 1^{n-2,\gamma}_{n-1,j^{1}_{2}}\circ\cdots\circ 1^{n-k,\gamma}_{n-k+1,j^{1}_{k}}\circ\cdots 1^{n_{2},\gamma}_{n_{2}+1,j^{1}_{n-n_{2}}}$ where $j^{1}=(j^{1}_{1},\cdots,j^{1}_{k},\cdots,j^{1}_{n-n_{2}})$, $k\in\llbracket 1,n-n_{2}\rrbracket$
 and 
 $$j^{1}_{1}\in\llbracket 1,n-1\rrbracket,\cdots j^{1}_{k}\in\llbracket 1,n-k\rrbracket\cdots j^{1}_{n-n_{2}}\in\llbracket 1,n_{2}\rrbracket.$$

Also for any partition $n_{p}<n_{p-1}<\cdots<n_{k}<\cdots<n_{2}<n_{1}=n$ with $(p-1)$
intervals $\llbracket n_{k+1},n_{k}\rrbracket$ we have $6$ different zigzags of reflexivities and connections :
\begin{itemize}

\item $1^{n_{2}}_{n,i}:=1^{n-1}_{n,i_{1}}\circ 1^{n-2}_{n-1,i_{2}}\circ\cdots\circ 1^{n-k}_{n-k+1,i_{k}}\circ\cdots 1^{n_{2}}_{n_{2}+1,i_{n-n_{2}}}$ where $i=(i_{1},\cdots,i_{k},\cdots,i_{n-n_{2}})$, $k\in\llbracket 1,n-n_{2}\rrbracket$ called strings of degeneracies of type $1$.

\item $1^{n_{2},\gamma}_{n,j}:=1^{n-1,\gamma}_{n,j_{1}}\circ 1^{n-2,\gamma}_{n-1,j_{2}}\circ\cdots\circ 1^{n-k,\gamma}_{n-k+1,j_{k}}\circ\cdots 1^{n_{2},\gamma}_{n_{2}+1,j_{n-n_{2}}}$ where $j=(j_{1},\cdots,j_{k},\cdots,j_{n-n_{2}})$, $k\in\llbracket 1,n-n_{2}\rrbracket$ called strings of degeneracies of type $\gamma$.

\item $1^{n_2}_{n,i^{1}}\circ 1^{n_3,\gamma}_{n_2,i^{2}}\circ\cdots\circ 1^{n_{k},\gamma}_{n_{k-1},i^{k}}
\circ 1^{n_{k+1}}_{n_{k},i^{k+1}}\circ\cdots\circ 1^{n_{p-1,\gamma}}_{n_{p-2},i^{p-2}}\circ 1^{n_{p}}_{n_{p-1},i^{p-1}}$
called zigzags of degeneracies of type $(1,1)$.

\item $1^{n_2,\gamma}_{n,i^{1}}\circ 1^{n_3}_{n_2,i^{2}}\circ\cdots\circ 1^{n_{k},\gamma}_{n_{k-1},i^{k}}
\circ 1^{n_{k+1}}_{n_{k},i^{k+1}}\circ\cdots\circ 1^{n_{p-1,\gamma}}_{n_{p-2},i^{p-2}}\circ 1^{n_{p}}_{n_{p-1},i^{p-1}}$
called zigzags of degeneracies of type $(\gamma,1)$.

\item $1^{n_2,\gamma}_{n,i^{1}}\circ 1^{n_3}_{n_2,i^{2}}\circ\cdots\circ 1^{n_{k},\gamma}_{n_{k-1},i^{k}}
\circ 1^{n_{k+1}}_{n_{k},i^{k+1}}\circ\cdots\circ 1^{n_{p-1}}_{n_{p-2},i^{p-2}}\circ 1^{n_{p},\gamma}_{n_{p-1},i^{p-1}}$
called zigzags of degeneracies of type $(\gamma,\gamma)$.

\item $1^{n_2}_{n,i^{1}}\circ 1^{n_3,\gamma}_{n_2,i^{2}}\circ\cdots\circ 1^{n_{k},\gamma}_{n_{k-1},i^{k}}
\circ 1^{n_{k+1}}_{n_{k},i^{k+1}}\circ\cdots\circ 1^{n_{p-1}}_{n_{p-2},i^{p-2}}\circ 1^{n_{p},\gamma}_{n_{p-1},i^{p-1}}$
called zigzags of degeneracies of type $(1,\gamma)$.

\end{itemize}
The number of occurrences of the operations $1^{r}_{r+1,i}$, $1^{r,\gamma}_{r+1,i}$ in such 
zigzags or such strings are respectively called \textit{the size of a zigzag} or \textit{the size of a string}. 

\begin{definition}
\label{depth}
Consider an $n$-cell $A\in\mathbb{R}(1)(n)$ which is not equal to $1(n)$. Thus it is 
a degenerate $n$-cell and is build with zigzag or string of degeneracies as
described just above. The \textit{depth} of $A$ is the integer $p\in\mathbb{N}$
such that $A$ is equal to a zigzag of size $n-p$ or a string of size $n-p$ of degeneracies of the $p$-cell $1(p)$ of the cubical site, i.e $A$ is written $a_A(1(p))$ where $a_A$ denotes its underlying string or zigzag of degeneracies and $a_A$ has size $n-p$.
\end{definition}

\begin{remark}
Thanks to the axioms of degeneracies the degenerate $n$-cell $A$ has zigzags or strings
of degeneracies with different shapes and which are equals.
\end{remark} 

Suppose $A$ is a degenerate $n$-cell in $\mathbb{R}(1)(n)$ with
depth $p<n$. Zigzags or strings of sources-targets of $A$ with sizes which are less or equal to $(n-p)$ are the
one which build a congruence $\equiv_A$ on faces of the basic $n$-box 
$\square^{dx^i_{k_i}}_{1(n)}$, and this congruence is defined as follow : if $p<q\leq n$, two $q$-faces
$x$ and $y$ of $\square^{dx^i_{k_i}}_{1(n)}$ are $A$-congruent : $x\equiv_A y$ if and only if 
any strings or zigzags of sources-targets $z_x$  of $x$ (i.e $z_x$ is the underlying string or the underlying zigzag of sources-targets of any link of $\square^{dx^i_{k_i}}_{1(n)}$ 
which gives the $q$-face $x$ (any two such links are equivalent)) and any strings or zigzags of sources-targets $z_y$
(i.e $z_y$ is the underlying string or the underlying zigzag of sources-targets of any link of $\square^{dx^i_{k_i}}_{1(n)}$ 
which gives the $q$-face $y$ (any two such links are equivalent)) of $y$, equalize $A$ i.e are such that $z_x(A)=z_y(A)$. 

The quotient $\square^{dx^i_{k_i}}_{1(n)}/\equiv_A$ is a boxe with coordinate $dx^i_{k_i}$ such 
that it sources and targets agree with those of $A$. We denote it with the bracket notation
$(\square^{dx^i_{k_i}}_{1(n)},\equiv_A)$. 

\begin{definition}
\begin{itemize}
\item Sources and targets of degenerate boxes :
$$s^{n}_{n-1,j}((\square^{dx^i_{k_i}}_{1(n)},\equiv_A)):=(\square^{dx^i_{k_i}\setminus j}_{1(n-1)},\equiv_{s^{n}_{n-1,j}(A)})$$
and
$$t^{n}_{n-1,j}((\square^{dx^i_{k_i}}_{1(n)},\equiv_A)):=(\square^{dx^i_{k_i}\setminus j}_{1(n-1)},\equiv_{t^{n}_{n-1,j}(A)})$$

\item Degeneracies of degenerate boxes :
$$1^{n}_{n+1,j}((\square^{dx^i_{k_i}}_{1(n)},\equiv_A)):=(\square^{dx^i_{k_i}+j}_{1(n+1)},\equiv_{1^{n}_{n+1,j}(A)})$$
 and 
$$1^{n,\gamma}_{n+1,j}((\square^{dx^i_{k_i}}_{1(n)},\equiv_A)):=(\square^{dx^i_{k_i}+j}_{1(n+1)},\equiv_{1^{n,\gamma}_{n+1,j}(A)})$$ 
\end{itemize}
\end{definition}

\begin{definition}
A \textit{basic divisor} is the expression $A dx^i_{k_i}$ which mean that the $n$-cell 
$A\in\mathbb{R}(1)(n)$ has coordinate $dx^i_{k_i}$ and when we write $A dx^i_{k_i}$;
we furthermore mean that $A$ is located in its degenerate box 
$(\square^{dx^i_{k_i}}_{1(n)},\equiv_A)$.
\end{definition}

We use the following notations for sources and targets of basic divisors :
\begin{itemize}
\item $s^{n}_{n-1,j}(Adx^i_{k_i}):=s^{n}_{n-1,j}(A)dx^i_{k_i}\setminus j$

\item $t^{n}_{n-1,j}(Adx^i_{k_i}):=t^{n}_{n-1,j}(A)dx^i_{k_i}\setminus j$
\end{itemize}

With it we get two formal inclusions :

$$\begin{tikzcd}
s^{n}_{n-1,j}(Adx^{i}_{k_i})\arrow[rr,"s^{n}_{n-1,j}"]&&Adx^{i}_{k_i}
\end{tikzcd}\qquad
\begin{tikzcd}
t^{n}_{n-1,j}(Adx^{i}_{k_i})\arrow[rr,"t^{n}_{n-1,j}"]&&Adx^{i}_{k_i}
\end{tikzcd},$$

We use the following notations for degeneracies of basic divisors :

\begin{itemize}
\item $1^{n}_{n+1,j}(Adx^i_{k_i}):=1^{n}_{n+1,j}(A)(dx^i_{k_i}+j)$
\item $1^{n,\gamma}_{n+1,j}(Adx^i_{k_i}):=1^{n,\gamma}_{n+1,j}(A)(dx^i_{k_i}+j)$
\end{itemize}

Two basic divisors $Adx^{i}_{k_i}$, $A'dx^{i}_{k'_i}$ in $X$ located
respectively in the coordinates $dx^{i}_{k_i}=dx^{1}_{k_1}\otimes\cdots\otimes dx^{n}_{k_n}$ and
$dx^{i}_{k'_i}=dx^{1}_{k'_1}\otimes\cdots\otimes dx^{n}_{k'_n}$ are $j$-adjacent for a direction 
$j\in\llbracket 1,n\rrbracket$ if their coordinates are $j$-adjacent and if 
$s^{n}_{n-1,j}(Adx^{i}_{k_i})=t^{n}_{n-1,j}(A'dx^{i}_{k'_i})$ if $k_j=k'_j+1$ or 
$t^{n}_{n-1,j}(Adx^{i}_{k_i})=s^{n}_{n-1,j}(A'dx^{i}_{k'_i})$ if $k_j=k'_j-1$. An
\textit{isolated basic divisor} in $X$ is a basic divisor $Adx^{i}_{k_i}$ which is not
$j$-adjacent to any other basic divisor of $X$ for all direction $j\in\llbracket 1,n\rrbracket$.

Also we have the following simple fact :
\begin{proposition} 
Any basic divisor has an underlying structure of cubical set with connections.
\end{proposition}

The set of basic divisors is denoted $\mathbb{B}\text{Div}$ and by the previous proposition it 
is straightforward that it has an underlying structure of cubical set with connections where its 
$n$-cells are the basic $n$-divisors. Consider the full subcategory 
$\Theta_{\mathbb{B}\text{Div}}\subset\CS$ which objects are
basic divisors. The Yoneda embedding 
$$\begin{tikzcd}
\Theta_{\mathbb{B}\text{Div}}\arrow[rr,"\text{Y}"]&&\CS\\
X\arrow[rr,mapsto]&&\text{hom}_{\CS}(X,-)
\end{tikzcd}$$
shall be useful in the next section. However a little comment 
is necessary here.

\subsection{The monad of reflexive cubical sets}
As we wrote in the previous section the forgetful functor : 

$$\begin{tikzcd}
\left[\mathbb{C}_r^{op},\mathbb{S}\text{ets}\right]\arrow[rr,"U"]&&\left[\mathbb{C}^{op},\mathbb{S}\text{ets}\right]
\end{tikzcd}$$

which sends cubical sets equipped with degeneracies and connections \cite{cam-cubique} to
cubical sets is right adjoint and its induced monad is written $\mathbb{R}=(R,i,m)$ where 
\begin{tikzcd}
1_{\CS}\arrow[rr,"i"]&&R
\end{tikzcd}
 is its unit and 
\begin{tikzcd}
R^2\arrow[rr,"m"]&&R
\end{tikzcd}
is its multiplication. 

If $C\in\CS$ is a cubical set, then we put :

$$R(C):=\underset{X\in\mathbb{B}\text{Div}}\bigcup{\text{hom}_{\CS}(\text{Y}(X),C)}$$

The multiplication $m$ of the monad $\mathbb{R}$ is very simple : it is obtained with the concatenation of two strings of degeneracies, or one string of degeneracies with one zigzag of degeneracies, or with two zigzags of degeneracies. The unit $i$ of the monad $\mathbb{R}$ sends $n$-cells $c$ to the decorated box $cdx^{i}_{k_{i}}$.

Let us be more precise :
the multiplication
\begin{tikzcd}
R^2(C)\arrow[rr,"m"]&&R(C)
\end{tikzcd}
is defined as follow : the cubical set $R^2(C)$ is defined by the formula :

$$R^2(C)=\underset{X\in\mathbb{B}\text{Div}}
\bigcup{\text{hom}_{\CS}\Big(Y(X),R(C)=\underset{X'\in\mathbb{B}\text{Div}}
\bigcup{\text{hom}_{\CS}(Y(X'),C)}\Big)}$$

thus an $n$-cell $x$ of $R^2(C)$ is an expression of the form :
$z(z'(c))$ where $c$ is a $p$-cell of $C$, $p\leq n$ (for the case $p=n$ it means that $x$ is non-degenerate and equal to $c$), $z'$ is a string or a zigzag of degeneracies which when apply to $c$ gives a degenerate 
$q$-cell $z'(c)$ of $R(C)$ ($p<q\leq n$), and where $z$ is a string or a zigzag of degeneracies which degenerates again $z'(c)$. The multiplication $m$ sends $z(z'(c))\in R^2(C)$ to $(z+z')(c)\in R(C)$ where
here $z+z'$ is just the concatenation of $z$ and $z'$.

\begin{proposition}
The monad $\mathbb{R}=(R,i,m)$ of cubical reflexive sets with connections is cartesian
\end{proposition}
\label{monad-R}

\begin{proof}

The definition of the endofunctor $R$ shows that it preserves fiber products.

We are going to prove that the multiplication $m$ is cartesian, i.e we are going to prove
that if $C\in\CS$ is a cubical set then the commutative diagram :

$$\begin{tikzcd}
R^{2}(C)\arrow[dd,"m(C)"{left}]\arrow[rr,"R^{2}(!)"]&&R^{2}(1)\arrow[dd,"m(1)"]\\\\
R(C)\arrow[rr,"R(!)"{below}]&&R(1)
\end{tikzcd}$$

is a cartesian square. Consider the commutative diagram in $\CS$ :
$$\begin{tikzcd}
C'\arrow[dd,"f"{left}]\arrow[rr,"g"]&&R^{2}(1)\arrow[dd,"m(1)"]\\\\
R(C)\arrow[rr,"R(!)"{below}]&&R(1)
\end{tikzcd}$$

Thus if $x$ is an $n$-cell of $C'$ then $f(x)=z(c)$ where 
$c\in C(q)$ ($q\leq n$) and $R(!)(f(x))=R(!)(z(c))=z(1(q))$, and
$g(x)=z"(z'(1(p)))$, thus $m(1)(g(x))=m(1)(z"(z'(1(p))))=(z"+z')(1(p))$,
thus the commutativity of the square gives $z=z"+z'$ and $p=q$

$$\begin{tikzcd}
C'\arrow[rrdd,"l",near end, dotted]\arrow[rrdddd,"f"{left}]\arrow[rrrrdd,"g"]\\\\
&&R^{2}(C)\arrow[dd,"m(C)"{left},near start]\arrow[rr,"R^{2}(!)"]&&R^{2}(1)\arrow[dd,"m(1)"]\\\\
&&R(C)\arrow[rr,"R(!)"{below}]&&R(1)
\end{tikzcd}$$

Thus the unique arrow $l$ is defined as follow : $l(x)=z"(z'(c))$, and
we can see that $m(C)(z"(z'(c)))=(z"+z')(c)=z(c)=f(x)$ and 
that $R^2(!)(z"(z'(c)))=z"(z'(1(p)))=g(x)$.

The cartesianity of the unit
$$\begin{tikzcd}
C\arrow[rr,"i"]&&R(C)
\end{tikzcd}$$
is easier and goes as follow : we start with a commutative diagram in $\CS$
$$\begin{tikzcd}
C'\arrow[dd,"f"{left}]\arrow[rr,"!"]&&1\arrow[dd,"i(1)"]\\\\
R(C)\arrow[rr,"R(!)"{below}]&&R(1)
\end{tikzcd}$$
Let $x$ be an $n$-cell of $C'$, thus we have $f(x)=z(c)$, thus 
$R(!)(z(c))=z(1(p))$ and the commutativity gives: 
$z(1(p))=i(1)(1(n))=1(n)$; which shows that
$z=\emptyset$ and $p=n$, thus $f(x)=c$.

 It shows that there 
is a unique map $l$ :
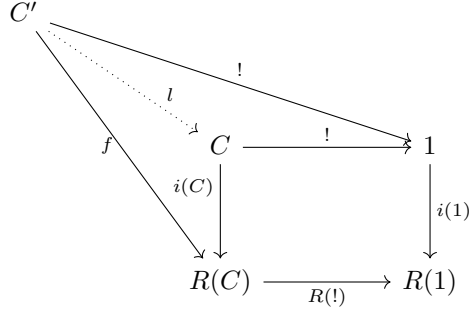
$$\begin{tikzcd}
C'\arrow[rrdd,"l",near end, dotted]\arrow[rrdddd,"f"{left}]\arrow[rrrrdd,"!"]\\\\
&&C\arrow[dd,"i(C)"{left},near start]\arrow[rr,"!"]&&1\arrow[dd,"i(1)"]\\\\
&&R(C)\arrow[rr,"R(!)"{below}]&&R(1)
\end{tikzcd}$$
defined by 
$l(x)=f(x)$.
\end{proof}

\subsection{Rectangular divisors}
\label{connected-divisors}

\begin{definition}
A $n$-\textit{divisor} is a configuration $C_n$ equipped with chosen boxes for each coordinate in it,
thus it is an expression :
$$X=A_{(k^{1}_{1},\cdots,k^{1}_{n})}dx^{i}_{k^{1}_i}+\cdots+
A_{(k^{l}_{1},\cdots,k^{l}_{n})}dx^{i}_{k^{l}_i}+\cdots+A_{(k^{r}_{1},
\cdots,k^{r}_{n})}dx^{i}_{k^{r}_i}$$
where $Adx^{i}_{k_i}$ are basic divisors.
\end{definition}

\begin{remark}
A $n$-divisor $X$ must be thought up to its translations in the network $\mathcal{Z}_n$. Coordinates
are used as guides to build their associated sketches \ref{sketches}.
\end{remark}

\begin{proposition}
Any $n$-divisor has an underlying structure of cubical set.
\end{proposition}

\begin{proof}
If $X=A_{(k^{1}_{1},\cdots,k^{1}_{n})}dx^{i}_{k^{1}_i}+\cdots+
A_{(k^{l}_{1},\cdots,k^{l}_{n})}dx^{i}_{k^{l}_i}+\cdots+A_{(k^{r}_{1},
\cdots,k^{r}_{n})}dx^{i}_{k^{r}_i}$ is a $n$-divisor, then put :

\begin{itemize}
\item $j$-sources $(j\in\llbracket 1,n\rrbracket)$ are given by :
$$s^{n}_{n-1,j}(X)=s^{n}_{n-1,j}(A_{(k^{1}_{1},\cdots,k^{1}_{n})}dx^{i}_{k^{1}_i})+\cdots+
s^{n}_{n-1,j}(A_{(k^{l}_{1},\cdots,k^{l}_{n})}dx^{i}_{k^{l}_i})+\cdots+s^{n}_{n-1,j}(A_{(k^{r}_{1},
\cdots,k^{r}_{n})}dx^{i}_{k^{r}_i})$$

\item $j$-targets $(j\in\llbracket 1,n\rrbracket)$ are given by :
$$t^{n}_{n-1,j}(X)=s^{n}_{n-1,j}(A_{(k^{1}_{1},\cdots,k^{1}_{n})}dx^{i}_{k^{1}_i})+\cdots+
t^{n}_{n-1,j}(A_{(k^{l}_{1},\cdots,k^{l}_{n})}dx^{i}_{k^{l}_i})+\cdots+t^{n}_{n-1,j}(A_{(k^{r}_{1},
\cdots,k^{r}_{n})}dx^{i}_{k^{r}_i})$$

\item If $x\in f_p(X)$ is a $p$-face of $X$, then it is an $(n-p)$-divisor obtained by a zigzag of 
sources-targets of the $n$-divisor $X$.

\end{itemize}
\end{proof}

Some notions attached to $n$-divisors shall be useful :

\begin{itemize}
\item Let $X=A_{(k^{1}_{1},\cdots,k^{1}_{n})}dx^{i}_{k^{1}_i}+\cdots+
A_{(k^{l}_{1},\cdots,k^{l}_{n})}dx^{i}_{k^{l}_i}+\cdots+A_{(k^{r}_{1},
\cdots,k^{r}_{n})}dx^{i}_{k^{r}_i}$ be a divisor. If two basic divisors 
$Adx^{i}_{k_i}$, $A'dx^{i}_{k'_i}$ in $X$ are $j$-adjacent we get two possible diagrams :

$$
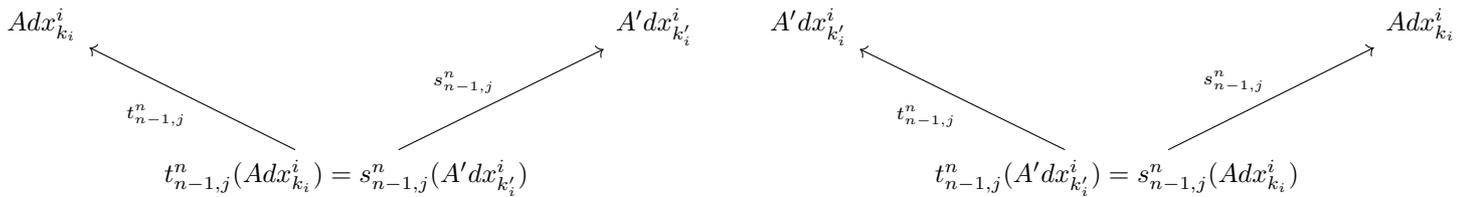
\begin{tikzcd}
Adx^{i}_{k_i}&&A'dx^{i}_{k'_i}\\\\
&t^{n}_{n-1,j}(Adx^{i}_{k_i})=s^{n}_{n-1,j}(A'dx^{i}_{k'_i})\arrow[luu,"t^{n}_{n-1,j}"]\arrow[ruu,"s^{n}_{n-1,j}"]
\end{tikzcd}\qquad
\begin{tikzcd}
A'dx^{i}_{k'_i}&&Adx^{i}_{k_i}\\\\
&t^{n}_{n-1,j}(A'dx^{i}_{k'_i})=s^{n}_{n-1,j}(Adx^{i}_{k_i})\arrow[luu,"t^{n}_{n-1,j}"]\arrow[ruu,"s^{n}_{n-1,j}"]
\end{tikzcd}$$

\end{itemize}

\begin{itemize}

\item A $j$-gluing data ($j\in\llbracket 1,n\rrbracket$) for $X$ is given by a couple $(A_{(k^{l}_{1},\cdots,k^{l}_{n})}dx^{i}_{k^{l}_i},
A_{(k^{l'}_{1},\cdots,k^{l'}_{n})}dx^{i}_{k^{l'}_i})$ of basic divisors of $X$ which are $j$-adjacent and are such that 

$$t^{n}_{n-1,j}(A_{(k^{l}_{1},\cdots,k^{l}_{n})}dx^{i}_{k^{l}_i})=s^{n}_{n-1,j}(A_{(k^{l'}_{1},\cdots,k^{l'}_{n})}dx^{i}_{k^{l'}_i})$$.

Such $j$-gluing data can be written $A_{(k^{l}_{1},\cdots,k^{l}_{n})}dx^{i}_{k^{l}_i}+_{j}
A_{(k^{l'}_{1},\cdots,k^{l'}_{n})}dx^{i}_{k^{l'}_i}$ and underlies the diagram of the type 

$$\begin{tikzcd}
A_{(k^{l}_{1},\cdots,k^{l}_{n})}dx^{i}_{k^{l}_i}&&A_{(k^{l'}_{1},\cdots,k^{l'}_{n})}dx^{i}_{k^{l'}_i}\\\\
&t^{n}_{n-1,j}(A_{(k^{l}_{1},\cdots,k^{l}_{n})}dx^{i}_{k^{l}_i})=s^{n}_{n-1,j}(A_{(k^{l'}_{1},\cdots,k^{l'}_{n})}dx^{i}_{k^{l'}_i})\arrow[luu,"t^{n}_{n-1,j}"]\arrow[ruu,"s^{n}_{n-1,j}"]
\end{tikzcd}$$

called a \textit{basic $j$-gluing locus} of $X$. The $j$-\textit{gluing locus} of $X$ is the set of such basic $j$-gluing locus. The \textit{gluing locus} of $X$ is the set of all $j$-gluing locus for all the direction $j\in\llbracket 1,n\rrbracket$.

\item A basic divisor $A_{(k^{l'}_{1},\cdots,k^{l'}_{n})}dx^{i}_{k^{l'}_i}$ of a divisor $X$ has \textit{free $j$-source} if there is no basic 
divisor $A_{(k^{l}_{1},\cdots,k^{l}_{n})}dx^{i}_{k^{l}_i}$ of $X$ such that the couple 
$(A_{(k^{l}_{1},\cdots,k^{l}_{n})}dx^{i}_{k^{l}_i},
A_{(k^{l'}_{1},\cdots,k^{l'}_{n})}dx^{i}_{k^{l'}_i})$ is a $j$-gluing data. The set of formal arrows :

$$\begin{tikzcd}
s^{n}_{n-1,j}(A_{(k^{l'}_{1},\cdots,k^{l'}_{n})}dx^{i}_{k^{l'}_i})\arrow[rr,"s^{n}_{n-1,j}"]&&A_{(k^{l'}_{1},\cdots,k^{l'}_{n})}dx^{i}_{k^{l'}_i}
\end{tikzcd}$$

associated to basic divisors in $X$ which have free $j$-sources is called the \textit{$j$-free sources locus} of $X$. The 
\textit{free sources locus} of $X$ is the set of all $j$-free sources locus for all the direction $j\in\llbracket 1,n\rrbracket$.

\item A basic divisor $A_{(k^{l}_{1},\cdots,k^{l}_{n})}dx^{i}_{k^{l}_i}$ of a divisor $X$ has \textit{free $j$-target} if there is no basic 
divisor $A_{(k^{l'}_{1},\cdots,k^{l'}_{n})}dx^{i}_{k^{l'}_i}$ of $X$ such that the couple 
$(A_{(k^{l}_{1},\cdots,k^{l}_{n})}dx^{i}_{k^{l}_i},
A_{(k^{l'}_{1},\cdots,k^{l'}_{n})}dx^{i}_{k^{l'}_i})$ is a $j$-gluing data. The set of formal arrows :

$$\begin{tikzcd}
t^{n}_{n-1,j}(A_{(k^{l}_{1},\cdots,k^{l}_{n})}dx^{i}_{k^{l}_i})\arrow[rr,"t^{n}_{n-1,j}"]&&A_{(k^{l}_{1},\cdots,k^{l}_{n})}dx^{i}_{k^{l}_i}
\end{tikzcd}$$

associated to basic divisors in $X$ which have free $j$-targets is called the \textit{$j$-free targets locus} of $X$. The 
\textit{free targets locus} of $X$ is the set of all $j$-free targets locus for all the direction $j\in\llbracket 1,n\rrbracket$.

 Consider a $n$-divisor $X=A_{(k^{1}_{1},\cdots,k^{1}_{n})}dx^{i}_{k^{1}_i}+\cdots+
A_{(k^{l}_{1},\cdots,k^{l}_{n})}dx^{i}_{k^{l}_i}+\cdots+A_{(k^{r}_{1},
\cdots,k^{r}_{n})}dx^{i}_{k^{r}_i}$. We are going to define degeneracies of $X$ :
$1^{n}_{n+1,i}(X)$, $1^{n,\gamma}_{n+1,i}(X)$ by using \textit{the 
axioms of degeneracies-compositions} :

\begin{definition}

\begin{itemize}
\item $1^{n}_{n+1,i}(a\circ^{n}_{j}b)=1^{n}_{n+1,i}(a)\circ^{n+1}_{j+1}1^{n}_{n+1,i}(b)$ if $1\leq i\leq j\leq n$
  
  $1^{n}_{n+1,i}(a\circ^{n}_{j}b)=1^{n}_{n+1,i}(a)\circ^{n+1}_{j}1^{n}_{n+1,i}(b)$ if $1\leq j<i\leq n+1$  
  
  \item $1^{n,\gamma}_{n+1,i}(a\circ^{n}_{j}b)=1^{n,\gamma}_{n+1,i}(a)\circ^{n+1}_{j+1}1^{n,\gamma}_{n+1,i}(b)$ if $1\leq i< j\leq n$
  
  $1^{n,\gamma}_{n+1,i}(a\circ^{n}_{j}b)=1^{n,\gamma}_{n+1,i}(a)\circ^{n+1}_{j}1^{n,\gamma}_{n+1,i}(b)$ if $1\leq j< i\leq n$
  
\item First transport laws : for $1\leq j\leq n$
\[
1^{n,+}_{n+1,j}(a\circ^{n}_{j}b)=
  \begin{bmatrix}
    1^{n,+}_{n+1,j}(a) & 1^{n}_{n+1,j}(a) \\
    1^{n}_{n+1,j+1}(a)& 1^{n,+}_{n+1,j}(b) 
 \end{bmatrix}
\]

   \item Second transport laws : for $1\leq j\leq n$
\[
1^{n,-}_{n+1,j}(a\circ^{n}_{j}b)=
  \begin{bmatrix}
    1^{n,-}_{n+1,j}(a) & 1^{n}_{n+1,j+1}(b) \\
    1^{n}_{n+1,j}(b)& 1^{n,-}_{n+1,j}(b) 
 \end{bmatrix}
\]
\end{itemize}
\end{definition}

\begin{definition}
Consider a divisor $X=A_{(k^{1}_{1},\cdots,k^{1}_{n})}dx^{i}_{k^{1}_i}+\cdots+
A_{(k^{l}_{1},\cdots,k^{l}_{n})}dx^{i}_{k^{l}_i}+\cdots+A_{(k^{r}_{1},
\cdots,k^{r}_{n})}dx^{i}_{k^{r}_i}$. If $Adx^{i}_{k^{l}_i}$
is an isolated basic divisor of $X$ then we just define
$1^{n}_{n+1,i}(Adx^{i}_{k^{l}_i})$, $1^{n,-}_{n+1,i}(Adx^{i}_{k^{l}_i})$
and $1^{n,+}_{n+1,i}(Adx^{i}_{k^{l}_i})$ as in \ref{degenerate-boxes};
for the direction $j\in\llbracket 1,n\rrbracket$ consider 
a $j$-gluing data $A_{(k^{l}_{1},\cdots,k^{l}_{n})}dx^{i}_{k^{l}_i}+_j
A_{(k^{l'}_{1},\cdots,k^{l'}_{n})}dx^{i}_{k^{l'}_i}$ of $X$ (see just above). 
\begin{itemize}
\item If $1\leq i\leq j\leq n$ or if $1\leq j<i\leq n+1$, then we put :
$$1^{n}_{n+1,i}(A_{(k^{l}_{1},\cdots,k^{l}_{n})}dx^{i}_{k^{l}_i}+_j
A_{(k^{l'}_{1},\cdots,k^{l'}_{n})}dx^{i}_{k^{l'}_i}):=1^{n}_{n+1,i}(A_{(k^{l}_{1},\cdots,k^{l}_{n})}dx^{i}_{k^{l}_i})+
1^{n}_{n+1,i}(A_{(k^{l'}_{1},\cdots,k^{l'}_{n})}dx^{i}_{k^{l'}_i})$$

\item if $1\leq i< j\leq n$ or $1\leq j< i\leq n$ the we put :
$$1^{n,\gamma}_{n+1,i}(A_{(k^{l}_{1},\cdots,k^{l}_{n})}dx^{i}_{k^{l}_i}+_j
A_{(k^{l'}_{1},\cdots,k^{l'}_{n})}dx^{i}_{k^{l'}_i}):=1^{n,\gamma}_{n+1,i}(A_{(k^{l}_{1},\cdots,k^{l}_{n})}dx^{i}_{k^{l}_i})+
1^{n,\gamma}_{n+1,i}(A_{(k^{l'}_{1},\cdots,k^{l'}_{n})}dx^{i}_{k^{l'}_i})$$

\item If $1\leq j\leq n$ then we put : 
$$1^{n,+}_{n+1,j}(A_{(k^{l}_{1},\cdots,k^{l}_{n})}dx^{i}_{k^{l}_i}+_j
A_{(k^{l'}_{1},\cdots,k^{l'}_{n})}dx^{i}_{k^{l'}_i}):=1^{n,+}_{n+1,j}(A_{(k^{l}_{1},\cdots,k^{l}_{n})}dx^{i}_{k^{l}_i})+\\
1^{n,+}_{n+1,j}(A_{(k^{l'}_{1},\cdots,k^{l'}_{n})}dx^{i}_{k^{l'}_i})+1^{n}_{n+1,j}(A_{(k^{l}_{1},\cdots,k^{l}_{n})}dx^{i}_{k^{l}_i})
+1^{n}_{n+1,j+1}(A_{(k^{l}_{1},\cdots,k^{l}_{n})}dx^{i}_{k^{l}_i})$$

and 
$$1^{n,-}_{n+1,j}(A_{(k^{l}_{1},\cdots,k^{l}_{n})}dx^{i}_{k^{l}_i}+_j
A_{(k^{l'}_{1},\cdots,k^{l'}_{n})}dx^{i}_{k^{l'}_i}):=1^{n,-}_{n+1,j}(A_{(k^{l}_{1},\cdots,k^{l}_{n})}dx^{i}_{k^{l}_i})+
1^{n,-}_{n+1,j}(A_{(k^{l'}_{1},\cdots,k^{l'}_{n})}dx^{i}_{k^{l'}_i})+1^{n}_{n+1,j}(A_{(k^{l'}_{1},\cdots,k^{l'}_{n})}dx^{i}_{k^{l'}_i})
+1^{n}_{n+1,j+1}(A_{(k^{l'}_{1},\cdots,k^{l'}_{n})}dx^{i}_{k^{l'}_i})$$
\end{itemize}
These give the definition of $1^{n}_{n+1,i}(X)$ for $i\in\llbracket 1,n+1\rrbracket$ and of $1^{n,-}_{n+1,i}(X)$, 
$1^{n,-}_{n+1,i}(X)$ for $i\in\llbracket 1,n\rrbracket$
\end{definition}

\begin{proposition}
Any $n$-divisor has an underlying structure of reflexive cubical set.
\end{proposition}

\begin{definition}
A \textit{connected divisor} is a divisor with underlying connected configuration
such that all its pairs of basic divisors which have adjacent coordinates are adjacents.
\end{definition}

Now we are going to define specific connected $n$-divisors, called \textit{rectangular $n$-divisors}, which are
our cubical pasting diagrams. These rectangular $n$-divisors have another notions of sources and targets
that we call the \textit{pasting-sources} and the \textit{pasting-targets}. Their $j$-sources and $j$-targets
that we have defined above
shall be used to build their sketches, but their $j$-pasting-sources and their $j$-pasting-targets shall be used
to build a structure of cubical $\infty$-category on cubical pasting diagrams.

\begin{itemize}

\item A rectangular $n$-configuration is given by a $n$-configuration $C_n=\{dx^{i}_{k_{i}}/k_1\in K_1,\cdots,k_n\in K_n, 
\forall i\in\llbracket 1,n\rrbracket, K_i\subset\mathbb{N}\text{ is a finite set }\}$ of coordinates 
$dx^{i}_{k_i}\in\mathcal{Z}_n$ such that for all $i\in\llbracket 1,n\rrbracket$, the set $K_i\subset\mathbb{N}$ has
the following form : 
$K_i=\{n^{i}_1,n^{i}_1+1,\cdots,n^{i}_1+l^{i}\}$ where $l^{i}\in\mathbb{N}$. When the set of integers $\mathbb{N}$
is seen as a category which morphisms are given by its order, then another way to describe rectangular $n$-configurations
is to see them as $n$-configurations $C_n$ such that all finite subcategories $K_i\subset\mathbb{N}$ are connected.
It is evident to see that rectangular $n$-configurations are specific connected $n$-configurations. 

\begin{definition}
For all integer $n\in\mathbb{N}$, a \textit{rectangular $n$-divisor} is a connected divisor $X$ which underlying 
$n$-configuration is rectangular. Rectangular divisors are also called \textit{cubical pasting $n$-diagrams}.
\end{definition}

\label{teta-divisors}
The set of rectangular divisors is denoted $\mathbb{C}\text{-}\mathbb{P}\text{ast}$. The full subcategory
$\Theta_{0}\subset{\CS}$
which objects are rectangular divisors is called the \textit{cubical} $\Theta_{0}$ and the Yoneda embedding 
$$\begin{tikzcd}
\Theta_{0}\arrow[rr,"\text{Y}"]&&\CS\\
X\arrow[rr,mapsto]&&\text{hom}_{\CS}(X,-)
\end{tikzcd}$$
shall be useful when we will describe the monad $\mathbb{S}$ of cubical strict $\infty$-categories with connections.

\item Consider a rectangular divisors $X=A_{(k^{1}_{1},\cdots,k^{1}_{n})}dx^{i}_{k^{1}_i}+\cdots+
A_{(k^{l}_{1},\cdots,k^{l}_{n})}dx^{i}_{k^{l}_i}+\cdots+A_{(k^{r}_{1},
\cdots,k^{r}_{n})}dx^{i}_{k^{r}_i}$ with underlying $n$-configuration $C_n=\{dx^{i}_{k_{i}}/k_1\in K_1,\cdots,k_n\in K_n, 
\forall i\in\llbracket 1,n\rrbracket, K_i\subset\mathbb{N}\text{ is a finite set }\}$. For each direction
$j\in\llbracket 1,n\rrbracket$ the subset $K_j=\{n^{j}_1,n^{j}_1+1,\cdots,n^{j}_1+l^{j}\}\subset\mathbb{N}$
has a minimum $n_{j}$ and a maximum $n^{j}_1+l^{j}$. The finite set of basic divisors 
$A_{(k^{l}_{1},\cdots,k^{l}_{n})}dx^{i}_{k^{l}_i}$ of $X$ which $dx^{i}_{k^{l}_i}$ has its $j$-depth
equal to $n_{j}$, gives a new rectangular divisor denoted by $\text{pre-}\sigma^{n}_{n-1,j}(X)$ that 
we call the $j$-\textit{pre-source} of $X$. Also the finite set of basic divisors 
$A_{(k^{l}_{1},\cdots,k^{l}_{n})}dx^{i}_{k^{l}_i}$ of $X$ which $dx^{i}_{k^{l}_i}$ has its $j$-depth
equal to $n_{j}+l^{j}$, gives a new rectangular $n$-divisor denoted by $\text{pre-}\tau^{n}_{n-1,j}(X)$ that 
we call the $j$-\textit{pre-target} of $X$. 

The important fact here is : if $X'$ is another rectangular $n$-divisor such that its $j$-pre-source 
$\text{pre-}\sigma^{n}_{n-1,j}(X')$ is built with basic divisors $A'_{(k^{l'}_{1},\cdots,k^{l'}_{n})}dx^{i}_{k^{l'}_i}$
which are $j$-adjacent to basic divisors $A_{(k^{l}_{1},\cdots,k^{l}_{n})}dx^{i}_{k^{l}_i}$ in 
$\text{pre-}\tau^{n}_{n-1,j}(X)$ and vice-versa : each basic divisor $A_{(k^{l}_{1},\cdots,k^{l}_{n})}dx^{i}_{k^{l}_i}$ in 
$\text{pre-}\tau^{n}_{n-1,j}(X)$ is $j$-adjacent to a basic divisor $A'_{(k^{l'}_{1},\cdots,k^{l'}_{n})}dx^{i}_{k^{l'}_i}$
in $\text{pre-}\sigma^{n}_{n-1,j}(X')$ then the sum $X+X'$ is a rectangular $n$-divisor. Furthermore
$s^{n}_{n-1,j}(\text{pre-}\sigma^{n}_{n-1,j}(X'))$ and $t^{n}_{n-1,j}(\text{pre-}\sigma^{n}_{n-1,j}(X))$
are rectangular $(n-1)$-divisors and are equal.

Thus we put :

\begin{definition}
If $X$ is a rectangular divisor, then its \textit{$j$-pasting source} is :
$$\sigma^{n}_{n-1,j}(X):=s^{n}_{n-1,j}(\text{pre-}\sigma^{n}_{n-1,j}(X))$$ 
and its \textit{$j$-pasting target} is :
$$\tau^{n}_{n-1,j}(X):=t^{n}_{n-1,j}(\text{pre-}\tau^{n}_{n-1,j}(X))$$
\end{definition}

\begin{proposition}
$\sigma^{n}_{n-1,j}(X)$ and $\tau^{n}_{n-1,j}(X)$ are rectangular $(n-1)$-divisors called
respectively the $j$-source and the $j$-target of $X$. These faces are the one adapted to define
a cubical $\infty$-categorical structure on cubical pasting diagrams.
\end{proposition}

\end{itemize}
 
\begin{proposition}
If $X$ is a rectangular $n$-divisor, then $1^{n}_{n+1,j}(X)$ and $1^{n,\gamma}_{n+1,j}(X)$ are rectangular $(n+1)$-divisors.
\end{proposition}

\begin{proof}
Evident : just see that the dilatation of a rectangular $n$-configuration  
\end{proof}

\begin{proposition}
The set $\mathbb{C}\text{-}\mathbb{P}\text{ast}$ of rectangular divisors is equipped with a structure of cubical strict $\infty$-category with connections.
\end{proposition}
\label{rectangle}

\begin{proof}
\begin{itemize}
\item For all integer $n\in\mathbb{N}$, $n$-cells of $\mathbb{C}\text{-}\mathbb{P}\text{ast}$ are 
rectangular $n$-divisors.

\item Consider two rectangular $n$-divisors $X$ and $X'$ such that 
$\tau^{n}_{n-1,j}(X)=\sigma^{n}_{n-1,j}(X')$. Then we define :
$$X\circ^{n}_{n-1,j}X':=X+X'$$

\item If $X$ is a rectangular $n$-divisor, then we defined degeneracies 
$1^{n}_{n+1,j}(X)$ for $j\in\llbracket 1,n+1\rrbracket$ and of $1^{n,-}_{n+1,j}(X)$, 
$1^{n,-}_{n+1,j}(X)$ for $j\in\llbracket 1,n\rrbracket$.
\end{itemize}
\end{proof}

\subsection{Cubical inductive sketches}
 \label{sketches}
 Let $X=A_{(k^{1}_{1},\cdots,k^{1}_{n})}dx^{i}_{k^{1}_i}+\cdots+
A_{(k^{l}_{1},\cdots,k^{l}_{n})}dx^{i}_{k^{l}_i}+\cdots+A_{(k^{r}_{1},
\cdots,k^{r}_{n})}dx^{i}_{k^{r}_i}$ be a divisor. In this section
we associate to it a sketch $\mathcal{E}_X$ which has a cubical
set structure such that $n$-cells of it are themselves sketches build with cocones 
called \textit{basic gluing locus} and \textit{free sources-targets locus}. When $X$ is
a rectangular divisor, its \textit{rectangular sketch} $\mathcal{E}_X$ has like rectangular
divisors, another notions of sources-targets, called again \textit{pasting-sources}
and \textit{pasting-targets}. These last notion of sources-targets allow
to build a cubical strict $\infty$-categorical structure with connections 
on rectangular sketches.

\begin{definition}
If $X$ is a divisor, its $j$-\textit{sketch source} denoted by $\mathcal{E}^{n-1}_{s^{n}_{n-1,j}(X)}$ is the sketch 
obtained as the union of its $j$-gluing locus and its $j$-free sources locus. Its $j$-\textit{sketch target} denoted by 
$\mathcal{E}^{n-1}_{t^{n}_{n-1,j}(X)}$ is the sketch obtained as the union of its $j$-gluing locus and its $j$-free targets locus. 
\end{definition}

Thus these sketches $\mathcal{E}^{n-1}_{s^{n}_{n-1,j}(X)}$ and $\mathcal{E}^{n-1}_{t^{n}_{n-1,j}(X)}$  are characterized
as follow : we consider the $(n-1)$-divisors $s^{n}_{n-1,j}(X)$ and $t^{n}_{n-1,j}(X)$, and with it we select as above 
$j$-gluing locus, $j$-free sources locus and $j$-free targets locus of 
$X$. This characterization is crucial because just by using $j$-sources $s^{n}_{n-1,j}(X)$ and $j$-targets 
$t^{n}_{n-1,j}(X)$ of the divisor $X$ we can identify these sketches $\mathcal{E}^{n-1}_{s^{n}_{n-1,j}(X)}$ and 
$\mathcal{E}^{n-1}_{t^{n}_{n-1,j}(X)}$. We easily see that we can do the same construction for not 
necessary connected $n$-divisors. 

\item Let $X$ be a fixed divisor. As we saw it has a canonical structure of $n$-cubical set, and we denote 
$f_p(X)$ the set of its $p$-faces (which are not necessarily connected divisors). If $x\in f_p(X)$ is a
 $p$-face of $X$ then we denote $f_q(x)$ the set of $q$-faces of $x$. We associate to it the following cubical set 
$\mathcal{E}_X$ of sketches :
\begin{itemize}
\item $\mathcal{E}_X$ has only one $n$-cell still denoted by $X$, which is a \textit{punctual sketch}, i.e which base is reduced to
the point $X$ see as the unique formal point of this base. We denote this singleton by $\mathcal{E}^{n}_X$.

\item Consider the following cubical set :
 
 $$\begin{tikzcd}
 \mathcal{E}^{n}_X \arrow[r, yshift=1.5ex,"s^{n}_{n-1,1}"]
  \arrow[r, yshift=-1.5ex,"t^{n}_{n-1,1}"{below}]
  \arrow[r, yshift=4.5ex,dotted] \arrow[r, yshift=-4.5ex,dotted]
  \arrow[r, yshift=7.5ex,"s^{n}_{n-1,j}"] 
  \arrow[r, yshift=-7.5ex,"t^{n}_{n-1,j}"{below}]
  \arrow[r, yshift=10.5ex,dotted] 
  \arrow[r, yshift=-10.5ex,dotted]      
\arrow[r, yshift=13.5ex,"s^{n}_{n-1,n}"]
\arrow[r, yshift=-13.5ex,"t^{n}_{n-1,n}"{below}]  
& \mathcal{E}^{n-1}_{f_{n-1}(X)} 
 \cdots \mathcal{E}^{4}_{f_{4}(X)} \arrow[r, yshift=1.5ex,"s^{4}_{3,1}"]
  \arrow[r, yshift=-1.5ex,"t^{4}_{3,1}"{below}]
  \arrow[r, yshift=4.5ex,"s^{4}_{3,2}"] \arrow[r, yshift=-4.5ex,"t^{4}_{3,2}"{below}]
  \arrow[r, yshift=7.5ex,"s^{4}_{3,3}"] 
  \arrow[r, yshift=-7.5ex,"t^{4}_{3,3}"{below}]
  \arrow[r, yshift=10.5ex,"s^{4}_{3,4}"] 
  \arrow[r, yshift=-10.5ex,"t^{4}_{3,4}"{below}]      & 
  \mathcal{E}^{3}_{f_{3}(X)} \arrow[r, yshift=1.5ex,"s^{3}_{2,1}"]
  \arrow[r, yshift=-1.5ex,"t^{3}_{2,1}"{below}]
  \arrow[r, yshift=4.5ex,"s^{3}_{2,2}"] 
  \arrow[r, yshift=-4.5ex,"t^{3}_{2,2}"{below}]
  \arrow[r, yshift=7.5ex,"s^{3}_{2,3}"] 
  \arrow[r, yshift=-7.5ex,"t^{3}_{2,3}"{below}]
      & \mathcal{E}^{2}_{f_{2}(X)} 
      \arrow[r, yshift=1.5ex,"s^{2}_{1,1}"]
      \arrow[r, yshift=-1.5ex,"t^{2}_{1,1}"{below}]
      \arrow[r, yshift=4.5ex,"s^{2}_{1,2}"] 
      \arrow[r, yshift=-4.5ex,"t^{2}_{1,2}"{below}]
      & \mathcal{E}^{1}_{f_{1}(X)} \arrow[r, yshift=1.5ex,"s^{1}_{0}"]\arrow[r, yshift=-1.5ex,"t^{1}_{0}"{below}]   
        & \mathcal{E}^{0}_{f_{0}(X)}
 \end{tikzcd}$$
 where $\mathcal{E}^{p}_{f_{p}(X)}$ denotes the sketch which is the set of all sketches $\mathcal{E}^{p}_{x}$ 
 associated to $p$-faces $x\in f_p(X)$ of $X$. We define it by using a simple finite decreasing induction :

 \item as we wrote, the set $\mathcal{E}^{n}_X$ has only the punctual sketch $X$ as element. The maps
 $s^{n}_{n-1,j}$ send $X$ to the sketches $\mathcal{E}^{n-1}_{s^{n}_{n-1,j}(X)}\subset{\mathcal{E}^{n-1}_{f_{n-1}(X)}}$ 
 (the $j$-sketch source of the $n$-divisor $X$), and the maps 
 $t^{n}_{n-1,j}$ send $X$ to the sketches $\mathcal{E}^{n-1}_{t^{n}_{n-1,j}(X)}\subset{\mathcal{E}^{n-1}_{f_{n-1}(X)}}$
 (the $j$-sketch target of the $n$-divisor $X$), for all
 directions $j\in\llbracket 1,n\rrbracket$. 
 
 \item If $x\in f_p(X)$ is a $p$-face of $X$ (it is obtained by zigzags of $(n-p)$-sequences of 
 sources and targets of $X$) then  
 the maps $s^{p}_{p-1,k}$ send the sketch $\mathcal{E}^{p}_x$
 to the sketches $\mathcal{E}^{p-1}_{s^{p}_{p-1,k}(x)}\subset{\mathcal{E}^{p-1}_{f_{p-1}(X)}}$ 
 ($\mathcal{E}^{p-1}_{s^{p}_{p-1,k}(x)}$ is the $k$-sketch source of the $(n-p)$-divisor $x$)
 and the maps
 $t^{p}_{p-1,k}$ send the sketch $\mathcal{E}^{p}_x$
 to the sketches $\mathcal{E}^{p-1}_{t^{p}_{p-1,k}(x)}\subset{\mathcal{E}^{p-1}_{f_{p-1}(X)}}$, 
 ($\mathcal{E}^{p-1}_{t^{p}_{p-1,k}(x)}$ is the $k$-sketch target of the $(n-p)$-divisor $x$),
for all directions $k\in\llbracket 1,p\rrbracket$, where $s^{p}_{p-1,k}(x)\in f_{p-1}(x)$ and 
$t^{p}_{p-1,k}(x)\in f_{p-1}(x)$ are $(p-1)$-faces of $x$.
\end{itemize}
\end{itemize}

\begin{remark}
When we associate a sketch $\mathcal{E}_X$ to a $n$-divisor $X$ we forget coordinates. Here we can see the crucial 
use of coordinates : it allows us to have an accurate description of such sketches and these coordinates are
here used jus as "guides" for building them.   
\end{remark}

Now we are going to give another description of the sources and targets of $\mathcal{E}_X$.

\begin{itemize}
\item Consider the following $(n-p)$-divisor : 
$$x=A_{(k^{1}_{1},\cdots,k^{1}_{p})}dx^{i}_{k^{1}_i}+\cdots+
A_{(k^{l}_{1},\cdots,k^{l}_{p})}dx^{i}_{k^{l}_i}+\cdots+A_{(k^{r}_{1},
\cdots,k^{r}_{p})}dx^{i}_{k^{r(X)}_i}$$
which is a $p$-face of $X$. The $j$-gluing locus $d$ of $x$ are the following cocones $d$ of the sketches $\mathcal{E}^{p-1}_{s^{p}_{p-1,j}(x)}$ and 
$\mathcal{E}^{p-1}_{t^{p}_{p-1,j}(x)}$ :

$$\begin{tikzcd}
A_{(k^{l}_{1},\cdots,k^{l}_{j},\cdots,k^{l}_{p})}dx^{i}_{k^{l}_i} &&A_{(k^{l'}_{1},\cdots,k^{l'}_{j},\cdots,k^{l'}_{p})}dx^{i}_{k^{l'}_i}\\\\
&t^{p}_{p-1,j}(A_{(k^{l}_{1},\cdots,k^{l}_{j},\cdots,k^{l}_{p})}dx^{i}_{k^{l}_i})=
s^{p}_{p-1,j}(A_{(k^{l'}_{1},\cdots,k^{l'}_{j},\cdots,k^{l'}_{p})}dx^{i}_{k^{l'}_i})\arrow[luu,"t^{p}_{p-1,j}"]\arrow[ruu,"s^{p}_{p-1,j}"]
\end{tikzcd}$$

where $A_{(k^{l}_{1},\cdots,k^{l}_{j},\cdots,k^{l}_{p})}dx^{i}_{k^{l}_i}$ and $A_{(k^{l'}_{1},\cdots,k^{l'}_{j},\cdots,k^{l'}_{p})}dx^{i}_{k^{l'}_i}$
are basic divisors of $x$ which are $j$-adjacent and thus are such that : 
 $$(k^{l'}_{1},\cdots,k^{l'}_{j},\cdots,k^{l'}_{p})=(k^{l}_{1},\cdots,k^{l}_{j}+1,\cdots,k^{l}_{p})$$
 and where $t^{p}_{p-1,j}(A_{(k^{l}_{1},\cdots,k^{l}_{j},\cdots,k^{l}_{p})}dx^{i}_{k^{l}_i})$ is itself
 a basic divisor of the divisor $t^{p}_{p-1,j}(x)$ and $s^{p}_{p-1,j}(A_{(k^{l'}_{1},\cdots,k^{l'}_{j},\cdots,k^{l'}_{p})}dx^{i}_{k^{l'}_i})$
 is a basic divisor of the divisor $s^{p}_{p-1,j}(x)$. With this notation we can also describe morphisms of sketches :
 
 $$\begin{tikzcd}
 \mathcal{E}^{p-1}_{s^{p}_{p-1,j}(x)}
 \arrow[r, yshift=1.5ex,"s^{p-1}_{p-2,1}"]
  \arrow[r, yshift=-1.5ex,"t^{p-1}_{p-2,1}"{below}]
  \arrow[r, yshift=4.5ex,dotted] \arrow[r, yshift=-4.5ex,dotted]
  \arrow[r, yshift=7.5ex,"s^{p-1}_{p-2,k}"] 
  \arrow[r, yshift=-7.5ex,"t^{p-1}_{p-2,k}"{below}]
  \arrow[r, yshift=10.5ex,dotted] 
  \arrow[r, yshift=-10.5ex,dotted]      
\arrow[r, yshift=13.5ex,"s^{p-1}_{p-2,p-1}"]
\arrow[r, yshift=-13.5ex,"t^{p-1}_{p-2,p-1}"{below}]  
& \mathcal{E}^{p-2}_{f_{p-2}(s^{p}_{p-1,j}(x))} 
 \end{tikzcd}\qquad\begin{tikzcd}
 \mathcal{E}^{p-1}_{t^{p}_{p-1,j}(x)}
 \arrow[r, yshift=1.5ex,"s^{p-1}_{p-2,1}"]
  \arrow[r, yshift=-1.5ex,"t^{p-1}_{p-2,1}"{below}]
  \arrow[r, yshift=4.5ex,dotted] \arrow[r, yshift=-4.5ex,dotted]
  \arrow[r, yshift=7.5ex,"s^{p-1}_{p-2,k}"] 
  \arrow[r, yshift=-7.5ex,"t^{p-1}_{p-2,k}"{below}]
  \arrow[r, yshift=10.5ex,dotted] 
  \arrow[r, yshift=-10.5ex,dotted]      
\arrow[r, yshift=13.5ex,"s^{p-1}_{p-2,p-1}"]
\arrow[r, yshift=-13.5ex,"t^{p-1}_{p-2,p-1}"{below}]  
& \mathcal{E}^{p-2}_{f_{p-2}(t^{p}_{p-1,j}(x))} 
 \end{tikzcd}$$
 on the $j$-gluing locus of $x$; the description of these morphisms of sketches on the $j$-free sources locus of $x$ and
 on the $j$-free targets locus of $x$ is straightforward when we define their actions only on the $j$-gluing locus.
 
 We describe these morphisms of sketches  
 by defining cocones $s^{p-1}_{p-2,k}(d)$ and $t^{p-1}_{p-2,k}(d)$ as precomposition of the $j$-gluing data $d$ just 
 above 
 $$\begin{tikzcd}
A_{(k^{l}_{1},\cdots,k^{l}_{j},\cdots,k^{l}_{p})}dx^{i}_{k^{l}_i} &&A_{(k^{l'}_{1},\cdots,k^{l'}_{j},\cdots,k^{l'}_{p})}dx^{i}_{k^{l'}_i}\\\\
&A'dx^{i}_{k^{l}_i}\setminus j\arrow[luu,"t^{p}_{p-1,j}"]\arrow[ruu,"s^{p}_{p-1,j}"]
\end{tikzcd}$$
 
 where we denoted $A'=t^{p}_{p-1,j}(A_{(k^{l}_{1},\cdots,k^{l}_{j},\cdots,k^{l}_{p})}dx^{i}_{k^{l}_i})
=s^{p}_{p-1,j}(A_{(k^{l'}_{1},\cdots,k^{l'}_{j},\cdots,k^{l'}_{p})}dx^{i}_{k^{l'}_i})$ and 
$A"_s=s^{p-1}_{p-2,k}(A')$, $A"_t=t^{p-1}_{p-2,k}(A')$, 
which is precomposed with the maps 
 $$\begin{tikzcd}
A"_sdx^{i}_{k"^{l}_i}\setminus(j,k)
\arrow[rr,"s^{p-1}_{p-2,k}"]&&A'dx^{i}_{k^{l}_i}\setminus j\
\end{tikzcd}\qquad
\begin{tikzcd}
A"_tdx^{i}_{k"^{l}_i}\setminus(j,k)
\arrow[rr,"t^{p-1}_{p-2,k}"]&&A'dx^{i}_{k^{l}_i}\setminus j\
\end{tikzcd}$$  

 More precisely the maps $s^{p-1}_{p-2,k}$, $t^{p-1}_{p-2,k}$ send each diagrams $d$ of $\mathcal{E}^{p-1}_{s^{p}_{p-1,j}(x)}$
 and of $\mathcal{E}^{p-1}_{t^{p}_{p-1,j}(x)}$ 
 to diagrams $s^{p-1}_{p-2,k}(d)$, $t^{p-1}_{p-2,k}(d)$ in the sketches $\mathcal{E}^{p-2}_{f_{p-2}(s^{p}_{p-1,j}(x))}$,  
 $\mathcal{E}^{p-2}_{f_{p-2}(t^{p}_{p-1,j}(x))}$, by the precompositions :

$$
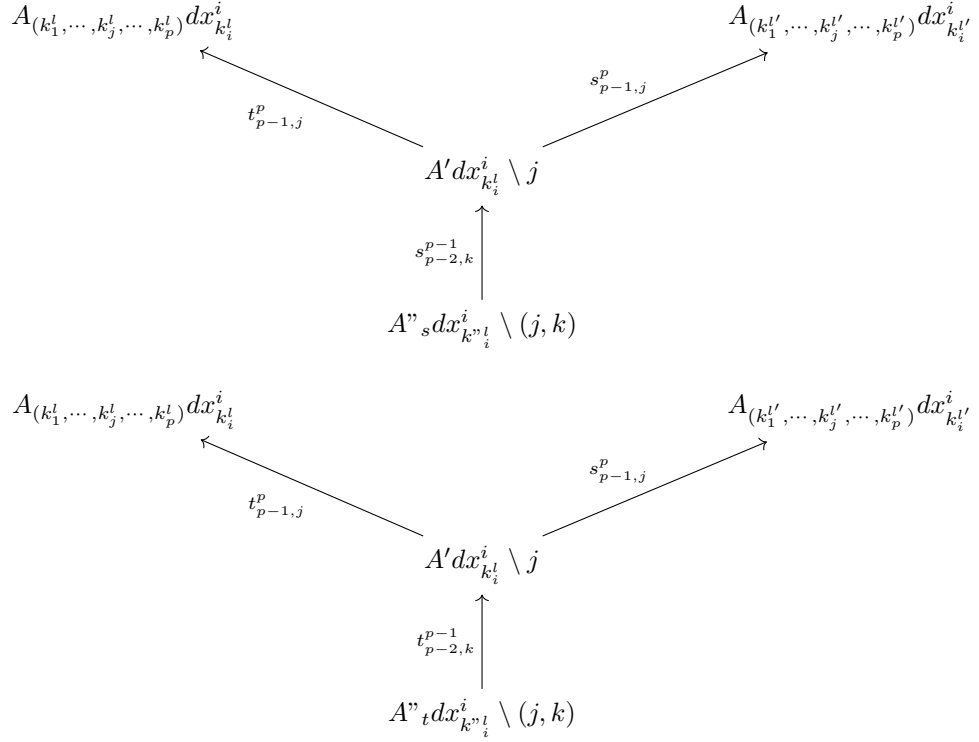
\begin{tikzcd}
&&A_{(k^{l}_{1},\cdots,k^{l}_{j},\cdots,k^{l}_{p})}dx^{i}_{k^{l}_i}&&&&
A_{(k^{l'}_{1},\cdots,k^{l'}_{j},\cdots,k^{l'}_{p})}dx^{i}_{k^{l'}_i}\\\\
&&&&A'dx^{i}_{k^{l}_i}\setminus j\arrow[lluu,"t^{p}_{p-1,j}"]
\arrow[rruu,"s^{p}_{p-1,j}"]\\\\
&&&&A"_sdx^{i}_{k"^{l}_i}\setminus(j,k)\arrow[uu,"s^{p-1}_{p-2,k}"]
\end{tikzcd}$$

$$\begin{tikzcd}
&&A_{(k^{l}_{1},\cdots,k^{l}_{j},\cdots,k^{l}_{p})}dx^{i}_{k^{l}_i}&&&&
A_{(k^{l'}_{1},\cdots,k^{l'}_{j},\cdots,k^{l'}_{p})}dx^{i}_{k^{l'}_i}\\\\
&&&&A'dx^{i}_{k^{l}_i}\setminus j\arrow[lluu,"t^{p}_{p-1,j}"]
\arrow[rruu,"s^{p}_{p-1,j}"]\\\\
&&&&A"_tdx^{i}_{k"^{l}_i}\setminus(j,k)\arrow[uu,"t^{p-1}_{p-2,k}"]
\end{tikzcd}$$

For this description of the maps $s^{p-1}_{p-2,k}$, $t^{p-1}_{p-2,k}$, we use the same arguments as in \ref{cubic-trees} (the one to get sources and targets for cubical trees) :

\end{itemize}

\begin{itemize}

\item When $j=k$ we obtain $s^{p-1}_{p-2,k}(d)$ by using the diagram :
 
$$\begin{tikzcd}
&&A_{(k^{l}_{1},\cdots,k^{l}_{j},\cdots,k^{l}_{p})}dx^{i}_{k^{l}_i}&&&&
A_{(k^{l'}_{1},\cdots,k^{l'}_{j},\cdots,k^{l'}_{p})}dx^{i}_{k^{l'}_i}\\\\
A'dx^{i}_{k^{l}_i}\setminus j\arrow[rruu,"s^{p}_{p-1,j+1}"{left}]&&&&A'dx^{i}_{k^{l}_i}\setminus j\arrow[lluu,"t^{p}_{p-1,j}"]
\arrow[rruu,"s^{p}_{p-1,j}"{left}]&&&&A'dx^{i}_{k^{l}_i}\setminus j\arrow[lluu,"s^{p}_{p-1,j+1}"]\\\\
&&&&A"dx^{i}_{k"^{l}_i}\setminus(j,k)\arrow[lllluu,"t^{p-1}_{p-2,j}"{left}]\arrow[uu,"s^{p-1}_{p-2,k=j}"]\arrow[rrrruu,"s^{p-1}_{p-2,j}"]
\end{tikzcd}$$

where we denote $A"dx^{i}_{k"^{l}_i}\setminus(j,k)=s^{p-1}_{p-2,k}(A'dx^{i}_{k^{l}_i}\setminus j)$.

\begin{remark}
Of course we have also : 
$t^{p-1}_{p-2,j}(A'dx^{i}_{k^{l}_i}\setminus j)=s^{p-1}_{p-2,k}(A'dx^{i}_{k^{l}_i}\setminus j)=s^{p-1}_{p-2,j}(A'dx^{i}_{k^{l}_i}\setminus j)$
but in $t^{p-1}_{p-2,j}(A'dx^{i}_{k^{l}_i}\setminus j)=t^{p-1}_{p-2,j}(A')dx^{i}_{k^{l}_i}\setminus(j,j)$ and 
$s^{p-1}_{p-2,j}(A'dx^{i}_{k^{l}_i}\setminus j)=s^{p-1}_{p-2,j}(A')dx^{i}_{k^{l}_i}\setminus(j,j)$ the basic divisors $A"$, 
$t^{p-1}_{p-2,j}(A')$ and $s^{p-1}_{p-2,j}(A')$ are not necessarily equals. 
\end{remark}

and thus the morphism of sketches $s^{p-1}_{p-2,k}$ sends $d$ to the following diagram $s^{p-1}_{p-2,k}(d)$ 
of the sketches $\mathcal{E}^{p-2}_{f_{p-2}(s^{p}_{p-1,j}(x))}$, $\mathcal{E}^{p-2}_{f_{p-2}(t^{p}_{p-1,j}(x))}$ :

$$\begin{tikzcd}
 A'dx^{i}_{k^{l}_i}\setminus j&&A'dx^{i}_{k^{l}_i}\setminus j\\\\
&A"dx^{i}_{k"^{l}_i}\setminus (j,j)\arrow[luu,"t^{p-1}_{p-2,j}"]\arrow[ruu,"s^{p-1}_{p-2,j}"]
\end{tikzcd}$$

 And we obtain $t^{p-1}_{p-2,k}(d)$ by using the diagram :
 
 $$
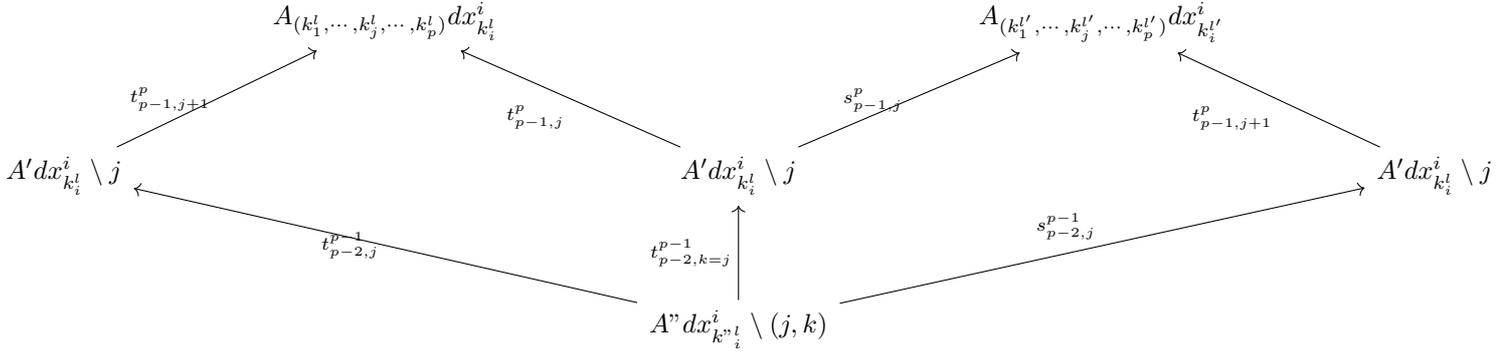
\begin{tikzcd}
&&A_{(k^{l}_{1},\cdots,k^{l}_{j},\cdots,k^{l}_{p})}dx^{i}_{k^{l}_i}&&&&
A_{(k^{l'}_{1},\cdots,k^{l'}_{j},\cdots,k^{l'}_{p})}dx^{i}_{k^{l'}_i}\\\\
A'dx^{i}_{k^{l}_i}\setminus j\arrow[rruu,"t^{p}_{p-1,j+1}"{left}]&&&&A'dx^{i}_{k^{l}_i}\setminus j\arrow[lluu,"t^{p}_{p-1,j}"]
\arrow[rruu,"s^{p}_{p-1,j}"{left}]&&&&A'dx^{i}_{k^{l}_i}\setminus j\arrow[lluu,"t^{p}_{p-1,j+1}"]\\\\
&&&&A"dx^{i}_{k"^{l}_i}\setminus(j,k)\arrow[lllluu,"t^{p-1}_{p-2,j}"{left}]\arrow[uu,"t^{p-1}_{p-2,k=j}"]\arrow[rrrruu,"s^{p-1}_{p-2,j}"]
\end{tikzcd}$$

where we denote $A"dx^{i}_{k"^{l}_i}\setminus(j,k)=t^{p-1}_{p-2,k}(A'dx^{i}_{k^{l}_i}\setminus j)$,
and thus the morphism of sketches $t^{p-1}_{p-2,k}$ sends $d$ to the following diagram $t^{p-1}_{p-2,k}(d)$ 
of the sketches $\mathcal{E}^{p-2}_{f_{p-2}(s^{p}_{p-1,j}(x))}$ and $\mathcal{E}^{p-2}_{f_{p-2}(t^{p}_{p-1,j}(x))}$ :

 $$\begin{tikzcd}
 A'dx^{i}_{k^{l}_i}\setminus j&&A'dx^{i}_{k^{l}_i}\setminus j\\\\
&A"dx^{i}_{k"^{l}_i}\setminus (j,j)\arrow[luu,"t^{p-1}_{p-2,j}"]\arrow[ruu,"s^{p-1}_{p-2,j}"]
\end{tikzcd}$$

\item When $k<j$ then we obtain $s^{p-1}_{p-2,k}(d)$ by using the diagram :
 
$$\begin{tikzcd}
&&A_{(k^{l}_{1},\cdots,k^{l}_{j},\cdots,k^{l}_{p})}dx^{i}_{k^{l}_i}&&&&
A_{(k^{l'}_{1},\cdots,k^{l'}_{j},\cdots,k^{l'}_{p})}dx^{i}_{k^{l'}_i}\\\\
A'dx^{i}_{k^{l}_i}\setminus j\arrow[rruu,"s^{p}_{p-1,k}"{left}]&&&&A'dx^{i}_{k^{l}_i}\setminus j\arrow[lluu,"t^{p}_{p-1,j}"]
\arrow[rruu,"s^{p}_{p-1,j}"{left}]&&&&A'dx^{i}_{k^{l}_i}\setminus j\arrow[lluu,"s^{p}_{p-1,k}"]\\\\
&&&&A"dx^{i}_{k"^{l}_i}\setminus(j,k)\arrow[lllluu,"t^{p-1}_{p-2,j-1}"{left}]\arrow[uu,"s^{p-1}_{p-2,k}"]\arrow[rrrruu,"s^{p-1}_{p-2,j-1}"]
\end{tikzcd}$$ 

where we denote $A"dx^{i}_{k"^{l}_i}\setminus(j,k)=s^{p-1}_{p-2,k}(A'dx^{i}_{k^{l}_i}\setminus j)$

\begin{remark}
Of course we have also : 
$t^{p-1}_{p-2,j-1}(A'dx^{i}_{k^{l}_i}\setminus j)=s^{p-1}_{p-2,k}(A'dx^{i}_{k^{l}_i}\setminus j)=s^{p-1}_{p-2,j-1}(A'dx^{i}_{k^{l}_i}\setminus j)$
but in $t^{p-1}_{p-2,j-1}(A'dx^{i}_{k^{l}_i}\setminus j)=t^{p-1}_{p-2,j-1}(A')dx^{i}_{k^{l}_i}\setminus(j,j-1)$ and 
$s^{p-1}_{p-2,j-1}(A'dx^{i}_{k^{l}_i}\setminus j)=s^{p-1}_{p-2,j-1}(A')dx^{i}_{k^{l}_i}\setminus(j,j-1)$ the basic divisors $A"$, 
$t^{p-1}_{p-2,j-1}(A')$ and $s^{p-1}_{p-2,j-1}(A')$ are not necessarily equals. 
\end{remark}
and thus the morphism of sketches $s^{p-1}_{p-2,k}$ sends $d$ to the following diagram $s^{p-1}_{p-2,k}(d)$ 
of the sketches $\mathcal{E}^{p-2}_{f_{p-2}(s^{p}_{p-1,j}(x))}$ and $\mathcal{E}^{p-2}_{f_{p-2}(t^{p}_{p-1,j}(x))}$ :

$$\begin{tikzcd}
 A'dx^{i}_{k^{l}_i}\setminus j&&A'dx^{i}_{k^{l}_i}\setminus j\\\\
&A"dx^{i}_{k"^{l}_i}\setminus (j,j-1)\arrow[luu,"t^{p-1}_{p-2,j-1}"]\arrow[ruu,"s^{p-1}_{p-2,j-1}"]
\end{tikzcd}$$

And we obtain $t^{p-1}_{p-2,k}(d)$ by using the diagram :
 
$$
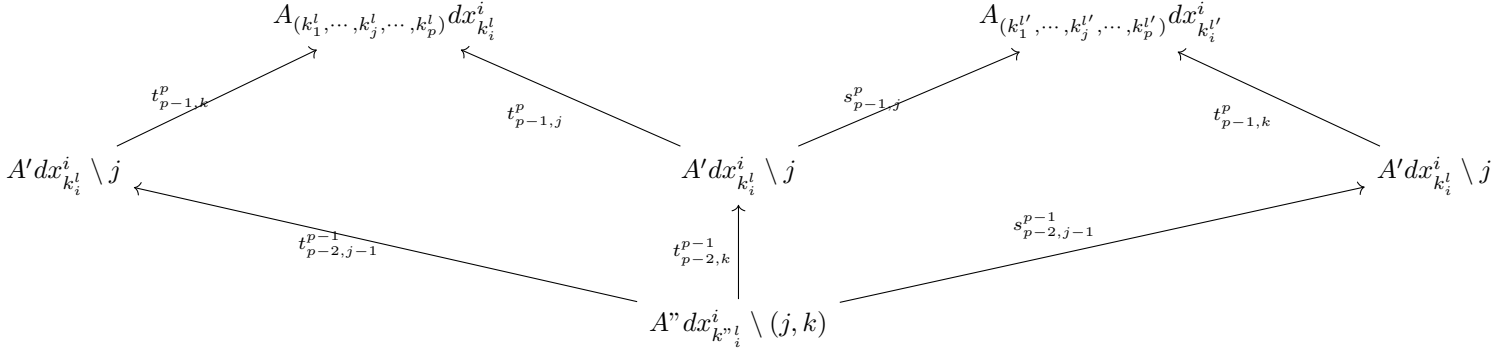
\begin{tikzcd}
&&A_{(k^{l}_{1},\cdots,k^{l}_{j},\cdots,k^{l}_{p})}dx^{i}_{k^{l}_i}&&&&
A_{(k^{l'}_{1},\cdots,k^{l'}_{j},\cdots,k^{l'}_{p})}dx^{i}_{k^{l'}_i}\\\\
A'dx^{i}_{k^{l}_i}\setminus j\arrow[rruu,"t^{p}_{p-1,k}"{left}]&&&&A'dx^{i}_{k^{l}_i}\setminus j\arrow[lluu,"t^{p}_{p-1,j}"]
\arrow[rruu,"s^{p}_{p-1,j}"{left}]&&&&A'dx^{i}_{k^{l}_i}\setminus j\arrow[lluu,"t^{p}_{p-1,k}"]\\\\
&&&&A"dx^{i}_{k"^{l}_i}\setminus(j,k)\arrow[lllluu,"t^{p-1}_{p-2,j-1}"{left}]\arrow[uu,"t^{p-1}_{p-2,k}"]\arrow[rrrruu,"s^{p-1}_{p-2,j-1}"]
\end{tikzcd}$$

where we denote $A"dx^{i}_{k"^{l}_i}\setminus(j,k)=t^{p-1}_{p-2,k}(A'dx^{i}_{k^{l}_i}\setminus j)$,
and thus the morphism of sketches $t^{p-1}_{p-2,k}$ sends $d$ to the following diagram $t^{p-1}_{p-2,k}(d)$ 
of the sketches $\mathcal{E}^{p-2}_{f_{p-2}(s^{p}_{p-1,j}(x))}$ and $\mathcal{E}^{p-2}_{f_{p-2}(t^{p}_{p-1,j}(x))}$ :

$$\begin{tikzcd}
 A'dx^{i}_{k^{l}_i}\setminus j&&A'dx^{i}_{k^{l}_i}\setminus j\\\\
&A"dx^{i}_{k"^{l}_i}\setminus (j,j-1)\arrow[luu,"t^{p-1}_{p-2,j-1}"]\arrow[ruu,"s^{p-1}_{p-2,j-1}"]
\end{tikzcd}$$

\item When $k>j$ then we obtain $s^{p-1}_{p-2,k}(d)$ by using the diagram :
 
$$\begin{tikzcd}
&&A_{(k^{l}_{1},\cdots,k^{l}_{j},\cdots,k^{l}_{p})}dx^{i}_{k^{l}_i}&&&&
A_{(k^{l'}_{1},\cdots,k^{l'}_{j},\cdots,k^{l'}_{p})}dx^{i}_{k^{l'}_i}\\\\
A'dx^{i}_{k^{l}_i}\setminus j\arrow[rruu,"s^{p}_{p-1,k+1}"{left}]&&&&A'dx^{i}_{k^{l}_i}\setminus j\arrow[lluu,"t^{p}_{p-1,j}"]
\arrow[rruu,"s^{p}_{p-1,j}"{left}]&&&&A'dx^{i}_{k^{l}_i}\setminus j\arrow[lluu,"s^{p}_{p-1,k+1}"]\\\\
&&&&A"dx^{i}_{k"^{l}_i}\setminus(j,k)\arrow[lllluu,"t^{p-1}_{p-2,j}"{left}]\arrow[uu,"s^{p-1}_{p-2,k}"]\arrow[rrrruu,"s^{p-1}_{p-2,j}"]
\end{tikzcd}$$

where we denote $A"dx^{i}_{k"^{l}_i}\setminus(j,k)=s^{p-1}_{p-2,k}(A'dx^{i}_{k^{l}_i}\setminus j)$.

\begin{remark}
Of course we have also : 
$t^{p-1}_{p-2,j}(A'dx^{i}_{k^{l}_i}\setminus j)=s^{p-1}_{p-2,k}(A'dx^{i}_{k^{l}_i}\setminus j)=s^{p-1}_{p-2,j}(A'dx^{i}_{k^{l}_i}\setminus j)$
but in $t^{p-1}_{p-2,j}(A'dx^{i}_{k^{l}_i}\setminus j)=t^{p-1}_{p-2,j}(A')dx^{i}_{k^{l}_i}\setminus(j,j)$ and 
$s^{p-1}_{p-2,j}(A'dx^{i}_{k^{l}_i}\setminus j)=s^{p-1}_{p-2,j}(A')dx^{i}_{k^{l}_i}\setminus(j,j)$ the basic divisors $A"$, 
$t^{p-1}_{p-2,j}(A')$ and $s^{p-1}_{p-2,j}(A')$ are not necessarily equals. 
\end{remark}

and thus the morphism of sketches $s^{p-1}_{p-2,k}$ sends $d$ to the following diagram $s^{p-1}_{p-2,k}(d)$ 
of the sketches $\mathcal{E}^{p-2}_{f_{p-2}(s^{p}_{p-1,j}(x))}$ and $\mathcal{E}^{p-2}_{f_{p-2}(t^{p}_{p-1,j}(x))}$ :

$$\begin{tikzcd}
 A'dx^{i}_{k^{l}_i}\setminus j&&A'dx^{i}_{k^{l}_i}\setminus j\\\\
&A"dx^{i}_{k"^{l}_i}\setminus (j,j)\arrow[luu,"t^{p-1}_{p-2,j}"]\arrow[ruu,"s^{p-1}_{p-2,j}"]
\end{tikzcd}$$

And we obtain $t^{p-1}_{p-2,k}(d)$ by using the diagram :
 
 $$
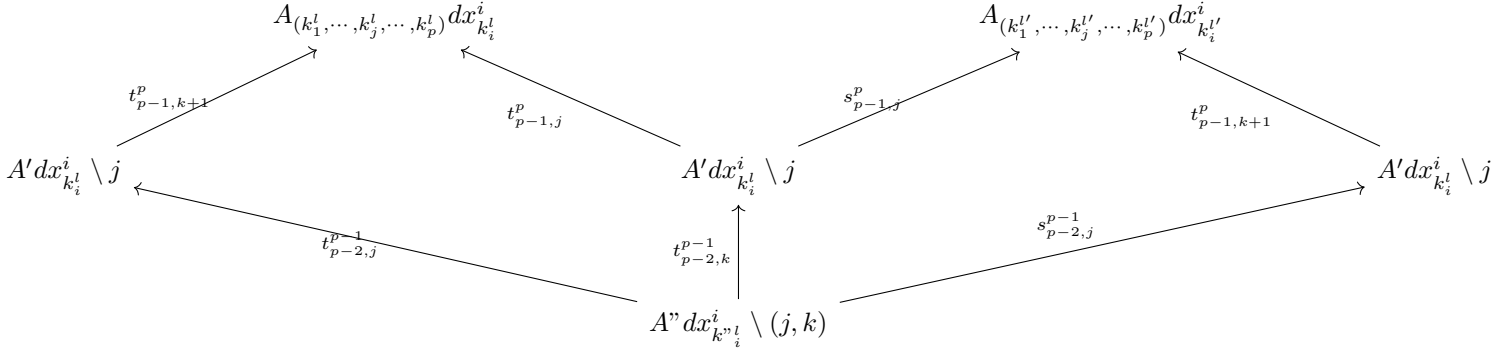
\begin{tikzcd}
&&A_{(k^{l}_{1},\cdots,k^{l}_{j},\cdots,k^{l}_{p})}dx^{i}_{k^{l}_i}&&&&
A_{(k^{l'}_{1},\cdots,k^{l'}_{j},\cdots,k^{l'}_{p})}dx^{i}_{k^{l'}_i}\\\\
A'dx^{i}_{k^{l}_i}\setminus j\arrow[rruu,"t^{p}_{p-1,k+1}"{left}]&&&&A'dx^{i}_{k^{l}_i}\setminus j\arrow[lluu,"t^{p}_{p-1,j}"]
\arrow[rruu,"s^{p}_{p-1,j}"{left}]&&&&A'dx^{i}_{k^{l}_i}\setminus j\arrow[lluu,"t^{p}_{p-1,k+1}"]\\\\
&&&&A"dx^{i}_{k"^{l}_i}\setminus(j,k)\arrow[lllluu,"t^{p-1}_{p-2,j}"{left}]\arrow[uu,"t^{p-1}_{p-2,k}"]\arrow[rrrruu,"s^{p-1}_{p-2,j}"]
\end{tikzcd}$$ 

where we denote $A"dx^{i}_{k"^{l}_i}\setminus(j,k)=t^{p-1}_{p-2,k}(A'dx^{i}_{k^{l}_i}\setminus j)$,
and thus the morphism of sketches $t^{p-1}_{p-2,k}$ sends $d$ to the following diagram $t^{p-1}_{p-2,k}(d)$ 
of the sketches $\mathcal{E}^{p-2}_{f_{p-2}(s^{p}_{p-1,j}(x))}$ and $\mathcal{E}^{p-2}_{f_{p-2}(t^{p}_{p-1,j}(x))}$ :

$$\begin{tikzcd}
 A'dx^{i}_{k^{l}_i}\setminus j&&A'dx^{i}_{k^{l}_i}\setminus j\\\\
&A"dx^{i}_{k"^{l}_i}\setminus (j,j)\arrow[luu,"t^{p-1}_{p-2,j}"]\arrow[ruu,"s^{p-1}_{p-2,j}"]
\end{tikzcd}$$

\item The $1$-faces $x\in f^{1}(X)$ of $X$ are all of the form 
$x=A_1 dx^{1}_{k_1}+A_2 dx^{1}_{k_1+1}+\cdots+A_l dx^{1}_{k_1+l-1}+A_{l+1} dx^{1}_{k_1+l}\cdots+A_r dx^{1}_{k_1+r-1}$ 
where any basic divisor $A_l dx^{1}_{k_1+l-1}$ of $x$ can be $1(1) dx^{1}_{k_1+l-1}$ or $1^{0}_1(1(0)) dx^{1}_{k_1+l-1}$ and
the sketch $\mathcal{E}^{0}_{f_0(X)}$ is a set of diagrams of the form :

$$
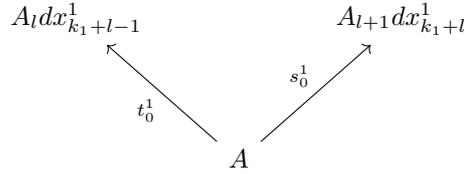
\begin{tikzcd}
 A_l dx^{1}_{k_1+l-1}&&A_{l+1} dx^{1}_{k_1+l}\\\\
&A\arrow[luu,"t^{1}_{0}"]\arrow[ruu,"s^{1}_{0}"]
\end{tikzcd}$$

where $A$ denotes the unique $0$-cell $1(0)$ of the cubical site $\mathbb{C}$.

\item It is interesting to notice that the sketch $\mathcal{E}_X$ can be seen also
as a $n$-cubical object in the category $\mathbb{S}\text{ketch}$ of sketches :
 
 $$
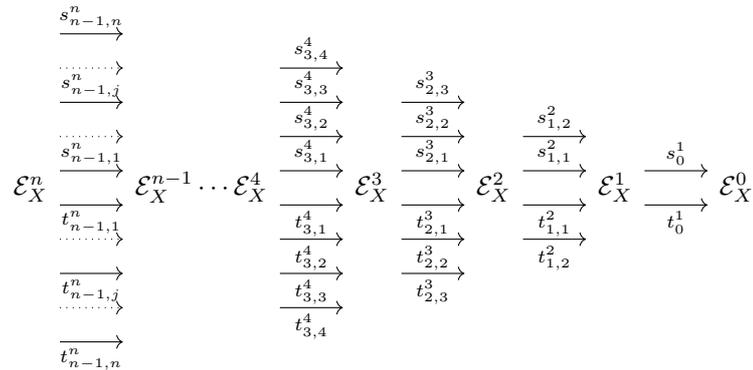
\begin{tikzcd}
 \mathcal{E}^{n}_X \arrow[r, yshift=1.5ex,"s^{n}_{n-1,1}"]
  \arrow[r, yshift=-1.5ex,"t^{n}_{n-1,1}"{below}]
  \arrow[r, yshift=4.5ex,dotted] \arrow[r, yshift=-4.5ex,dotted]
  \arrow[r, yshift=7.5ex,"s^{n}_{n-1,j}"] 
  \arrow[r, yshift=-7.5ex,"t^{n}_{n-1,j}"{below}]
  \arrow[r, yshift=10.5ex,dotted] 
  \arrow[r, yshift=-10.5ex,dotted]      
\arrow[r, yshift=13.5ex,"s^{n}_{n-1,n}"]
\arrow[r, yshift=-13.5ex,"t^{n}_{n-1,n}"{below}]  
& \mathcal{E}^{n-1}_{X} 
 \cdots \mathcal{E}^{4}_{X} \arrow[r, yshift=1.5ex,"s^{4}_{3,1}"]
  \arrow[r, yshift=-1.5ex,"t^{4}_{3,1}"{below}]
  \arrow[r, yshift=4.5ex,"s^{4}_{3,2}"] \arrow[r, yshift=-4.5ex,"t^{4}_{3,2}"{below}]
  \arrow[r, yshift=7.5ex,"s^{4}_{3,3}"] 
  \arrow[r, yshift=-7.5ex,"t^{4}_{3,3}"{below}]
  \arrow[r, yshift=10.5ex,"s^{4}_{3,4}"] 
  \arrow[r, yshift=-10.5ex,"t^{4}_{3,4}"{below}]      & 
  \mathcal{E}^{3}_{X} \arrow[r, yshift=1.5ex,"s^{3}_{2,1}"]
  \arrow[r, yshift=-1.5ex,"t^{3}_{2,1}"{below}]
  \arrow[r, yshift=4.5ex,"s^{3}_{2,2}"] 
  \arrow[r, yshift=-4.5ex,"t^{3}_{2,2}"{below}]
  \arrow[r, yshift=7.5ex,"s^{3}_{2,3}"] 
  \arrow[r, yshift=-7.5ex,"t^{3}_{2,3}"{below}]
      & \mathcal{E}^{2}_{X} 
      \arrow[r, yshift=1.5ex,"s^{2}_{1,1}"]
      \arrow[r, yshift=-1.5ex,"t^{2}_{1,1}"{below}]
      \arrow[r, yshift=4.5ex,"s^{2}_{1,2}"] 
      \arrow[r, yshift=-4.5ex,"t^{2}_{1,2}"{below}]
      & \mathcal{E}^{1}_{X} \arrow[r, yshift=1.5ex,"s^{1}_{0}"]\arrow[r, yshift=-1.5ex,"t^{1}_{0}"{below}]   
        & \mathcal{E}^{0}_{X}
 \end{tikzcd}$$
 
 where we put : $\mathcal{E}^{p}_X:=\underset{x\in f_p(X)}\bigcup\mathcal{E}^{p}_x$.
 
\item
\begin{definition}
If $\mathcal{E}_X$ is the sketch associated to a divisor $X$ then it has a straightforward structure of
cubical set given by different faces of $X$, and also it has a straightforward structure of reflexive
cubical set given by : 
$$1^{n}_{n+1,j}(\mathcal{E}_{X}):=\mathcal{E}_{1^{n}_{n+1,j}(X)}\text{ for }j\in\llbracket 1,n+1\rrbracket,
\text{ and }1^{n,\gamma}_{n+1,j}(\mathcal{E}_{X}):=\mathcal{E}_{1^{n,\gamma}_{n+1,j}(X)}\text{ for }j\in\llbracket 1,n\rrbracket$$ 
\end{definition} 

\item 

\begin{definition}
A \textit{rectangular $n$-sketch} are the one of the form $\mathcal{E}_X$ where $X$ is a rectangular $n$-divisor.
\end{definition}

Because rectangular divisors have two notions of sources and targets : 
\begin{itemize}
\item the $j$-sources and $j$-targets
($j\in\mathbb{N}$ is a direction)
which are useful to build associated sketches, 

\item the $j$-pasting sources and $j$-pasting targets 
which are useful to put a cubical strict $\infty$-categorical structure with connections on rectangular
divisors 
\end{itemize}

their associated rectangular sketches inherit also two notions of sources-targets :

\begin{definition}
If $X$ is a rectangular $n$-divisor and $\mathcal{E}_{X}$ is its associated sketch, then we define :
$$\sigma^{n}_{n-1,j}(\mathcal{E}_{X}):=\mathcal{E}_{\sigma^{n}_{n-1,j}(X)}\text{ and }
\tau^{n}_{n-1,j}(\mathcal{E}_{X}):=\mathcal{E}_{\tau^{n}_{n-1,j}(X)}$$
\end{definition}

\begin{proposition}
The set of sketches associated to rectangular divisors is equipped with a structure of cubical strict $\infty$-category with connections
that we denote by $\mathbb{C}\text{-}\mathbb{P'}\text{ast}$
\end{proposition}
\label{rectangle-sketch}

\begin{proof}
\begin{itemize}
\item For all integer $n\in\mathbb{N}$, $n$-cells of $\mathbb{C}\text{-}\mathbb{P'}\text{ast}$ are 
sketches $\mathcal{E}_X$ where $X$ is a rectangular $n$-divisor.

\item Consider two sketches $\mathcal{E}_X$ and $\mathcal{E}_{X'}$ such that 
$\tau^{n}_{n-1,j}(\mathcal{E}_X)=\sigma^{n}_{n-1,j}(\mathcal{E}_{X'})$. Then we define :
$$\mathcal{E}_X\circ^{n}_{n-1,j}\mathcal{E}_{X'}:=\mathcal{E}_{X\circ^{n}_{n-1,j}X'}$$

\item If $X$ is a rectangular $n$-divisor and $\mathcal{E}_{X}$ is its associated sketch, then we defined above degeneracies 
$1^{n}_{n+1,j}(\mathcal{E}_{X})$ for $j\in\llbracket 1,n+1\rrbracket$ and
$1^{n,-}_{n+1,j}(\mathcal{E}_{X})$, 
$1^{n,+}_{n+1,j}(\mathcal{E}_{X})$ for $j\in\llbracket 1,n\rrbracket$.
\end{itemize}
\end{proof}

We thus have another description of the cubical $\Theta_0$ which is the full subcategory $\Theta_0\subset{\CS}$ which
objects are rectangular sketches.

\end{itemize}

\subsection{The monad of cubical strict $\infty$-categories with connections}
\label{proof}

Consider a rectangular $n$-divisor 
$X=A_{(k^{1}_{1},\cdots,k^{1}_{n})}dx^{i}_{k^{1}_i}+\cdots+
A_{(k^{l}_{1},\cdots,k^{l}_{n})}dx^{i}_{k^{l}_i}+\cdots+A_{(k^{r}_{1},
\cdots,k^{r}_{n})}dx^{i}_{k^{r(X)}_i}$ and a cubical set $C\in\CS$. 

A \textit{decoration} of $X$ by cells of 
$C$ is given by a $C$-\textit{decorated rectangular $n$-divisor} :
$$\langle X,C\rangle=c_{(k^{1}_{1},\cdots,k^{1}_{n})}dx^{i}_{k^{1}_i}+\cdots+
c_{(k^{l}_{1},\cdots,k^{l}_{n})}dx^{i}_{k^{l}_i}+\cdots+c_{(k^{r}_{1},
\cdots,k^{r}_{n})}dx^{i}_{k^{r(X)}_i}$$
i.e a filling of $X$ with cells $c_{(k^{l}_{1},\cdots,k^{l}_{n})}$ of $\mathbb{R}(C)$ (i.e we substitute 
the $1(q)$'s in each basic divisors of $X$ which are formally degenerate or not with the $q$-cells 
of $C$) 
such that for all directions $j\in\llbracket 1,n\rrbracket$
if  $(A_{(k^{l}_{1},\cdots,k^{l}_{n})}dx^{i}_{k^{l}_i},
A_{(k^{l'}_{1},\cdots,k^{l'}_{n})}dx^{i}_{k^{l'}_i})$ are $j$-gluing datas for $X$, i.e are such that 
$t^{n}_{n-1,j}(A_{(k^{l}_{1},\cdots,k^{l}_{n})}dx^{i}_{k^{l}_i})=s^{n}_{n-1,j}(A_{(k^{l'}_{1},\cdots,k^{l'}_{n})}dx^{i}_{k^{l'}_i})$
then :
$$t^{n}_{n-1,j}(c_{(k^{l}_{1},\cdots,k^{l}_{n})}dx^{i}_{k^{l}_i})=s^{n}_{n-1,j}(c_{(k^{l'}_{1},\cdots,k^{l'}_{n})}dx^{i}_{k^{l'}_i})$$
The set of decorations of $X$ by cells of $C$ is denoted $\mathbb{D}\text{ecor}(X,C)$. The colimit 
$\text{colim}\mathcal{E}_{\langle X,C\rangle}$ is taken in $\E$ and gives an $n$-cell of the free cubical strict $\infty$-category 
$S(C)$ with connections. 

We have another description of these $n$-cells $\text{colim}\mathcal{E}_{\langle X,C\rangle}$ 
by using gluing of representables. But first let us define what are \textit{cubical sums} : consider a functor 
\begin{tikzcd}
\mathbb{C}\arrow[rr,"F"]&&\mathcal{C}
\end{tikzcd}
where we denote $F(1(n))=I^{n}$ and $F(s^{n}_{n-1,j})=\bf{s}^{n}_{n-1,j}$, and $F(t^{n}_{n-1,j})=\bf{t}^{n}_{n-1,j}$. 
The $F$-\textit{decorated rectangular $n$-divisor} :
$$\langle X,F\rangle=c_{(k^{1}_{1},\cdots,k^{1}_{n})}dx^{i}_{k^{1}_i}+\cdots+
c_{(k^{l}_{1},\cdots,k^{l}_{n})}dx^{i}_{k^{l}_i}+\cdots+c_{(k^{r}_{1},
\cdots,k^{r}_{n})}dx^{i}_{k^{r(X)}_i}$$
is a filling of $X$ by the $n$-cells $I^n$ of $\mathcal{C}$ in the sense that in each occurrence of the $1(n)$'s
in the degenerate boxes of $X$, we substitute $1(n)$ by $I^n$. Here this is important to notice that
the expressions $c_{(k^{l}_{1},\cdots,k^{l}_{n})}dx^{i}_{k^{l}_i}$ that we obtain are formal degenerate
terms build with the objects $I^n$ of $\mathcal{C}$. Also there is only a unique $F$-decoration $\langle X,F\rangle$ of $F$
for each rectangular $n$-divisor $X$.

We associate to $\langle X,F\rangle$
the sketch $\mathcal{E}_{\langle X,F\rangle}$ as in \ref{sketches}, and in fact this is just the realization of the sketch $\mathcal{E}_X$ by $F$, i.e all formal cosources-cotargets $s^{p}_{p-1,j}, t^{p}_{p-1,j}$ of $\mathcal{E}_X$ are sent to $\bf s^{p}_{p-1,j}, t^{p}_{p-1,j}$ 
by $F$, and all formal degenerate terms $c_{(k^{l}_{1},\cdots,k^{l}_{n})}dx^{i}_{k^{l}_i}$ in $\langle X,F\rangle$
must be well realized in $\mathcal{C}$. At this point this is interesting to notice that such realizations are possible in 
any category $\mathcal{C}$ of presheaves, because in any category of presheaves we can build degenerates
terms (like the one of the left adjoint of the forgetful functor $U$ in \ref{degenerate-boxes}).

If the colimit $\text{colim} \mathcal{E}_{\langle X,F\rangle}$ exists in $\mathcal{C}$ then we say that it has 
the cubical sum associated to the $F$-decoration $\langle X,F\rangle$, or it has $X$-\textit{cubical sum} for short. 
If the $X$-cubical sum exists in $\mathcal{C}$
for all such decorations $\langle X,F\rangle$ and for all rectangular $n$-divisor $X$ then we say that $F$ is a 
\textit{cubical extension} or has all \textit{cubical sums}. 
A morphism of cubical extensions is given by a commutative triangle

$$\begin{tikzcd}
&&\mathcal{C}\arrow[dd,"f"]\\
\mathbb{C}\arrow[rru,"F"]\arrow[rrd,"F'"{left}]\\
&&\mathcal{C}'
\end{tikzcd}$$
such that the functor $f$ preserves cubical sums, i.e we have 
$f(\text{colim} \mathcal{E}_{\langle X,F\rangle})=\text{colim} \mathcal{E}_{\langle X,F'\rangle}$
for each $F$-decoration $\langle X,F\rangle$ of each rectangular divisor $X$. We denote by
$\mathbb{C}\text{-}\mathbb{E}\text{xt}$ the category of cubical extensions. It is easy to see that the
functor $i$ :
$$\begin{tikzcd}
\mathbb{C}\arrow[rr,"i"]&&\Theta_0
\end{tikzcd}$$
which sends each objects $1(n)$ of $\mathbb{C}$ to the basic box $1(n)dx^{i}_{k_i}$ is an
initial object of $\mathbb{C}\text{-}\mathbb{E}\text{xt}$, 
such that the unique map is given by the functor 
$$\begin{tikzcd}
&&\Theta_0\arrow[dd,"\text{colim}\mathcal{E}_{\langle -,F\rangle}"]\\
\mathbb{C}\arrow[rru,"i"]\arrow[rrd,"F"{left}]\\
&&\mathcal{C}
\end{tikzcd}$$

where $\Theta_0$ is here seen as the category of rectangular $n$-divisors (for
each $n\in\mathbb{N}$) \ref{teta-divisors}, and thus $\Theta_0$ inherits a universal property.

The Yoneda embedding :
$$\begin{tikzcd}
\mathbb{C}\arrow[rr,"\text{Y}"]&&\CS\\
1(n)\arrow[rr,mapsto]&&\text{hom}_{\CS}(1(n),-)
\end{tikzcd}$$
is such cubical extension because $\CS$ has all small colimits.
The full image $\text{Y}(\mathbb{C})$ of $\text{Y}$ is a cubical set builds with representables of the
pre-cubical site $\mathbb{C}$, and we put $\text{Y}(1(n))=\text{hom}_{\CS}(1(n),-)$ for all integer
$n\in\mathbb{N}$, and $\text{Y}(s^{n}_{n-1,j})=\bf{s}^{n}_{n-1,j}$, and $\text{Y}(t^{n}_{n-1,j})=\bf{t}^{n}_{n-1,j}$. 
A decoration $\langle X,\text{Y}(\mathbb{C})\rangle$ of $X$ by cells of $\text{Y}(\mathbb{C})$ has here the
same meaning as a $\text{Y}$-decorated rectangular $n$-divisor 
$\langle X,\text{Y}\rangle$, i.e it is given by the $\text{Y}(\mathbb{C})$-decorated rectangular $n$-divisor
$$\langle X,\text{Y}(\mathbb{C})\rangle=\langle X,\text{Y}\rangle=c_{(k^{1}_{1},\cdots,k^{1}_{n})}dx^{i}_{k^{1}_i}+\cdots+
c_{(k^{l}_{1},\cdots,k^{l}_{n})}dx^{i}_{k^{l}_i}+\cdots+c_{(k^{r}_{1},
\cdots,k^{r}_{n})}dx^{i}_{k^{r(X)}_i}$$
i.e we substitute each $1(p)$'s of $X$ with the representable $\text{hom}_{\CS}(1(p),-)$ for $p\leq n$. 

The colimit $\text{colim} \mathcal{E}_{\langle X,\text{Y}\rangle}$ (in $\CS$) is in fact a cubical set builds by gluing representables
along their sources-targets : more precisely we have for all directions $j\in\llbracket 1,n\rrbracket$, 
$s^{n}_{n-1,j}(\text{colim} \mathcal{E}_{\langle X,\text{Y}\rangle})=\text{colim} \mathcal{E}_{\langle s^{n}_{n-1,j}(X),\text{Y}\rangle}$, 
$t^{n}_{n-1,j}(\text{colim} \mathcal{E}_{\langle X,\text{Y}\rangle})=\text{colim} \mathcal{E}_{\langle t^{n}_{n-1,j}(X),\text{Y}\rangle}$, 
and any zigzag of sources-targets of $\text{colim} \mathcal{E}_{\langle X,\text{Y}\rangle}$ is equal to the colimit
$\text{colim} \mathcal{E}_{\langle x,\text{Y}\rangle}$ where $x$ is the face of $X$ obtained by this zigzag.
Now we are ready to describe the monad $\mathbb{S}=(S,\lambda,\mu)$ of cubical strict $\infty$-categories 
with connections : as we wrote in \cite{cam-cubique} the forgetful functor : 

$$\begin{tikzcd}
\CC\arrow[rr,"U"]&&\left[\mathbb{C}^{op},\mathbb{S}\text{ets}\right]
\end{tikzcd}$$

which sends cubical strict $\infty$-categories with connections to
cubical sets is right adjoint and its induced monad is written $\mathbb{S}=(S,\lambda,\mu)$ where 
\begin{tikzcd}
1_{\CS}\arrow[rr,"\lambda"]&&S
\end{tikzcd}
 is its unit and 
\begin{tikzcd}
S^2\arrow[rr,"\mu"]&&S
\end{tikzcd}
is its multiplication. 

If $C\in\CS$ is a cubical set, then we put :

$$S(C):=\underset{X\in\Theta_0}\bigcup{\text{colim}\mathcal{E}_{\langle X,C\rangle}}=\underset{X\in\Theta_0}
\bigcup{\text{hom}_{\CS}(\text{colim} \mathcal{E}_{\langle X,\text{Y}\rangle},C)}$$

This description shows immediately that $S$ preserved fiber products and $S(C)$ is a cubical strict 
$\infty$-categories with connections.

The unit $\lambda$ of $\mathbb{S}$ is given by the map :
$$\begin{tikzcd}
C\arrow[rr,"\lambda"]&&S(C)\\
c\arrow[rr,mapsto]&&c
\end{tikzcd}$$

The multiplication $\mu$ of $\mathbb{S}$ :
$$\begin{tikzcd}
S^2(C)\arrow[rr,"\mu"]&&S(C)
\end{tikzcd}$$
is more subtle and need more work. Consider a decoration of $X$ by cells of 
$\mathbb{C}\text{-}\mathbb{P}\text{ast}$ i.e we start with a $\mathbb{C}\text{-}\mathbb{P}\text{ast}$-decorated rectangular $n$-divisor :
$$\langle X,\mathbb{C}\text{-}\mathbb{P}\text{ast}\rangle=X_{(k^{1}_{1},\cdots,k^{1}_{n})}dx^{i}_{k^{1}_i}+\cdots+
X_{(k^{l}_{1},\cdots,k^{l}_{n})}dx^{i}_{k^{l}_i}+\cdots+X_{(k^{r}_{1},
\cdots,k^{r}_{n})}dx^{i}_{k^{r(X)}_i}$$
i.e a filling of $X$ with cells $X_{(k^{l}_{1},\cdots,k^{l}_{n})}$ of $\mathbb{C}\text{-}\mathbb{P}\text{ast}$ 
such that for all directions $j\in\llbracket 1,n\rrbracket$
if  $(A_{(k^{l}_{1},\cdots,k^{l}_{n})}dx^{i}_{k^{l}_i},
A_{(k^{l'}_{1},\cdots,k^{l'}_{n})}dx^{i}_{k^{l'}_i})$ are $j$-gluing datas for $X$, 
then :
$$\tau^{n}_{n-1,j}(X_{(k^{l}_{1},\cdots,k^{l}_{n})}dx^{i}_{k^{l}_i})=
\sigma^{n}_{n-1,j}(X_{(k^{l'}_{1},\cdots,k^{l'}_{n})}dx^{i}_{k^{l'}_i})$$
Here sources $\sigma^{n}_{n-1,j}$ and targets $\tau^{n}_{n-1,j}$ are the pasting-sources
and the pasting-targets for rectangular divisors.
We obtain then a new rectangular $n$-divisor denoted $X_{\langle X,\mathbb{C}\text{-}\mathbb{P}\text{ast}\rangle}$ that
we obtain by reindexing all coordinates of basic divisors inside each 
rectangular $n$-divisors $X_{(k^{l}_{1},\cdots,k^{l}_{n})}$. An important fact is
the sketch $\mathcal{E}_{X_{\langle X,\mathbb{C}\text{-}\mathbb{P}\text{ast}\rangle}}$
is obtained by using the realization of the $s^{n}_{n-1,j}$ to the $\sigma^{n}_{n-1,j}$ and
the realization of the $t^{n}_{n-1,j}$ to the $\tau^{n}_{n-1,j}$ : this is here that we see the interplay
between sources-targets of a rectangular sketch and pasting-sources and pasting-targets
between rectangular sketches. This 
provides on $\mathbb{C}\text{-}\mathbb{P}\text{ast}$ an operation that
we call \textit{the substitution} of $\mathbb{C}\text{-}\mathbb{P}\text{ast}$.
The multiplication $\mu$ of $\mathbb{S}$ :
$$\begin{tikzcd}
S^2(C)\arrow[rr,"\mu"]&&S(C)
\end{tikzcd}$$
is then given by :

$$\begin{tikzcd}
\underset{X\in\Theta_0}
\bigcup{\text{hom}_{\CS}(\text{colim} \mathcal{E}_{\langle X,\text{Y}\rangle},S(C))}\arrow[rr,"\mu(C)"]&&\underset{X\in\Theta_0}
\bigcup{\text{hom}_{\CS}(\text{colim} \mathcal{E}_{\langle X,\text{Y}\rangle},C)}
\end{tikzcd}$$
where

$$\underset{X\in\Theta_0}
\bigcup{\text{hom}_{\CS}(\text{colim} \mathcal{E}_{\langle X,\text{Y}\rangle},S(C))}$$
is given by :

$$\underset{X\in\Theta_0}
\bigcup{\text{hom}_{\CS}\Big(\text{colim} \mathcal{E}_{\langle X,\text{Y}\rangle},\underset{X\in\Theta_0}
\bigcup{\text{hom}_{\CS}(\text{colim} \mathcal{E}_{\langle X,\text{Y}\rangle},C)}\Big)}$$

An element of $S^{2}(C)$ is described as follow : it underlies a decoration of $X$ by $C$-decorated rectangular $n$-divisors
$\langle X_{(k^{l}_{1},\cdots,k^{l}_{n})},C\rangle$, and we write such decoration :

$$\langle X,\mathbb{C}\text{-}\mathbb{P}\text{ast},C\rangle=
\langle X_{(k^{1}_{1},\cdots,k^{1}_{n})},C\rangle dx^{i}_{k^{1}_i}+\cdots+
\langle X_{(k^{l}_{1},\cdots,k^{l}_{n})},C\rangle dx^{i}_{k^{l}_i}+\cdots+
\langle X_{(k^{r}_{1},\cdots,k^{r}_{n})},C\rangle dx^{i}_{k^{r(X)}_i}$$

which itself underlies the $\mathbb{C}\text{-}\mathbb{P}\text{ast}$-decorated rectangular $n$-divisor :

$$\langle X,\mathbb{C}\text{-}\mathbb{P}\text{ast}\rangle=X_{(k^{1}_{1},\cdots,k^{1}_{n})}dx^{i}_{k^{1}_i}+\cdots+
X_{(k^{l}_{1},\cdots,k^{l}_{n})}dx^{i}_{k^{l}_i}+\cdots+X_{(k^{r}_{1},
\cdots,k^{r}_{n})}dx^{i}_{k^{r(X)}_i}$$

In the decoration $\langle X,\mathbb{C}\text{-}\mathbb{P}\text{ast},C\rangle$
we substitute each $\langle X_{(k^{l}_{1},\cdots,k^{l}_{n})},C\rangle$
by the $n$-cells $\text{colim}\mathcal{E}_{\langle X_{(k^{l}_{1},\cdots,k^{l}_{n})},C\rangle}$ of $S(C)$, 
which are also described as maps between the gluing of representables 
$\text{colim} \mathcal{E}_{\langle X_{(k^{l}_{1},\cdots,k^{l}_{n})},\text{Y}\rangle}$
to $C$. Thus we obtain a decoration of a rectangular $n$-divisor
$X$ by $n$-cells of $S(C)$ that we write :

$$[X,\mathbb{C}\text{-}\mathbb{P}\text{ast},C]=
\text{colim}(\langle X_{(k^{1}_{1},\cdots,k^{1}_{n})},C\rangle) dx^{i}_{k^{1}_i}+\cdots+
\text{colim}(\langle X_{(k^{l}_{1},\cdots,k^{l}_{n})},C\rangle dx^{i}_{k^{l}_i})+\cdots+
\text{colim}(\langle X_{(k^{r}_{1},\cdots,k^{r}_{n})},C\rangle dx^{i}_{k^{r(X)}_i})$$

An element of $S^2(C)$ is then given by $\text{colim}\mathcal{E}_{[X,\mathbb{C}\text{-}\mathbb{P}\text{ast},C]}$.
Thus elements of $S^2(C)$ are described by a colimit of colimits of $n$-cells in $S(C)$. But colimit of colimits of
cocones is the colimit of a cocone obtained by gluing all cocones together, and we can already guess that 
the multiplication $\mu(C)$ uses this simple fact : we use $X_{\langle X,\mathbb{C}\text{-}\mathbb{P}\text{ast}\rangle}$, 
the substitution associated to the $\mathbb{C}\text{-}\mathbb{P}\text{ast}$-decorated rectangular $n$-divisor :

$$\langle X,\mathbb{C}\text{-}\mathbb{P}\text{ast}\rangle=X_{(k^{1}_{1},\cdots,k^{1}_{n})}dx^{i}_{k^{1}_i}+\cdots+
X_{(k^{l}_{1},\cdots,k^{l}_{n})}dx^{i}_{k^{l}_i}+\cdots+X_{(k^{r}_{1},
\cdots,k^{r}_{n})}dx^{i}_{k^{r(X)}_i}$$

in order to glue all cocones of the sketches $\mathcal{E}_X$ and $\mathcal{E}_{X_{(k^{l}_{1},\cdots,k^{l}_{n})}}$
($l\in\llbracket 1,r(X)\rrbracket$) and to obtain cocones of $\mathcal{E}_{X_{\langle X,\mathbb{C}\text{-}\mathbb{P}\text{ast}\rangle}}$.

Thus we have the following definition of $\mu(C)$ :

$$\begin{tikzcd}
S^2(C)\arrow[rr,"\mu(C)"]&&S(C)
\end{tikzcd}$$

$\mu(C)$ sends $\text{colim}\mathcal{E}_{[X,\mathbb{C}\text{-}\mathbb{P}\text{ast},C]}$
to $\text{colim}\mathcal{E}_{\langle X_{\langle X,\mathbb{C}\text{-}\mathbb{P}\text{ast}\rangle},C\rangle}$
which is the same thing as to say that it sends $\text{colim}\mathcal{E}_{[X,\mathbb{C}\text{-}\mathbb{P}\text{ast},C]}$
to a map between the gluing of representables 
$\text{colim}\mathcal{E}_{\langle X_{\langle X,\mathbb{C}\text{-}\mathbb{P}\text{ast}\rangle},\text{Y}\rangle}$ to
$C$.

\begin{theorem}
The monad $\mathbb{S}=(S,\lambda,\mu)$ acting on $\CS$ which algebras are cubical strict $\infty$-categories with connections
(described in \cite{cam-cubique,camark-cub-1}) is cartesian.
\end{theorem}
\label{monad-S}

\begin{proof}
The description of the monad $\mathbb{S}=(S,\lambda,\mu)$ above shows that 
its underlying endofunctor $S$ does preserve fibred products.
We are going to prove that the multiplication $\mu$ is cartesian, i.e we are going to prove
that if $C\in\CS$ is a cubical set then the commutative diagram :

$$\begin{tikzcd}
S^{2}(C)\arrow[dd,"\mu(C)"{left}]\arrow[rr,"S^{2}(!)"]&&S^{2}(1)\arrow[dd,"\mu(1)"]\\\\
S(C)\arrow[rr,"S(!)"{below}]&&S(1)
\end{tikzcd}$$

is a cartesian square. Consider the commutative diagram in $\CS$ :
$$\begin{tikzcd}
C'\arrow[dd,"f"{left}]\arrow[rr,"g"]&&S^{2}(1)\arrow[dd,"\mu(1)"]\\\\
S(C)\arrow[rr,"S(!)"{below}]&&S(1)
\end{tikzcd}$$

Thus if $x$ is an $n$-cell of $C'$ then $f(x)=\text{colim}\mathcal{E}_{\langle X',C\rangle}$ where 
$X'$ is a rectangular $n$-divisor, and $S(!)(f(x))=S(!)(\text{colim}\mathcal{E}_{\langle X',C\rangle})
=\text{colim}\mathcal{E}_{\langle X',1\rangle}$, and $g(x)=\text{colim}\mathcal{E}_{[X,\mathbb{C}\text{-}\mathbb{P}\text{ast},1]}$
thus $\mu(1)(g(x))=\mu(1)(\text{colim}\mathcal{E}_{[X,\mathbb{C}\text{-}\mathbb{P}\text{ast},1]})=
\text{colim}\mathcal{E}_{\langle X_{\langle X,\mathbb{C}\text{-}\mathbb{P}\text{ast}\rangle},1\rangle}$. But
the commutativity gives $S(!)(f(x))=\mu(1)(g(x))$ thus this commutativity gives $X'=X_{\langle X,\mathbb{C}\text{-}\mathbb{P}\text{ast}\rangle}$
and $f(x)=\text{colim}\mathcal{E}_{\langle X_{\langle X,\mathbb{C}\text{-}\mathbb{P}\text{ast}\rangle},C\rangle}$. It is then easy to see that we get a unique map $l$ :  

$$\begin{tikzcd}
C'\arrow[rrdd,"l",near end, dotted]\arrow[rrdddd,"f"{left}]\arrow[rrrrdd,"g"]\\\\
&&S^{2}(C)\arrow[dd,"\mu(C)"{left},near start]\arrow[rr,"S^{2}(!)"]&&S^{2}(1)\arrow[dd,"\mu(1)"]\\\\
&&S(C)\arrow[rr,"S(!)"{below}]&&S(1)
\end{tikzcd}$$
defined by 
$l(x)=\text{colim}\mathcal{E}_{[X,\mathbb{C}\text{-}\mathbb{P}\text{ast},C]}$, and is such that
$\mu(C)(l(x))=\text{colim}\mathcal{E}_{\langle X_{\langle X,\mathbb{C}\text{-}\mathbb{P}\text{ast}\rangle},C\rangle}=f(x)$
and $S^2(!)(l(x))=S^2(!)(\text{colim}\mathcal{E}_{[X,\mathbb{C}\text{-}\mathbb{P}\text{ast},C]})$
$=\text{colim}\mathcal{E}_{[X,\mathbb{C}\text{-}\mathbb{P}\text{ast},1]}=g(x)$.
The cartesianity of the unit
$$\begin{tikzcd}
C\arrow[rr,"\lambda"]&&S(C)
\end{tikzcd}$$
is easier and goes as follow : we start with a commutative diagram in $\CS$
$$\begin{tikzcd}
C'\arrow[dd,"f"{left}]\arrow[rr,"!"]&&1\arrow[dd,"\lambda(1)"]\\\\
S(C)\arrow[rr,"S(!)"{below}]&&S(1)
\end{tikzcd}$$
Let $x$ an $n$-cell of $C'$, thus we have $f(x)=\text{colim}\mathcal{E}_{\langle X',C\rangle}$
where $X'$ is a rectangular $n$-divisor. Thus $S(!)(f(x))=\text{colim}\mathcal{E}_{\langle X',1\rangle}$, 
and then the commutativity gives $\text{colim}\mathcal{E}_{\langle X',1\rangle}=1$, and this shows that
$X'=1(n)dx^{i}_{k_i}$ is just the basic $n$-box without degeneracies. It shows that there 
is a unique map $l$ :
$$\begin{tikzcd}
C'\arrow[rrdd,"l",near end, dotted]\arrow[rrdddd,"f"{left}]\arrow[rrrrdd,"!"]\\\\
&&C\arrow[dd,"\lambda(C)"{left},near start]\arrow[rr,"!"]&&1\arrow[dd,"\lambda(1)"]\\\\
&&S(C)\arrow[rr,"S(!)"{below}]&&S(1)
\end{tikzcd}$$
defined by 
$l(x)=\text{colim}\mathcal{E}_{\langle X',C\rangle}$, and such that
$\lambda(C)(l(x))=\text{colim}\mathcal{E}_{\langle X',C\rangle}=f(x)$. 
\end{proof}

With this theorem we solved the conjecture in \cite{camark-cub-1} for the monad of cubical strict $\infty$-categories 
with connections which provides a complete description of the cubical operad $\mathbb{B}^{0}_C$ of cubical weak 
$\infty$-categories with connections.

\begin{proposition}
The monad $\mathbb{S}=(S,\lambda,\mu)$ acting on $\CS$ which algebras are cubical strict $\infty$-categories (without connections)
is cartesian.
\end{proposition}
\label{monad-S-prime}

\begin{proof}
This is easy, here we just use the previous proof by using only rectangular divisors build with classical degeneracies $1^{n}_{n+1,j}$
 ($n\in\mathbb{N}\text{ and }j\in\llbracket 1,n+1\rrbracket$) and their associated rectangular sketches.
\end{proof}

With this proposition we can easily use the materials in \cite{camark-cub-1} to build the cubical operad $\mathbb{B}^{0}_C$ of 
cubical weak $\infty$-categories without connections. In particular it is interesting to know that $\mathbb{B}^{0}_C$-algebras
of dimensions $2$ are exactly \textit{double categories of Verity} \cite{verity-thesis}. The proof of such fact is made just by 
mimic the proof of Michael Batanin in \cite{batanin-main} where he proved that with globular operads, $\mathbb{B}^{0}_C$-algebras
of dimensions $2$ are exactly bicategories. 

\section{The cubical coherator $\Theta^{\infty}_W$ of cubical weak $\infty$-categories with connections}
\label{coherator-cubique-M}

\begin{itemize}

\item A \textit{cubical theory} is given by a cubical extension (\ref{proof}) :
$$\begin{tikzcd}
\mathbb{C}\arrow[rr,"F"]&&\mathcal{C}
\end{tikzcd}$$
such that the induced unique functor $\bar{F}$ :
$$\begin{tikzcd}
\Theta_0\arrow[rr,"\bar{F}"]&&\mathcal{C}
\end{tikzcd}$$
 is bijective on objects, and thus a cubical theory\footnote{A cubical extension is written by using the Greek letter 
 $\Theta$ when it is a cubical theory.} $\Theta$ is a small category which objects are identify with rectangular divisors. 
 In particular a chosen initial object in 
 $\mathbb{C}\text{-}\mathbb{E}\text{xt}$ :
 
 $$\begin{tikzcd}
 \mathbb{C}\arrow[rr,"i"]&&\Theta_0
 \end{tikzcd}$$
 
 is a specific cubical theory called the initial cubical theory. The full subcategory of $\mathbb{C}\text{-}\mathbb{E}\text{xt}$
 which objects are cubical theories is denoted $\mathbb{C}\text{-}\mathbb{T}\text{h}$ and the cubical theory 
\begin{tikzcd}
 \mathbb{C}\arrow[rr,"i"]&&\Theta_0
 \end{tikzcd}
 is initial in it. It is interesting 
 to notice that morphisms $G$ in $\mathbb{C}\text{-}\mathbb{T}\text{h}$ : 
 
 $$\begin{tikzcd}
 &&\Theta\arrow[dd,"G"]\\
 \mathbb{C}\arrow[rru,"F"]
 \arrow[rrd,"F'"{below}]\\
 &&\Theta' 
 \end{tikzcd}$$ 
 
 induce, thanks to the universality of $\Theta_0$,
 the following commutative triangles in the category $\Cat$ of small categories :
 $$\begin{tikzcd}
 &&\Theta\arrow[dd,"G"]\\
 \Theta_0\arrow[rru,"\bar{F}"]
 \arrow[rrd,"\bar{F}'"{below}]\\
 &&\Theta' 
 \end{tikzcd}$$
 
 and more precisely this is a commutative triangle in the subcategory 
 $\mathbb{C}\text{-}\mathbb{S}\text{ketch}\subset\mathbb{S}\text{ketch}$ of the category of small sketches equipped with
 the cocones which underly the rectangular sketches $\mathcal{E}_X$ of 
 the section \ref{sketches} where $X$ is a rectangular $n$-divisor, because 
 all the functors $G$, $\bar{F}$ and $\bar{F}'$ do preserve these cocones.

 \item A set-model for the theory $\Theta$ or a $\Theta$-model for short is given by a functor :
 $$\begin{tikzcd}
 \Theta^{\text{op}}\arrow[rr,"G"]&&\E
 \end{tikzcd}$$
 
 which sends cubical sums to cubical products : more precisely by using the diagram :  
 
  $$\begin{tikzcd}
 \Theta_0^{op}\arrow[rr,"\bar{F}^{op}"]&&\Theta^{\text{op}}\arrow[rr,"G"]&&\E
 \end{tikzcd}$$ 
 
 we require the equality $(G\circ \bar{F}^{op})(\mathcal{E}_{X})=\text{lim}\mathcal{E}^{op}_{G(X)}$, 
 where $G(X)$ denotes the rectangular divisor decorated by occurrences of $G(I^p)$ i.e we substitute in each basic divisor 
 of $X$, $1(p)$ by $G(I^p)$, and the projective sketch $\mathcal{E}^{op}_{G(X)}$ is defined as the opposite sketch of 
 $\mathcal{E}_{G(X)}$.

\item A crucial example of cubical theory is the one $\Theta_{\mathbb{M}}$ of cubical reflexive $\infty$-magmas. We recall
their definition (\cite{cam-cubique}) :
consider a cubical reflexive set 
$$(C,(1^{n}_{n+1,j})_{n\in\mathbb{N},j\in\llbracket 1,n+1 \rrbracket}, (1^{n,\gamma}_{n+1,j})_{n\geq 1,j\in\llbracket 1,n \rrbracket})$$
equipped with partial operations 
$(\circ^{n}_{j})_{n\geq 1, j\in\llbracket 1,n \rrbracket}$ 
where if $a, b \in C(n)$ then $a\circ^{n}_{j}b$ is defined for $j\in\{1,...,n\}$ if 
$s^{n}_{j}(b)=t^{n}_{j}(a)$. We also require these operations to follow the following axioms of positions :

\begin{enumerate}[(i)]

\item For $1\leq j\leq n$ we have : $s^{n}_{n-1,j}(a\circ^{n}_{j}b)=s^{n}_{n-1,j}(a)$ and $t^{n}_{n-1,j}(a\circ^{n}_{j}b)=t^{n}_{n-1,j}(a)$,\\

\item $s^{n}_{n-1,i}(a\circ^{n}_{j}b)=\left\{
\begin{array}{rl}
  s^{n}_{n-1,i}(a)\circ^{n-1}_{j-1}s^{n}_{n-1,i}(b) \hspace{.1cm}\text{if} \hspace{.1cm}1\leq i<j\leq n \\
  s^{n}_{n-1,i}(a)\circ^{n-1}_{j}s^{n}_{n-1,i}(b)\hspace{.1cm}\text{if}\hspace{.1cm}1\leq j<i\leq n
 \end{array}
\right.$ 

\item $t^{n}_{n-1,i}(a\circ^{n}_{j}b)=\left\{
\begin{array}{rl}
  t^{n}_{n-1,i}(a)\circ^{n-1}_{j-1}t^{n}_{n-1,i}(b)\hspace{.1cm}\text{if}\hspace{.1cm}1\leq i<j\leq n \\
  t^{n}_{n-1,i}(a)\circ^{n-1}_{j}t^{n}_{n-1,i}(b)\hspace{.1cm}\text{if}\hspace{.1cm}1\leq j<i\leq n
 \end{array}
\right.$ 
\end{enumerate}

\end{itemize}

\begin{definition}
Cubical $\infty$-magmas are cubical sets equipped with partial operations like above. A morphism between two cubical $\infty$-magmas is a morphism of their underlying
cubical sets which respects partial operations 
$(\circ^{n}_{j})_{n\geq 1, j\in\llbracket 1,n \rrbracket}$.The category of cubical $\infty$-magmas is noted $\I\Cu\Mag$
\end{definition}

\begin{definition} 
Cubical reflexive $\infty$-magmas are cubical reflexive set equipped a structure of $\infty$-magmas. A morphism between 
two cubical reflexive $\infty$-magmas is a morphism of their underlying
cubical reflexive sets which respects partial operations $(\circ^{n}_{j})_{n\geq 1, j\in\llbracket 1,n \rrbracket}$. 
The category of cubical reflexive $\infty$-magmas is noted $\I\Cu\Mag_\text{r}$
\end{definition}

Now the forgetful functor
$$\begin{tikzcd}
\I\Cu\Mag_\text{r}\arrow[rr,"V"]&&\CS
\end{tikzcd}$$
is right adjoint and it induces the monad $\mathbb{M}=(M,\eta,\mu)$ of cubical reflexive
$\infty$-magmas with its Kleisli category $\mathbb{K}\text{l}(\mathbb{M})$. Denote by 
$\Theta_{\mathbb{M}}$ the full subcategory of $\mathbb{K}\text{l}(\mathbb{M})$ which
objects are objects of $\Theta_0$. This small category $\Theta_{\mathbb{M}}$ equipped
with the canonical inclusion functor
\begin{tikzcd}
\Theta_0\arrow[rr,"j"]&&\Theta_{\mathbb{M}}
\end{tikzcd}
is an important cubical theory because it
is the basic data we need to build the coherator $\Theta_W$ which models are cubical
weak $\infty$-categories with connections.

Consider an object 
\begin{tikzcd}
 \mathbb{C}\arrow[rr,"i"]&&\Theta
 \end{tikzcd}
 of the category $\mathbb{C}\text{-}\mathbb{T}\text{h}$ of cubical theories and the unique 
 functor 
 \begin{tikzcd}
 \Theta_0\arrow[rr,"\bar{i}"]&&\Theta
 \end{tikzcd}. 
 A $I^{n}$-arrow in $\Theta$ is one arrow
 of it with domain the object $I^{n}$ (which is by definition 
 equal to $i(1(n))$). A pair $(f,g)$ of $I^{n}$-arrows in $\Theta$ :
$$\begin{tikzcd}
 I^{n}\arrow[rr, yshift=1.1ex,"f"]\arrow[rr, yshift=-1.1ex,"g"{below}]   
        &&X 
 \end{tikzcd}$$
 
 is called 
 
 \begin{itemize}
  \item admissible if it doesn't belong to the image of $\bar{i}$
  
  \item $j$-admissible (for a direction $j\in\llbracket 1,n\rrbracket$) if it is admissible and it is 
  $j$-parallel, i.e $f\circ s^{n}_{n-1,j}=g\circ s^{n}_{n-1,j}$ and 
  $f\circ t^{n}_{n-1,j}=g\circ t^{n}_{n-1,j}$

$$\begin{tikzcd}
 I^{n}\arrow[rrr, yshift=1.1ex,"f"]\arrow[rrr, yshift=-1.1ex,"g"{below}]&&&X\\\\
 I^{n-1}\arrow[uu,xshift=1.1ex,"t^{n}_{n-1,j}"{right}]\arrow[uu,xshift=-1.1ex,"s^{n}_{n-1,j}"]  
 \end{tikzcd}$$
 
 \end{itemize}

If a pair $(f,g)$ of $I^{n}$-arrows :
 \begin{tikzcd}
 I^{n}\arrow[rr, yshift=1.1ex,"f"]\arrow[rr, yshift=-1.1ex,"g"{below}]   
        &&X 
 \end{tikzcd}
 is admissible, then we define its \textit{liftings} which are all $j$-\textit{lifting arrows} $[f,g]^{n}_{n+1,j}$ for all $j\in\llbracket 1,n+1\rrbracket$ :

 $$\begin{tikzcd}
 I^{n+1}\arrow[rrrdd,dotted,"{[f,g]^{n}_{n+1,j}}"]\\\\
 I^{n}\arrow[uu,xshift=1.1ex,"t^{n+1}_{n,i}"{right}]\arrow[uu,xshift=-1.1ex,"s^{n+1}_{n,i}"]
 \arrow[rrr, yshift=1.1ex,"f"]\arrow[rrr, yshift=-1.1ex,"g"{below}]&&&X\\\\
 I^{n-1}\arrow[uu,xshift=1.1ex,"t^{n}_{n-1,i}"{right}]\arrow[uu,xshift=-1.1ex,"s^{n}_{n-1,i}"]  
 \end{tikzcd}$$
 
 by using an induction. Thus we suppose that such operations 
 $[-,-]^{p}_{p+1,k}$ ($p\leq n-1 \text{ and }k\in\llbracket 1,p+1\rrbracket$) exists for all faces of $f$ and $g$. The
 definition of $[f,g]^{n}_{n+1,j}$ goes as follow :
 
 \begin{itemize}
 \item if $1\leq i<j\leq n+1$ then $[f,g]^{n}_{n+1,j}\circ s^{n+1}_{n,i}=[f\circ s^{n}_{n-1,i},g\circ s^{n}_{n-1,i}]^{n-1}_{n,j-1}$
 and $[f,g]^{n}_{n+1,j}\circ t^{n+1}_{n,i}=[f\circ t^{n}_{n-1,i},g\circ t^{n}_{n-1,i}]^{n-1}_{n,j-1}$ 
 
 \item if $1\leq j\leq n+1$ then $[f,g]^{n}_{n+1,j}\circ s^{n+1}_{n,j}=f$ and $[f,g]^{n}_{n+1,j}\circ t^{n+1}_{n,j}=g$

 \item if $1\leq j<i\leq n+1$ then $[f,g]^{n}_{n+1,j}\circ s^{n+1}_{n,i}=[f\circ s^{n}_{n-1,i-1},g\circ s^{n}_{n-1,i-1}]^{n-1}_{n,j}$ 
 and $[f,g]^{n}_{n+1,j}\circ t^{n+1}_{n,i}=[f\circ t^{n}_{n-1,i-1},g\circ t^{n}_{n-1,i-1}]^{n-1}_{n,j}$

 \end{itemize}
 
%
%
%
%
%
%
%
%
%
%
 
If a pair $(f,g)$ of $I^{n}$-arrows :
 \begin{tikzcd}
 I^{n}\arrow[rr, yshift=1.1ex,"f"]\arrow[rr, yshift=-1.1ex,"g"{below}]   
        &&X 
 \end{tikzcd}
 is $j$-admissible (for a fixed direction $j\in\llbracket 1,n\rrbracket$), then we define its $(j,-)$-\textit{lifting arrow} $[f,g]^{n,-}_{n+1,j}$
 or its $(j,-)$-lifting for short, the following $I^{n+1}$-arrow : 
 
 $$\begin{tikzcd}
 I^{n+1}\arrow[rrrdd,dotted,"{[f,g]^{n,-}_{n+1,j}}"]\\\\
 I^{n}\arrow[uu,xshift=1.1ex,"t^{n+1}_{n,i}"{right}]\arrow[uu,xshift=-1.1ex,"s^{n+1}_{n,i}"]
 \arrow[rrr, yshift=1.1ex,"f"]\arrow[rrr, yshift=-1.1ex,"g"{below}]&&&X\\\\
 I^{n-1}\arrow[uu,xshift=1.1ex,"t^{n}_{n-1,i}"{right}]\arrow[uu,xshift=-1.1ex,"s^{n}_{n-1,i}"] 
 \end{tikzcd}$$
 
 by using an induction. Thus we suppose that such operations 
 $[-,-]^{p,-}_{p+1,k}$ ($p\leq n-1 \text{ and }k\in\llbracket 1,p\rrbracket$) exists for all faces of $f$ and $g$, but also 
 (see the induction used just below) we have to suppose that the operations $[-,-]^{p}_{p+1,k}$ ($p\leq n-1 \text{ and }k\in\llbracket 1,p+1\rrbracket$) defined above exists for such faces.
 The definition of $[f,g]^{n,-}_{n+1,j}$ goes as follow :
 
 \begin{itemize}
 
 \item if $1\leq i<j\leq n$ then 
 $[f;g]^{n,-}_{n+1,j}\circ s^{n+1}_{n,i}=[f\circ s^{n}_{n-1,i};g\circ s^{n}_{n-1,i}]^{n-1,-}_{n,j-1}$
 and 
  $[f;g]^{n,-}_{n+1,j}\circ t^{n+1}_{n,i}=[f\circ t^{n}_{n-1,i};g\circ t^{n}_{n-1,i}]^{n-1,-}_{n,j-1}$

 \item if $1\leq j\leq n$ then $[f;g]^{n,-}_{n+1,j}\circ s^{n+1}_{n,j}=f$ and $[f;g]^{n,-}_{n+1,j}\circ s^{n+1}_{n,j+1}=g$,
 and $[f;g]^{n,-}_{n+1,j}\circ t^{n+1}_{n,j}=[f;g]^{n,-}_{n+1,j}\circ t^{n+1}_{n,j+1}=[f\circ t^{n}_{n-1,j},g\circ t^{n}_{n-1,j}]^{n-1}_{n,j}$
 
 \item if $2\leq j+1<i\leq n+1$ then $[f;g]^{n,-}_{n+1,j}\circ s^{n+1}_{n,i}=[f\circ s^{n}_{n-1,i-1};g\circ s^{n}_{n-1,i-1}]^{n-1,-}_{n,j}$ and 
 $[f;g]^{n,-}_{n+1,j}\circ s^{n+1}_{n,i}=[f\circ s^{n}_{n-1,i-1};g\circ s^{n}_{n-1,i-1}]^{n-1,-}_{n,j}$

 \end{itemize} 
 
%
%
%
%
%
%

If a pair $(f,g)$ of $I^{n}$-arrows :
 \begin{tikzcd}
 I^{n}\arrow[rr, yshift=1.1ex,"f"]\arrow[rr, yshift=-1.1ex,"g"{below}]   
        &&X 
 \end{tikzcd}
 is $j$-admissible (for a fixed direction $j\in\llbracket 1,n\rrbracket$), then we define its $(j,+)$-\textit{lifting arrow} $[f,g]^{n,+}_{n+1,j}$  
 or its $(j,+)$-lifting for short, the following $I^{n+1}$-arrow : 
 
 $$\begin{tikzcd}
 I^{n+1}\arrow[rrrdd,dotted,"{[f,g]^{n,+}_{n+1,j}}"]\\\\
 I^{n}\arrow[uu,xshift=1.1ex,"t^{n+1}_{n,i}"{right}]\arrow[uu,xshift=-1.1ex,"s^{n+1}_{n,i}"]
 \arrow[rrr, yshift=1.1ex,"f"]\arrow[rrr, yshift=-1.1ex,"g"{below}]&&&X\\\\
 I^{n-1}\arrow[uu,xshift=1.1ex,"t^{n}_{n-1,i}"{right}]\arrow[uu,xshift=-1.1ex,"s^{n}_{n-1,i}"] 
 \end{tikzcd}$$
 
 by using an induction. Thus we suppose that such operations 
 $[-,-]^{p,+}_{p+1,k}$ ($p\leq n-1 \text{ and }k\in\llbracket 1,p\rrbracket$) exists for all faces of $f$ and $g$, but also 
 (see the induction used just below) we have to suppose that the operations $[-,-]^{p}_{p+1,k}$ ($p\leq n-1 \text{ and }k\in\llbracket 1,p+1\rrbracket$) defined above exists for such faces.
 The definition of $[f,g]^{n,+}_{n+1,j}$ goes as follow :
 
 \begin{itemize}
 
 \item if $1\leq i<j\leq n$ then 
 $[f;g]^{n,+}_{n+1,j}\circ s^{n+1}_{n,i}=[f\circ s^{n}_{n-1,i};g\circ s^{n}_{n-1,i}]^{n-1,+}_{n,j-1}$
 and 
  $[f;g]^{n,+}_{n+1,j}\circ t^{n+1}_{n,i}=[f\circ t^{n}_{n-1,i};g\circ t^{n}_{n-1,i}]^{n-1,+}_{n,j-1}$

 \item if $1\leq j\leq n$ $[f;g]^{n,-}_{n+1,j}\circ s^{n+1}_{n,j}=[f;g]^{n,-}_{n+1,j}\circ s^{n+1}_{n,j+1}=[f\circ s^{n}_{n-1,j},g\circ s^{n}_{n-1,j}]^{n-1}_{n,j}$
 and $[f;g]^{n,+}_{n+1,j}\circ t^{n+1}_{n,j}=f$ and $[f;g]^{n,+}_{n+1,j}\circ t^{n+1}_{n,j+1}=g$,

 \item if $2\leq j+1<i\leq n+1$ then $[f;g]^{n,+}_{n+1,j}\circ s^{n+1}_{n,i}=[f\circ s^{n}_{n-1,i-1};g\circ s^{n}_{n-1,i-1}]^{n-1,+}_{n,j}$ and 
 $[f;g]^{n,+}_{n+1,j}\circ t^{n+1}_{n,i}=[f\circ t^{n}_{n-1,i-1};g\circ t^{n}_{n-1,i-1}]^{n-1,+}_{n,j}$ 
 
 \end{itemize}

%
%
%
%
%
%
%
%

\begin{definition}
A cubical theory $\Theta$ is contractible if for all integer $n\geq 1$, for all pairs $(f,g)$ of $I^{n}$-arrows in it which are admissible,
have liftings, and for all pairs $(f,g)$ of $I^{n}$-arrows in it which are $j$-admissible ($j\in\llbracket 1,n\rrbracket$), 
have a $(j,-)$-lifting and have a $(j,+)$-lifting
\end{definition}

Now we are going to build a cubical contractible theory $\Theta^{\infty}_W$ which set-models are cubical
weak $\infty$-categories with connections. This theory $\Theta^{\infty}_W$ is a coherator in the sense
of Grothendieck (\cite{Maltsin-Gr}), i.e it is obtained as a colimit of a diagram 
$\mathcal{D}_{\Theta_W}$ in $\Cat$ 
of cubical theories :

$$
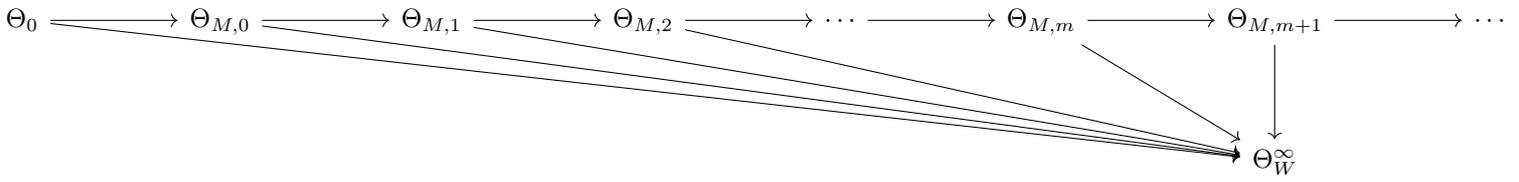
\begin{tikzcd}
\Theta_0\arrow[rrrrrrrrrrrrdd]\arrow[rr]&&\Theta_{M,0}\arrow[rrrrrrrrrrdd]\arrow[rr]&&\Theta_{M,1}
\arrow[rrrrrrrrdd]\arrow[rr]&&\Theta_{M,2}\arrow[rrrrrrdd]\arrow[rr,]&&\cdots
\arrow[rr]&&\Theta_{M,m}\arrow[rrdd]\arrow[rr]&&\Theta_{M,m+1}\arrow[dd]\arrow[rr]&&\cdots\\\\
&&&&&&&&&&&&\Theta^{\infty}_W
\end{tikzcd}$$

and this diagram $\mathcal{D}_{\Theta_W}$ is a sequence in the category $\mathbb{C}\text{-}\mathbb{T}\text{h}$ 
of cubical theories :

$$
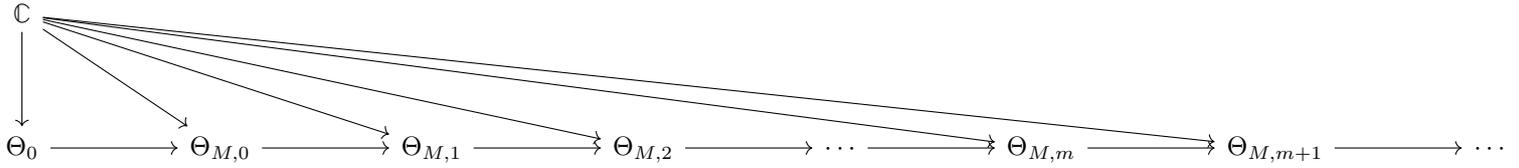
\begin{tikzcd}
\mathbb{C}\arrow[dd]\arrow[rrdd]\arrow[rrrrdd]\arrow[rrrrrrdd]\arrow[rrrrrrrrrrdd]\arrow[rrrrrrrrrrrrdd]\\\\
\Theta_0\arrow[rr]&&\Theta_{M,0}\arrow[rr]&&\Theta_{M,1}
\arrow[rr]&&\Theta_{M,2}\arrow[rr,]&&\cdots
\arrow[rr]&&\Theta_{M,m}\arrow[rr]&&\Theta_{M,m+1}\arrow[rr]&&\cdots
\end{tikzcd}$$

that we may define inductively :

\begin{itemize}

\item We start the induction with $\Theta_{M,0}=\Theta_{\mathbb{M}}$ i.e with the
cubical theory of cubical reflexive $\infty$-magmas.

\item We denote by $E_{M,0}$ the set which is the union of all admissible pairs of $I^n$-arrows in $\Theta_{\mathbb{M}}$
(for all $n\geq 1$), all $j$-admissible pairs of $I^n$-arrows in $\Theta_{\mathbb{M}}$ 
(for all directions $j\in\llbracket 1,n\rrbracket$ for all $n\geq 1$);

\item $\Theta_{M,1}$ is obtained by formally (see just below a precise meaning of "formally") adding in $\Theta_{M,0}$ all kind of liftings 
of elements of $E_{M,0}$.

\item Denote by $E_{M,1}$ the set which is the union of  : all admissible pairs of $I^n$-arrows in $\Theta_{M,1}$ which are
not in $E_{M,0}$, and all $j$-admissible pairs of $I^n$-arrows in $\Theta_{M,1}$ which are
not in $E_{M,0}$. 

\item $\Theta_{M,2}$ is obtained by formally adding in $\Theta_{M,1}$ all kind of liftings of elements of $E_{M,1}$.

\item we suppose that until the integer $m-1$ the sequence :

$$\begin{tikzcd}
(\Theta_{M,0},E_{M,0})\arrow[rr]&&(\Theta_{M,1},E_{M,1})
\arrow[rr]&&\cdots
\arrow[rr]&&(\Theta_{M,m-2},E_{M,m-2})\arrow[rr]&&(\Theta_{M,m-1},E_{M,m-1})
\end{tikzcd}$$

is well defined. Thus $\Theta_{M,m}$ is obtained by formally adding in $\Theta_{M,m-1}$ all
kind of liftings of elements of $E_{M,m-1}$.

\item we associate to $\Theta_{M,m}$ the set $E_{M,m}$ which is the union of : 
all admissible pairs of $I^n$-arrows in $\Theta_{M,m}$ which are
not in $E_{M,m-1}$, all $j$-admissible pairs of $I^n$-arrows in $\Theta_{M,m}$ which are
not in $E_{M,m-1}$.
\end{itemize}

An important fact is the cubical theory $\Theta_{M,m}$ obtained by formally adding in $\Theta_{M,m-1}$ all
liftings of elements of $E_{M,m-1}$ is universal for this adding. To give a precise meaning of "formally adding" 
is just an application of the following theorem of Christian Lair\footnote{This result was found by Christian Lair, but we
were not able to find an exact reference of it.} :
\begin{theorem}[Lair]
The category $\mathbb{S}\text{ketch}$ of Sketches is projectively sketchable, that is there 
a projective sketch $\mathcal{E}_{\mathbb{S}\text{ketch}}$ such that the category 
$\mathbb{M}\text{od}(\mathcal{E}_{\mathbb{S}\text{ketch}})$ of set-models 
 of $\mathcal{E}_{\mathbb{S}\text{ketch}}$ is equivalent to the category $\mathbb{S}\text{ketch}$.
\end{theorem}
Also the category $\Cat$ of small categories is also projectively sketchable by a projective 
sketch $\mathcal{E}_{\Cat}$ and we have an easy morphism of projective sketches :
\begin{tikzcd}
\mathcal{E}_{\Cat}\arrow[r,"i"]&\mathcal{E}_{\mathbb{S}\text{ketch}}
\end{tikzcd}
which induces a left adjunction $F$ with the functor $\mathbb{M}\text{od}(i)$ :
\begin{tikzcd}
\mathbb{S}\text{ketch}\arrow[r,"F"]&\Cat
\end{tikzcd}
This construction is called the \textit{free prototype functor}\footnote{Private communication with Christian Lair.}. With these results in hands
it is useful to see the cubical theory $\Theta_{M,m}$ obtained by formally adding in $\Theta_{M,m-1}$ all
liftings of elements of $E_{M,m-1}$ as the free category (with the free prototype functor) generated by this adding. 
Thus we start with the object $\Theta_{M,m-1}+E_{M,m-1}$ of $\mathbb{S}\text{ketch}$, where we 
formally\footnote{Here "formally" has an accurate logical sense.}
add all liftings of elements of $E_{M,m-1}$ in the sketch $\Theta_{M,m-1}$, then $\Theta_{M,m}$ is just the free category 
$F(\Theta_{M,m-1}+E_{M,m-1})$ generated by the free prototype functor.

The colimits $\Theta^{\infty}_W$ in $\Cat$ :

$$\begin{tikzcd}
\Theta_0\arrow[rrrrrrrrrrrrdd]\arrow[rr]&&\Theta_{M,0}\arrow[rrrrrrrrrrdd]\arrow[rr]&&\Theta_{M,1}
\arrow[rrrrrrrrdd]\arrow[rr]&&\Theta_{M,2}\arrow[rrrrrrdd]\arrow[rr,]&&\cdots
\arrow[rr]&&\Theta_{M,m}\arrow[rrdd]\arrow[rr]&&\Theta_{M,m+1}\arrow[dd]\arrow[rr]&&\cdots\\\\
&&&&&&&&&&&&\Theta^{\infty}_W
\end{tikzcd}$$

is called the \textit{cubical coherator} of cubical weak $\infty$-categories with connections. Denote
by $\mathbb{M}\text{od}(\Theta^{\infty}_W)$ the category of $\Theta^{\infty}_W$-models in $\E$. The category
$\mathbb{M}\text{od}(\Theta^{\infty}_W)$ is a category of models of cubical weak $\infty$-categories with connections.
\begin{remark}
Of course we suspect this category to be equivalent to the category of algebras for the cubical operad built in \cite{camark-cub-1}, 
especially because this is true for the globular geometry \cite{bourke-injectif}. But we prefer to avoid such question here.
\end{remark}

\section{The cubical coherator $\Theta^{\infty}_{W^{0}}$ of cubical weak $\infty$-groupoids with connections}
\label{coherator-cubique-M0}

 Cubical $(\infty,0)$-sets underly a new sketch
(see diagrams below) 
which we use to define a coherator which models are cubical weak $\infty$-groupoids.
Here we define cubical version of
the formalism developed in \cite{Cam} for globular $(\infty,0)$-sets. This formalism of this cubical world is very similar to its globular analogue.
\label{reversible-omega-graphs}
Consider a cubical set $\mathcal{C}=(C_n,s^{n}_{n-1,j},t^{n}_{n-1,j})_{1\leq j\leq n}$. If $n\geq 1$ and 
$1\leq j\leq n$, then a $(n,j)$-reversor on it 
is given by a map $\xymatrix{C_{n}\ar[rr]^{j^{n}_{j}}&&C_{n}}$ such that the following two diagrams commute :
\[\xymatrix{C_{n}\ar[rr]^{j^{n}_{j}}\ar[rd]_{s^{n}_{n-1,j}}&&C_n\ar[ld]^{t^{n}_{n-1,j}}\\
&C_{n-1}}\qquad\xymatrix{C_{n}\ar[rr]^{j^{n}_{j}}\ar[rd]_{t^{n}_{n-1,j}}&&C_m\ar[ld]^{s^{n}_{n-1,j}}\\
&C_{n-1}}\]

If for each $n>0$ and for each $1\leq j\leq n$, there are such $(n,j)$-reversor $j^{n}_{j}$ on $\mathcal{C}$, then 
we say that $\mathcal{C}$ is a cubical $(\infty,0)$-set. The family of maps $(j^{n}_{j})_{n>0,1\leq j\leq n}$ for all
 $(n\in\mathbb{N}^*)$   is
called an $(\infty,0)$-structure and in that case we shall say that $\mathcal{C}$ is equipped with the $(\infty,0)$-structure 
$(j^{n}_{j})_{n>0,1\leq j\leq n}$. When we speak about such $(\infty,0)$-structure $(j^{n}_{j})_{n>0,1\leq j\leq n}$ on $\mathcal{C}$, 
it means that it is for all integers $n\in\mathbb{N}^*$ such that $C_n$ is non-empty. Seen as cubical $(\infty,0)$-set we denote it by
$\mathcal{C}=((C_n,s^{n}_{n-1,j},t^{n}_{n-1,j})_{1\leq j\leq n},(j^{n}_{j})_{n>0,1\leq j\leq n})$. If 
$\mathcal{C}'=((C_n',s'^{n}_{n-1,j},t'^{n}_{n-1,j})_{1\leq j\leq n},(j'^{n}_{j})_{n>0,1\leq j\leq n})$ is another
$(\infty,0)$-set, then a morphism of $(\infty,0)$-sets 
\[\xymatrix{\mathcal{C}\ar[rr]^{f}&&\mathcal{C}'}\]
is given by a morphism of cubical sets such that for each $n>0$ and for each $1\leq j\leq n$ we have the following commutative diagrams
\[\xymatrix{C_{n}\ar[d]_{f_{n}}\ar[rrr]^{j^{n}_{j}}&&&C_{n}\ar[d]^{f_{n}}\\
C'_{n}\ar[rrr]_{j'^{n}_{j}}&&&C'_{n}}\]
The category of cubical $(\infty,0)$-sets is denoted $(\infty,0)$-$\mathbb{C}\E$.
A cubical reflexive $(\infty,0)$-magma is an object of $\I\mathbb{C}\mathbb{M}\text{ag}_\text{r}$ such that its underlying
cubical set is equipped with an $(\infty,0)$-structure. Morphisms between cubical reflexive $(\infty,0)$-magmas are
those of $\I\mathbb{C}\mathbb{M}\text{ag}_\text{r}$ which are also morphisms of $(\infty,0)\text{-}\mathbb{C}\mathbb{S}ets$, i.e
they preserve the underlying $(\infty,0)$-structures. The category of cubical reflexive $(\infty,0)$-magmas is denoted 
$(\infty,0)\text{-}\mathbb{C}\mathbb{M}\text{ag}_\text{r}$.

Now the forgetful functor
$$\begin{tikzcd}
(\infty,0)\text{-}\mathbb{C}\mathbb{M}\text{ag}_\text{r}\arrow[rr,"V"]&&\CS
\end{tikzcd}$$
is right adjoint and it induces the monad $\mathbb{M}^{0}=(M^{0},\eta^{0},\mu^{0})$ of cubical reflexive
$(\infty,0)$-magmas with its Kleisli category $\mathbb{K}\text{l}(\mathbb{M}^{0})$. Denote by 
$\Theta_{\mathbb{M}^{0}}$ the full subcategory of $\mathbb{K}\text{l}(\mathbb{M}^{0})$ which
objects are objects of $\Theta_0$. This small category $\Theta_{\mathbb{M}^{0}}$ equipped
with the canonical inclusion functor
\begin{tikzcd}
\Theta_0\arrow[rr,"j^{0}"]&&\Theta_{\mathbb{M}^{0}}
\end{tikzcd}
is an important cubical theory because it
is the basic data we need to build the coherator $\Theta^{\infty}_{W^{0}}$ which models are cubical
weak $\infty$-groupoids with connections.

This theory $\Theta^{\infty}_{W^{0}}$ is a coherator in the sense
of Grothendieck (\cite{Maltsin-Gr}), i.e it is obtained as a colimit of a diagram 
$\mathcal{D}_{\Theta_{W^{0}}}$ in $\Cat$ 
of cubical theories :

$$
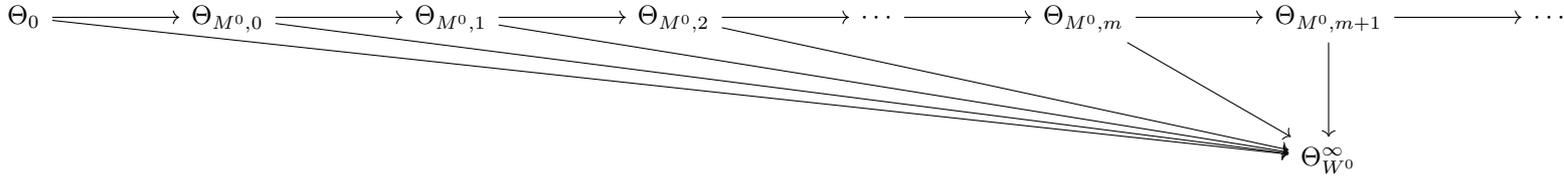
\begin{tikzcd}
\Theta_0\arrow[rrrrrrrrrrrrdd]\arrow[rr]&&\Theta_{M^{0},0}\arrow[rrrrrrrrrrdd]\arrow[rr]&&\Theta_{M^{0},1}
\arrow[rrrrrrrrdd]\arrow[rr]&&\Theta_{M^{0},2}\arrow[rrrrrrdd]\arrow[rr,]&&\cdots
\arrow[rr]&&\Theta_{M^{0},m}\arrow[rrdd]\arrow[rr]&&\Theta_{M^{0},m+1}\arrow[dd]\arrow[rr]&&\cdots\\\\
&&&&&&&&&&&&\Theta^{\infty}_{W^{0}}
\end{tikzcd}$$

and this diagram $\mathcal{D}_{\Theta^{\infty}_{W^{0}}}$ is a sequence in the category $\mathbb{C}\text{-}\mathbb{T}\text{h}$ 
of cubical theories :

$$
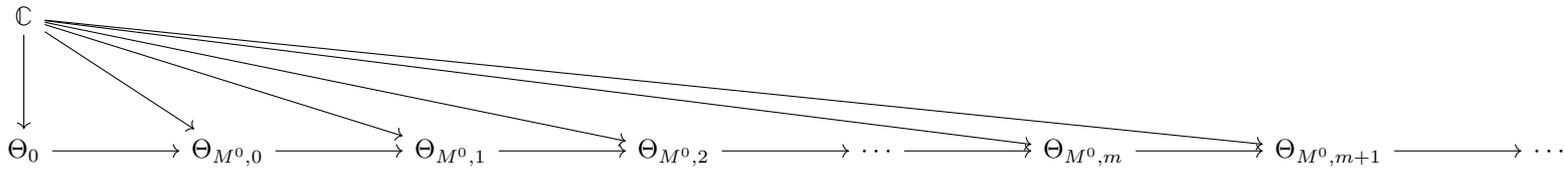
\begin{tikzcd}
\mathbb{C}\arrow[dd]\arrow[rrdd]\arrow[rrrrdd]\arrow[rrrrrrdd]\arrow[rrrrrrrrrrdd]\arrow[rrrrrrrrrrrrdd]\\\\
\Theta_0\arrow[rr]&&\Theta_{M^{0},0}\arrow[rr]&&\Theta_{M^{0},1}
\arrow[rr]&&\Theta_{M^{0},2}\arrow[rr,]&&\cdots
\arrow[rr]&&\Theta_{M^{0},m}\arrow[rr]&&\Theta_{M^{0},m+1}\arrow[rr]&&\cdots
\end{tikzcd}$$

that we may define inductively :

\begin{itemize}

\item We start the induction with $\Theta_{M^{0},0}=\Theta_{\mathbb{M}^{0}}$ i.e with the
cubical theory of cubical reflexive $(\infty,0)$-magmas.

\item We denote by $E_{M^{0},0}$ the set which is the union of all admissible pairs of $I^n$-arrows in $\Theta_{\mathbb{M}^{0}}$
(for all $n\geq 1$), all $j$-admissible pairs of $I^n$-arrows in $\Theta_{\mathbb{M}^{0}}$ 
(for all directions $j\in\llbracket 1,n\rrbracket$ for all $n\geq 1$);

\item $\Theta_{M^{0},1}$ is obtained by formally adding in $\Theta_{M^{0},0}$ all kind of liftings of elements of $E_{M^{0},0}$.

\item Denote by $E_{M^{0},1}$ the set which is the union of  : all admissible pairs of $I^n$-arrows in $\Theta_{M^{0},1}$ which are
not in $E_{M^{0},0}$, all $j$-admissible pairs of $I^n$-arrows in $\Theta_{M^{0},1}$ which are
not in $E_{M^{0},0}$. 

\item $\Theta_{M^{0},2}$ is obtained by formally adding in $\Theta_{M^{0},1}$ all kind of liftings of elements of $E_{M^{0},1}$.

\item we suppose that until the integer $m-1$ the sequence :

$$\begin{tikzcd}
(\Theta_{M^{0},0},E_{M^{0},0})\arrow[rr]&&(\Theta_{M^{0},1},E_{M^{0},1})
\arrow[rr]&&\cdots
\arrow[rr]&&(\Theta_{M^{0},m-2},E_{M^{0},m-2})\arrow[rr]&&(\Theta_{M^{0},m-1},E_{M^{0},m-1})
\end{tikzcd}$$

is well defined. Thus $\Theta_{M^{0},m}$ is obtained by formally adding in $\Theta_{M^{0},m-1}$ all
kind of liftings of elements of $E_{M^{0},m-1}$.

\item we associate to $\Theta_{M^{0},m}$ the set $E_{M^{0},m}$ which is the union of : 
all admissible pairs of $I^n$-arrows in $\Theta_{M^{0},m}$ which are
not in $E_{M^{0},m-1}$, all $j$-admissible pairs of $I^n$-arrows in $\Theta_{M^{0},m}$ which are
not in $E_{M^{0},m-1}$. 
\end{itemize}

The colimits $\Theta^{\infty}_{W^{0}}$ in $\Cat$ :

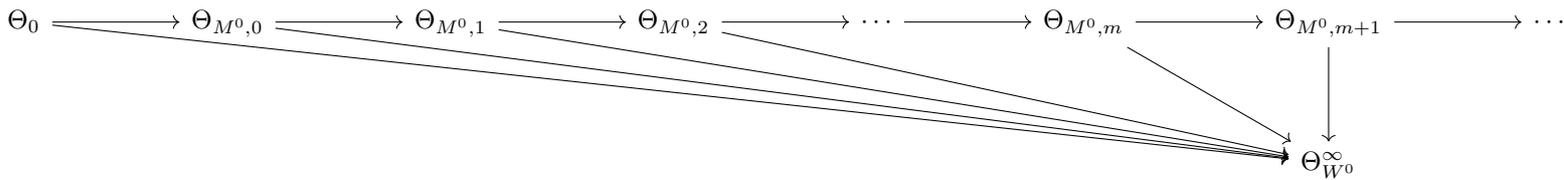
$$\begin{tikzcd}
\Theta_0\arrow[rrrrrrrrrrrrdd]\arrow[rr]&&\Theta_{M^{0},0}\arrow[rrrrrrrrrrdd]\arrow[rr]&&\Theta_{M^{0},1}
\arrow[rrrrrrrrdd]\arrow[rr]&&\Theta_{M^{0},2}\arrow[rrrrrrdd]\arrow[rr,]&&\cdots
\arrow[rr]&&\Theta_{M^{0},m}\arrow[rrdd]\arrow[rr]&&\Theta_{M^{0},m+1}\arrow[dd]\arrow[rr]&&\cdots\\\\
&&&&&&&&&&&&\Theta^{\infty}_{W^{0}}
\end{tikzcd}$$

is called the \textit{cubical coherator} of cubical weak $\infty$-groupoids with connections. Denote
by $\mathbb{M}\text{od}(\Theta^{\infty}_{W^{0}})$ the category of $\Theta^{\infty}_{W^{0}}$-models in $\E$. The category
$\mathbb{M}\text{od}(\Theta^{\infty}_{W^{0}})$ is a category of models of cubical weak $\infty$-groupoids with 
connections.

\bigbreak{}
\begin{minipage}{1.0\linewidth}
Laboratoire de Math\'ematiques d'Orsay, UMR 8628\\
Universit\'e de Paris-Saclay and CNRS\\
B\^atiment 307, Facult\'e des Sciences d'Orsay\\
94015 ORSAY Cedex, FRANCE\\
\href{mailto:camell.kachour@universite-paris-saclay.fr}{\url{camell.kachour@universite-paris-saclay.fr}}
\end{minipage}
\end{document}